\documentclass[11pt,a4paper]{article}
\usepackage{latexsym,amssymb,amsmath,amsthm,graphicx}
\usepackage[pagebackref=false]{hyperref}
\usepackage{hyperref}
\usepackage{tikz-cd}

\newcommand{\RNum}[1]{\uppercase\expandafter{\romannumeral #1\relax}}
\newtheorem{theorem}{Theorem}[section]
\newtheorem{definition}[theorem]{Definition}
\newtheorem{lemma}[theorem]{Lemma}

\newtheorem{exercise}[theorem]{Exercise}
\newtheorem{proposition}[theorem]{Proposition}
\newtheorem{corollary}[theorem]{Corollary}
\newtheorem{remark}[theorem]{Remark}
\newtheorem{axiom}{Axiom}

\begin{document}
\title{Axiomatic Music Theory}
\author{Seyyed Mehdi Nemati}
\date{January 1, 2020}
\maketitle
\abstract{It is tried to axiomatize the transparent theory of music.}

\paragraph*{}

\paragraph*{}
\section{Introduction}

\paragraph*{}
One of the most important events in the history of mathematics is distinguished by going back over two thousand years to the ancient Greeks, beginning with the initial efforts at using deductive reasoning in geometrical demonstrations by Thales and culminating with Euclid around 300 B.C. as observed in his \textit{Elements}, which in fact gave birth to the \emph{axiomatic method} in mathematics, asserted that statements of geometry could (and must) be established by logical deduction rather than experimentation and empirical procedure leading to approximate achievements which were most often sufficient for practical intentions. During the nineteenth century, the discovery of non-Euclidean geometry caused geometricians to construct new axiom systems for Euclidean geometry and reveal flaws existed in Euclid's presentation of geometry by reexamining its foundations. In the late nineteenth century, Pasch made the first rigorous attempt to fill in gaps that had spoiled Euclid's work. Pasch's axiomatization of geometry made geometric notions and arguments highly explicit and remarkably modified. He emphasized the importance of purification of axiomatic treatment of the world of geometry on a logical foundation to make it absolutely accurate. At the end of the century, the most intuitive system of axioms had undoubtedly been proposed by Hilbert in which the clear simplicity of the axioms received a widespread acceptance of his contemporary mathematicians and produced a formalized axiomatic approach to develop neutral geometry. On the other hand, some axiomatizations of arithmetic was also done by other mathematicians of the time like Peano and so forth. The aim was to found a logical basis for arithmetic by virtue of the axiomatic development to make it and a fortiori number theory (concerning natural numbers to complex ones) completely rigorous. Of course, the idea of arithmetizing mathematical analysis in the real sense relied on such a deeper perspective on the foundations of mathematics some other mathematicians put into in the direction of resolving the crises in those foundations. For instance, in the 1900s, Russell's paradox in Cantor's set theory was removed through the first successful axiomatic system of set theory published by Zermelo and expanded later by Fraenkel.

Although more than 2000 years have passed, an axiomatization of the third subject of the Pythagorean quadrivium consisting of \textit{arithmetica}, \textit{geometria}, \textit{harmonia}, and \textit{astrologia} (\cite{B3}, p. 90), i.e. music, has not yet appeared. This is not at all because of a deficiency of attention to the mathematical aspects of music which is extensively noticed, from the ancient time by Pythagoras who calculated the correct ratios of harmonic intervals, to modern mathematicians working on the field of diatonic theory --- the study of music fundamentals from the mathematical point of view \cite{J1}. In this direction, many scientists have certainly provided effective help in evolving the mathematical comprehension of music rules by means of different areas of research such as basic music theory \cite{M4} (introducing the international language of music), acoustics \cite{E2} (describing physics of sound), psychoacoustics \cite{H3} (study of the vibrations of musical tones), aesthetics \cite{S2} (philosophizing on the nature of harmonic intervals), psychology \cite{D1} (researching on mental influences of pleasant melodies), anatomy \cite{A1} (studying the structure of the human ear), et cetera.

In summary, we plan to combine the mathematical perceptions within the first two sciences of quadrivium as mentioned above, take the resultant axiomatic approach to the third one, and create a vastly precise theorization of the essential structure of music theory on the base of the language of set theory (since no elements with the exception of \textit{set} really need ever be considered for mathematical purposes) in such a manner that by accepting the most intuitive facts of the music universe, we allow only pure logic to lead us to desired conclusions. Thus, the principal goal of the paper, which exhibits a mathematically axiomatized theorizing on music, is to study the foundations of this knowledge, by returning to the beginnings of that, exposing exactly what is assumed there, and building the entire subject on those foundations (as expressed in \cite{V1}, p. ix). In that regard, we will achieve explicit answers to some challenging wh-questions such as follows. What are the exact definitions of note and chord? Why do sounds with the same tone make a couple of octave intervals? Which intervals are harmonic and which are not, and how is their measure identified or well-defined at all? Where does the number $7$ originate from? When is a heptatonic scale maximally even? Whom is the chief idea of melody transposition due to? Whose time signature is equal to $\frac{2}{7}$?
\paragraph*{}
`Why has not such a substantial knowledge been axiomatized so far, if it could be possible?' one may ask. We must stress that during almost two thousand years, mathematicians were involved in the theory of parallelism dealing with the validity of Euclid's fifth postulate in the context of neutral geometry (see \cite{G1} for further information). Besides, the artistic aspect of music more than the scientific one, as always, used to be considered by different cultures throughout the history of mankind. Now, `What is art?' we do ask. We know ourselves how to answer. No art is able to compete against the art of mathematics which is also a knowledge. Factually, music is antecedently a knowledge before being an art, since we have the right to ask what it is primarily. Here, what is roughly meant by \emph{knowledge} is a collection of compatible information that is logically axiomatizable (formal definition could be found in mathematical logic textbooks). This attitude toward music theory in the spirit of Euclid's tradition (maybe regarded in his \textit{Elements of Music} completely lost) to clarify the questionable nature of music, providing a definite improvement over the other approaches, converts it into one branch of mathematics. (Readers are referred to \cite{E3}, \cite{B3}, and \cite{R2} for extra details about the history of mathematics.)

We believe that applying strictly logical rigor to the musical phenomena leads to astonishing mathematical insights into the perspicuous world of music never dreamed of.

\subsection{What is the axiomatic method?}
\paragraph*{}
A knowledge consists of objects with some properties. Speaking informally (refraining from introducing the formal definition of theory and other related logical notions), an \textit{axiom system} for a knowledge is a list of undefined terms, called \textit{primitive concepts}, which are selected for some objects together with a list of (declarative) statements dealing with properties of those objects, called \textit{axioms}, which are stipulated and presupposed to be true. Also, the logic rules, that is, tautologies of the propositional logic system, are latently contained in this list. In fact, the axioms are fresh rules to bind and restrict the primitive concepts not to admit any properties. Since no characteristics of undefined concepts may be used in a proof other than the properties provided by the axioms, one may also call the undefined concepts by any arbitrary names, just as Hilbert said (see \cite{G1} 3rd ed. p. 72). Every definable term in the axiom system is called a \textit{technical concept} and every provable statement by logical inferences is called a \textit{proposition}. In order to formulate a new definition with precision in the system, we are only allowed to use undefined concepts and all other technical concepts that have previously been introduced. Similarly, the process of proving a statement, say \emph{proof}, must be based on logic rules and correct statements consisting of all axioms and propositions that have already been proved. By doing so, we will not be caught in circularity about definitions and propositions of the system.

The main motivation to construct an axiom system for a knowledge is to determine precisely which properties of objects can be deduced from which others. In other words, the goal of such a construction is to discover new statements about (probably new) technical concepts, that is, propositions in the context of the axiom system. Whenever we feel to be misled into resulting in undesirable consequences, we have to turn back to modify the axioms immediately. Therefore, a great advantage of axiom systems is to deduce so many facts from a few intuitively and, of course, logically desirable principles in place of performing physical experiments with trial and error to check the truthfulness of the results. This is actually what is known as the axiomatic method to guide us via purified recognition of our ambient knowledge.
\paragraph*{}
The ambient knowledge in this paper is music. A system of axioms for this knowledge must be rigorously laid out. Our selection of undefined concepts is popular and the method of introducing all technical concepts is based on the most natural expectations of a typical musician. It is customary to postulate the axioms in such a way that they seem sufficiently self-evident as simple as intuitively obvious realities of the knowledge. Hence, we hardly ever intend to profoundly go into expert discussions of mathematical logic, even though we use formulas of symbolic logic as well (see \cite{E1} p. 161 \& 162 for beneficial notes) to write the axioms succinctly. Instead, it would be far preferable to making some efficient comments on the logical base of our axiom system. Applying the language of first-order logic, we are also going to use the language of sets so as not to discuss the existence of those abstract objects that have set-theoretic nature because it could be found out as direct consequences of the axiomatic set theory; here we mention that the set theory we make use of in this paper, ecpecially in working out cardinal numbers, is the one proposed by Zermelo-Fraenkel accompanied the axiom of choice and the Continuum Hypothesis (recommended sources in this context are \cite{H1, H5, A4, M5}). However, in this direction, we might occasionally lie on the threshold of the absolute obsessiveness. Aside from that, having not included any diagram because of representing some special case (perhaps misleadingly), the danger of inaccuracy in reasoning has decreased significantly. No claim is made that all musical notions and all mathematical features have been presented, just that those included would suffice for our purposes to furnish a regular strong basis for music theory. This is because of axiomatically doing music just for its own sake not for any kind of application, perhaps being not so gifted. Due to the orientation of the paper, tending to be much more conceptual in music principles, mathematically minded readers will achieve an excellent introduction to basic musical conceptions; at the opposite extreme, musically experienced readers will be able to get familiar with music fundamentals from the mathematically axiomatic perspective. This work is organized in a textbooklike manner suitable for every stratum of audiences.

It is necessary to mention that for discourse briefness, seemingly clear propositions for which no proof has been supplied are intentionally left to the readers, and some less obvious ones including a sketch of proof provide a warm-up for interested readers to proceed with the full proof. Furthermore, some similar proofs previously contained in the literature have been omitted. Important propositions from the mathematical or musical point of view, and those which play a helping role to prove the others, in addition to immediate consequences are respectively entitled \textit{theorem}, \textit{lemma}, and \textit{corollary}. Also, some considerable remarks are every so often given for more explanation. Finally, most of the untitled paragraphs talk beyond the axiomatic scope of the paper.

\subsection{Notes on Logic}
\paragraph*{}
If the axioms of a system does not lead to a contradiction, the axiom system is said to be \textit{consistent}; equivalently, a consistent axiom system contains not all well-formed formulas; in other words, there is no statement in the language of the system (namely concerning its technical concepts) which is both correct and incorrect. A statement (possibly axiom) is said to be \textit{independent} of the (other) axioms if it cannot be either proved or refuted from the axioms. An axiom system is said to be \textit{complete} if there are no independent statements in the language of that system; equivalently, it is possible to either prove or refute every statement in the language of the system (informally speaking, any meaningful question is answerable). A \textit{model} for an axiom system is a structure satisfying all the axioms. Roughly speaking, a model of an axiom system is an interpretation of that system by giving each undefined concept a particular meaning in such a way that all axioms are true (see \cite{G1} 4th ed. p. 72 for more information). In fact, objects of a model are concrete in spite of mathematical objects of the ambient knowledge which are abstract. A system is consistent if and only if there exists a model for it. A consistent axiom system is called \textit{categorical} if all its models are pairwise isomorphic; i.e., there is only one model up to isomorphism for it. By G\"{o}del's completeness theorem, categoricalness is stronger than completeness (\cite{E1} p. 135), but the converse is not right in first-order logic.

It is worth noting every logical notion above has its own specialized definition in the context of advanced mathematical logic with which we do not ever deal (our recommendations to interested readers for deeper study are \cite{H2, E1, M3} and probably \cite{M2}).

As different axiom systems may generate the same propositions, then there may be many alternative axiomatizations of the ambient knowledge. However, it is considered inelegant in mathematics to assume more axioms than are necessary. Therefore, we have got to pay for elegance by meticulously providing a minimal and irreducible set of axioms and arduously proving results (even seemingly obvious ones) on the basis of that. We will pay close attention to this aspect in the paper. Also, the axioms need themselves be chosen wisely to avoid redundancy and preserve consistency. In so doing, we are the first to propound a complete axiomatic system the music we enjoy is up to isomorphism the only one model for (Appendix \RNum{1}). We understand such a music system being satisfiable and decidable as well (because it is finitely axiomatizable too), though these logical discussions of our axiomatization will not be included in the article. Meanwhile, all references are displayed in order of priority.

\section{Pitch Music}
\paragraph*{}
In order to introduce the elementary part of music, i.e. \textit{pitch music}, we assume the undefined concepts ``sound'' and ``lower-pitched'', the first of which is just a property designated by lower-case letters and the second one is a relation between sounds, written ``$x$ is lower-pitched than $y$'' or ``$x$ is lower in pitch than $y$'', and denoted $x*y$. The inverse relation and the negation of both guide around great discrimination in sounds having higher pitch and those identical tuning respectively in such a way as follows.
\begin{remark}[Notation]
Throughout the paper $\mathbb{S}$ denotes the set of all sounds (so, $*$ is a binary relation on $\mathbb{S}$).
\end{remark}
\begin{definition}
If a sound $x$ is lower in pitch than a sound $y$, then $y$ is said to be \emph{higher-pitched} or \emph{higher in pitch} than $x$. In other words, the inverse of the binary relation of being lower-pitched, namely $*^{-1}$, is named being higher-pitched.
\end{definition}
\begin{definition}
A sound $x$ is said to be \emph{identical-pitched} or \emph{identical in pitch} with a sound $y$ if none is lower-pitched than the other. In this case, we write $x\sim y$; i.e.,
\begin{center}
$\forall x,y\in \mathbb{S}(x\sim y\Leftrightarrow \neg (x*y\vee y*x)).$
\end{center}
Otherwise, they are said to be \emph{non-identical (in pitch)}, written $x\nsim y$.
\end{definition}
\begin{corollary}
For every (not necessarily distinct) pair of sounds $x$ and $y$, exactly one of the following holds: $x*y$, $y*x$, or $x\sim y$.
\end{corollary}
\paragraph*{}
We differentiate between the identity relation $\sim$ and the equality relation $=$.
\begin{axiom}[Axiom of Irreflexivity in Pitch]
No sound is lower in pitch than itself; i.e.,
\begin{center}
$\forall x\in \mathbb{S}(\neg (x*x)).$
\end{center}
\end{axiom}
\begin{axiom}[Axiom of Transitivity in Pitch]
For every triple of sounds $x$, $y$, $z$, if $x$ is lower in pitch than $y$ and $y$ is lower in pitch than $z$, then $x$ is lower in pitch than $z$; that is,
\begin{center}
$\forall x,y,z\in \mathbb{S}((x*y\wedge y*z)\Rightarrow x*z).$
\end{center}
\end{axiom}
\begin{axiom}[Axiom of Existence]
There exists a pair of non-identical sounds; more precisely,
\begin{center}
$\exists x,y\in \mathbb{S}(x*y\vee y*x).$
\end{center}
\end{axiom}
\begin{corollary}
$*$ is a nonempty strict (not partial) ordering on $\mathbb{S}$.
\end{corollary}
\begin{remark}
If we replace Axiom 1 with a (slightly more intuitive and surely) stronger one, that is,
\begin{center}
$\forall x,y\in \mathbb{S}(\neg (x*y \wedge y*x))$,
\end{center}
say ``Axiom of Asymmetry in Pitch'', we then obtain the same consequences.
\end{remark}
\paragraph*{}
In order to distinguish the order $*$ from a type of lexicographic one, which is musically counter-intuitive, a new axiom is required as follows. The reason for the name of such an axiom will be clarified at the end of the section.
\begin{axiom}[Axiom of Separation]
If a sound $x$ is lower-pitched than a sound $y$, then there is no sound $z$ which is identical-pitched with both $x$ and $y$; that is,
\begin{center}
$\forall x,y,z\in \mathbb{S}(x*y \Rightarrow \neg (x\sim z \wedge y\sim z)).$
\end{center}
\end{axiom}
\paragraph*{}
Relaxing the mind, we now have got the following.
\begin{corollary}
$\sim$ is a nonempty equivalence relation on $\mathbb{S}$.
\end{corollary}
\begin{remark}[Notation]
The equivalence class of $x\in \mathbb{S}$ under $\sim$ is denoted by $\tilde{x}$; that is, $\tilde{x}:=\{a\in \mathbb{S}: a\sim x\}.$
\end{remark}
\begin{lemma}
For sounds $x$, $y$, $a$, $b$, if $x\sim a$, $y\sim b$, and $x*y$, then $a*b$.
\end{lemma}
\begin{remark}
One can define the order $*'$ on the partition $\mathbb{S}/\sim \,=\{\tilde{x}: x\in \mathbb{S}\}$ of $\mathbb{S}$ induced by $\sim$ as follows:
\begin{center}
$\tilde{x}*'\tilde{y}\Longleftrightarrow x*y\vee x\sim y.$
\end{center}
It is easy to check that $*'$ is a linear (or total) ordering on $\mathbb{S}/\sim$.
\end{remark}

\subsection*{Pitch in Between}
\paragraph*{}
Up to this point, $ \mathbb{S}$ has two elements. Now, we develop our musical universe.
\begin{definition}
We say that a sound $x$ is \emph{between-pitched} a pair of (non-identical) sounds if x is higher-pitched than one of the sounds and lower-pitched than the other.
\end{definition}
\begin{remark}[Notation]
We employ the following shorthand for the ternary relation in Definition 2.11:
\begin{center}
$x*y*z\Longleftrightarrow x*y\wedge y*z.$
\end{center}
\end{remark}
\begin{proposition}
For every triple of pairewise non-identical sounds, exactly one is between-pitched the other two.
\end{proposition}
\begin{axiom}[Axiom of Betweenness in Pitch]
$\,\,\,\,\,\,\,\,\,\,\,\,\,\,\,\,$
\begin{enumerate}
\item Each sound is between-pitched some sounds; i.e.,
\begin{center}
$\forall x\in \mathbb{S}(\exists a,b\in \mathbb{S}(a*x*b))$.
\end{center}
\item For any two non-identical sounds, there exists a sound between-pitched them; i.e.,
\begin{center}
$\forall x,y\in \mathbb{S}(x*y\Rightarrow \exists a\in \mathbb{S}(x*a*y))$.
\end{center}
\end{enumerate}
\end{axiom}
\begin{remark}
Item 1 of the Axiom of Betweenness in Pitch is practically a compound statement expressing that for a given sound $x$, there exists a sound lower-pitched than $x$, and there exists a sound higher-pitched than $x$ (the two sounds obtained are obviously non-identical in pitch). The second item analytically expresses that the ordered sets $(\mathbb{S},*)$ and $(\mathbb{S}/\sim \,,*')$ are dense.
\end{remark}
\begin{proposition}[Pitch Directness Property]
For arbitrary sounds $x$ and $y$,
\begin{enumerate}
\item there exists a sound lower in pitch than both $x$ and $y$.
\item there exists a sound higher in pitch than both $x$ and $y$.
\end{enumerate}
\begin{proof}[Sketch of Proof]
Apply item 1 of Axiom 5 to $x$ and $y$ separately and then use Corollary 2.4.
\end{proof}
\end{proposition}
\begin{definition}[Generalization of Definition 2.11]
For every natural number $n>3$, one can generalize the binary relation $*$ to an n-ary relation on $\mathbb{S}$ (namely a subset of $\mathbb{S}^n$) as follows:
\begin{center}
$\forall x_0,x_1, ..., x_{n-1}\in \mathbb{S}(x_0*x_1*...*x_{n-1}\Leftrightarrow \forall 1\leq k < n(x_{k-1}*x_k)).$
\end{center}
\end{definition}
\begin{proposition}
For every $n\in \mathbb{N}$, there are $x_1, x_2, ..., x_n\in \mathbb{S}$ such that $x_1*x_2*...*x_n$.
\end{proposition}
\begin{proof}[Sketch of Proof]
By induction on $n$.
\end{proof}
\begin{corollary}
$\mathbb{S}$, $*$, $*'$, $\sim$, and $\mathbb{S}/\sim$ are at least countable.
\end{corollary}
\begin{proposition}
If $A\subseteq \mathbb{S}$ is finite, then there exist two (non-identical) sounds so that any element of $A$ is between-pitched them.
\end{proposition}
\begin{proof}[Sketch of Proof]
Take the minimum and maximum elements of $(A/\sim \,,*')$ and call them $\tilde{x}$ and $\tilde{y}$. Applying item 1 of Axiom 5 to $x$ and $y$ does the job.
\end{proof}
\paragraph*{}
Until now, we just know that $\mathbb{S}$ has infinitely many elements. In the future, we will realize that all sets mentioned in Corollary 2.18 are uncountable. In the next definition, we are going to introduce a new ternary relation on $\mathbb{S}$ by virtue of which the so-called ``Sound Separation Property'' is more conveniently capable of being stated.
\begin{definition}
Let $\theta$ be an arbitrary sound, $x$ and $y$ any sounds that which not identical in pitch with $\theta$. If $x\sim y$ or $\theta$ is not between-pitched $x$ and $y$, we say $x$ and $y$ are \emph{on the same side} of $\theta$. Otherwise, we say $x$ and $y$ are \emph{on opposite sides} of $\theta$.
\end{definition}
\begin{remark}
Notice that given an arbitrary sound $\theta$, the two relations of definition 2.20 in relation to $\theta$ are regarded as binary relations on the set of sounds not identical-pitched with $\theta$.
\end{remark}
\begin{proposition}
For every $\theta \in \mathbb{S}$, the relation of \emph{being on the same side} of $\theta$ is an equivalence relation on $\mathbb{S}-\tilde{\theta}$.
\end{proposition}
\begin{proof}
Let $\theta$ be given once and for all. We need to check the following three properties: reflexivity, symmetry, and transitivity. The first two are manifest according to Axiom 1 and Definition 2.20 respectively. For transitivity, suppose $x,y,z\in \mathbb{S}-\tilde{\theta}$, and $x$, $y$ are on the same side of $\theta$, and $y$, $z$ are on the same side of $\theta$. If $x$, $z$ are not on opposite sides of $\theta$, then either $x*\theta *z$ or $z*\theta *x$. we assume that $x*\theta *z$ (the proof of the second case is similar). By corollary 2.4, only one of the following holds: $x\sim y$, $y*x$, or $x*y$.
The first two imply that $y*\theta$ (by virtue of Lemma 2.9 and Axiom 2, respectively). Then by assumption we get $y*\theta *z$ which contradicts the fact that sounds $y$ and $z$ are on the same side of $\theta$. Inevitably, we must have $x*y$. Having noted $x$ and $y$ are on the same side of $\theta$, since $x$ is lower-pitched than $\theta$, so is $y$, concluding $x*y*\theta$. Now we again obtain $y*\theta *z$ leading to the same contradiction.
\end{proof}
\paragraph*{}
The following theorem is the very reason why the Axiom of Separation is named so.
\begin{theorem}[Sound Separation Property]
For every $\theta \in \mathbb{S}$, the relation of being on the same side of $\theta$ on $\mathbb{S}-\tilde{\theta}$ has precisely two equivalence classes.
\end{theorem}
\begin{proof}
Fix a $\theta \in \mathbb{S}$. There are two sounds $x$ and $y$ such that $x*\theta *y$ (why?). The two equivalence classes $[x]$ and $[y]$ are distinct, because if they have some sound in common, they must then be on the same side of $\theta$ by Proposition 2.22, contradicting the first statement. Thus, there are at least two equivalence classes $[x]$ and $[y]$. On the other hand, if $a\in \mathbb{S}-\tilde{\theta}$, then either $a*\theta$ or $\theta*a$. So, either $a$ and $x$ are both lower in pitch than (and consequently on the same side of) $\theta$, or $a$ and $y$ are both higher in pitch than $\theta$. It follows that either $a\in [x]$ or $a\in [y]$. Thus, there are at most two equivalence classes $[x]$ and $[y]$.
\end{proof}
\begin{remark}
Notice that in the proof of Theorem 2.23, the two classes $[x]$ and $[y]$ respectively coincide with the set of sounds lower-pitched than $\theta$ and the set of sounds higher-pitched than $\theta$. In conclusion, every sound of $[x]$ is lower in pitch than ($\theta$ and) every sound of $[y]$.
\end{remark}
\begin{definition}
Each equivalence class mentioned in Theorem 2.23 is called a \emph{side} of the sound $\theta$. The side consisting of those sounds which are lower in pitch than $\theta$ is called the \emph{low side} of $\theta$ and denoted by $\overset{\leftarrow}{\theta}$, and the side whose sounds are higher-pitched than $\theta$ is called the \emph{high side} of $\theta$ and denoted by $\overset{\rightarrow}{\theta}$.
\end{definition}
\begin{corollary}
For every $\theta \in \mathbb{S}$, $\mathbb{S}-\tilde{\theta}=\overset{\leftarrow}{\theta} \overset{\circ}{\cup} \overset{\rightarrow}{\theta}$.
\end{corollary}
\begin{remark}
Theorem 2.23 roughly states that every $\theta\in \mathbb{S}$ has just two sides and any pair of sounds non-identical with $\theta$ are either on one side or on the two sides of $\theta$; thus in this sense, $\theta$ separates all sounds of the system in the form of Corollary 2.26. Note each side of $\theta$ has at least countably many sounds.
\end{remark}
\paragraph*{}
Eventually, loosely speaking, the aim of pitch music so presented was that \textit{pitch} is one-dimensional. However, we do not know enough about identical sounds thus far. We do not yet know whether there exist some distinct sounds identical in pitch with a given sound (other representatives of $\tilde{x}$ except $x$) or the relation of identity $\sim$ exactly coincides with the relation of equality $=$; equivalently, whether or not any pair of sounds are comparable under the order relation $*$. Of course, musicians would like to have got different musical instruments available! We postpone discussing the matter to the prospective sections. The only thing we are at the moment aware of is that there are at least one set of representatives for the equivalence $\sim$ due to the axiom of choice of set theory. In fact, there will be uncountably many ones in the future (again due to the axiom of choice). The significance of such kind of set of sounds is that, having introduced the notion of \textit{frequency}, they represent a minimal subset of $\mathbb{S}$ containing a copy of every sound with arbitrary frequency.

\section{Interval Music}
\paragraph*{}
In this section, we supply the most basic musical notion, i.e. interval, which intuitively describes the musical distance between sounds.
\begin{definition}
Given sounds $a$ and $b$. The \emph{interval} $[a,b]$ is defined as the ordered pair $(\tilde{a},\tilde{b})$.
\end{definition}
\begin{corollary}
$\forall a,b,c,d\in \mathbb{S}([a,b]=[c,d]\Leftrightarrow a\sim c\wedge b\sim d).$
\end{corollary}
\paragraph*{}
We also need an additional primitive concept to explain the relationship between intervals, namely ``congruence''. We use the notation $[a,b]\cong [c,d]$ to express that ``$[a,b]$ is congruent to $[c,d]$''.
\begin{remark}
We denote the set of all intervals by $\mathbb{I}$; i.e., $\mathbb{I}:=(\mathbb{S}/\sim)\times(\mathbb{S}/\sim)$. Based on Corollary 2.18, $\mathbb{I}$ is at least countable. In addition, the congruence $\cong$ is a binary relation on $\mathbb{I}$.
\end{remark}
\begin{axiom}[Axiom of Reflexivity of Congruence]
Every interval is congruent to itself.
\end{axiom}
\begin{axiom}[Axiom of Transitivity of Congruence]
If $[a,b]\cong [c,d]$ and $[c,d]\cong [e,f]$, then $[a,b]\cong [e,f]$.
\end{axiom}
\begin{axiom}[Axiom of Motion]
Given any interval $[x,y]$ and any sound $a$, there is a unique sound $b$, up to identity in pitch, such that $[x,y]\cong [a,b]$; that is,
\begin{center}
$\forall x,y,a\in \mathbb{S}(\exists b\in \mathbb{S}([x,y]\cong [a,b]\wedge \forall b'\in \mathbb{S}([x,y]\cong [a,b']\Rightarrow b\sim b'))).$
\end{center}
\end{axiom}
The reason for naming Axiom 8 so is the intuitive fact that congruent intervals can move to lie on each other.
\begin{proposition}[Symmetry of Congruence]
For any sounds $a$, $b$, $c$, and $d$, $[a,b]\cong [c,d]$ implies $[c,d]\cong [a,b]$.
\end{proposition}
\begin{proof}
Suppose that $[a,b]\cong [c,d]$. By applying the Axiom of Motion to the interval $[c,d]$ and the sound $a$, there is $b'\in \mathbb{S}$ so that $[c,d]\cong [a,b']$. We deduce that $[a,b]\cong [a,b']$ by transitivity of congruence. On the other hand, we have $[a,b]\cong [a,b]$ by Axiom 6. Since the sound $b'$ is unique up to identity (Axiom of Motion), it follows that $b\sim b'$. Now, we obtain $[a,b']=[a,b]$ according to the definition. Finally, we conclude that $[c,d]\cong [a,b]$, as desired.
\end{proof}
\begin{corollary}
The relation $\cong$ is an equivalence on $\mathbb{I}$.
\end{corollary}
\begin{corollary}[Euclid's First Common Notion]
Intervals congruent to the same interval are congruent to each other.
\end{corollary}
\begin{remark}
If we replace Axiom 7 with Euclid's First Common Notion, i.e.,
\begin{center}
$\forall a,b,c,d,e,f\in \mathbb{S}(([a,b]\cong [c,d]\wedge [a,b]\cong [e,f])\Rightarrow [c,d]\cong [e,f])$,
\end{center}
we (more easily) obtain the same consequences.
\end{remark}
\begin{lemma}
Let $[a,b]\cong [c,d]$. Then $a\sim c$ if and only if $b\sim d$.
\end{lemma}
\begin{proof}
Suppose that $a\sim c$. Applying the Axiom of motion gives an $e\in \mathbb{S}$ such that $[a,b]\cong [c,e]$. On the other hand, $[a,b]=[c,b]$ by definition, and using Axiom 6 turns out that $[a,b]\cong [c,b]$. Moreover, $[a,b]\cong [c,d]$ (hypothesis). It follows that $b,d\in \tilde{e}$ by uniqueness up to identity. So $b\sim d$. The proof of the converse is similar.
\end{proof}
\begin{definition}
The interval $[a,b]$ is said to be \emph{unit} if $a\sim b$, it is said to be \emph{greater than unit} if $a*b$, and is said to be \emph{less than unit} if $a*^{-1}b$.
\end{definition}
\begin{definition}
The \emph{conversion} of the interval $[a,b]$ is defined to be the interval $[b,a]$, denoted $[a,b]^{-1}$.
\end{definition}
\begin{remark}[Notation]
We denote by $\mathbb{I}^{>1}$ the set of all intervals which are greater than unit and by $\mathbb{I}^{<1}$  the set of all intervals which are less than unit.
\end{remark}
\begin{corollary}
$\forall a,b\in \mathbb{S}(a\sim b\Leftrightarrow [a,b]=[a,b]^{-1})$.
\end{corollary}
\paragraph*{}
Up to this point, only geometric properties of intervals have been revealed within axioms based on congruence. Musical intuition does not allow us to think of the elements of $\mathbb{I}^{>1}$ to be congruent to the elements of $\mathbb{I}^{<1}$. Moreover, it is natural that whatever we want for $\mathbb{I}^{>1}$ should likewise hold for $\mathbb{I}^{<1}$. We expect the relationship between pitch and congruence to be \emph{regular} in this sense. To better visualize the matter, consider Corollary 3.12. Does the statement remain to be true if we replace $=$ with $\cong$? The directive part of the statement is obvious by definition and Axiom 6, but the converse (if $[a,b]\cong [b,a]$, then $a\sim b$) cannot be proved since it is actually independent. The correctness of the converse which is equivalent to a positive answer to the foregoing question displays a specific case of a musical property included in the following axiom.
\begin{axiom}[Axioms of Regularity]
Let $a,b,c,d\in \mathbb{S}$ and $a*b$.
\begin{enumerate}
\item If $[a,b]\cong [c,d]$, then $c*d$.
\item $[a,b]$ and $[c,d]$ are congruent if and only if so are their conversions.
\end{enumerate}
\end{axiom}
\begin{corollary}
No interval is congruent to its conversion unless it is unit; that is,
\begin{center}
$\forall a,b\in \mathbb{S}(a\sim b\Leftrightarrow [a,b]\cong [a,b]^{-1})$.
\end{center}
\end{corollary}
\paragraph*{}
Our first main motive for using the term \textit{regularity} as the name of Axiom 9 is in fact the ``only if'' part of Corollary 3.13; but unfortunately, it does not imply even item 1 of the Axioms of Regularity, and we will not be able to do all needed to be done by replacing them. In item 1 of Remark 3.14 we shall present some models for our music system up to Axiom 8, satisfying Corollary 3.13, in which the Axioms of Regularity fail to hold. This means Corollary 3.13 is logically weaker than Axiom 9 and not useful enough to be postulated.
\begin{remark}
\begin{enumerate}
\item Let \emph{sounds} mean positive real numbers, i.e., $\mathbb{S}=\mathbb{R}^+$. Interpret $*$ as the usual ordering $<$ on $\mathbb{R}$. So $\sim$ coincides with the equal relation $=$, and the interval $[x,y]$ is the same ordered pair as $(x,y)$. Now define $(x,y)\cong (z,t)$ by
$2y-x=2t-z$. One can check the veracity of each of the first eight axioms in this interpretation. But none of the items of Axiom 9 is satisfied. This model shows that the Axioms of Regularity are independent of all previous axioms. However, Corollary 3.13 does perfectly hold.
\item Here is another interesting model disclosing that even from the hypothesis of the satisfaction of the second item of the Axioms of Regularity, the first item is not deducible. We consider the unit circle $S^1$ in the Cartesian model for Euclidean geometry as a musical manifold, namely, $\mathbb{S}=S^1$. Thus the points of the circle represent \emph{sounds} of our music. By convention all circular arcs on $S^1$ are considered counterclockwise (in the trigonometric direction). A sound $x$ \emph{is lower in pitch than} a sound $y$ if they do not lie on the same diameter and the length of the circular arc $\stackrel{\frown}{xy}$ equals the distance from $x$ to $y$ using Riemannian metric on the $S^1$ (see \cite{L3} and \cite{L2} for having a profound outlook on the world of manifolds), i.e. the smallest distance between $x$ and $y$. Clearly, $x\sim y$ iff $x$ and $y$ lie on the same diameter of $S^1$, that is, $x$ and $y$ are antipodal points. Hence, any equivalence class $\tilde{x}$ consists of just two sounds $x$ and the other one making together antipodal points of the corresponding diameter.
Another simpler equivalent interpretation of $*$ is that $x*y$ if and only if the measure of the central angle corresponding to $\stackrel{\frown}{xy}$ is nonzero and less than $\pi$. So, $x$ is higher-pitched than $y$ if and only if the measure of the central angle corresponding to $\stackrel{\frown}{xy}$ is greater than $\pi$. $x\sim y$ iff the measure of the corresponding central angle is equal to $0$ or $\pi$. Having interpreted the two undefined technical concepts of pitch music and having satisfied its axioms, we get interpretations of all defined technical concepts. For example, each \emph{side} of $x$ is an open semicircle of $S^1$ surrounded by the two antipodal points belonging to $\tilde{x}$. We now interpret \emph{congruence} in such a way that $[x,y]\cong [z,t]$ if the central angles corresponding to $\stackrel{\frown}{xy}$ and $\stackrel{\frown}{zt}$ are congruent in Euclidean sense or their difference is $\pi$, preserving the orientation (in accordance with our convention about arcs). One can observe that all axioms thus far stated in interval music (in addition to Corollary 3.13) are satisfied except item 1 of Axiom 9. The reason for failing is left to the reader.
\end{enumerate}
\end{remark}
\paragraph*{}
Roughly speaking, on the base of the Axioms of Regularity, Remark 3.14 at once reveals pitch is no more circular but rather flat and the relationship between pitch and congruence is not certainly linear, as we will see perspicuously. Let us turn back to the main framework.
\begin{exercise}
Suppose that $[a,b]\cong [b,c]$. Prove that the following conditions are equivalent:
\begin{enumerate}
\item $a$ and $c$ are on opposite sides of $b$.
\item $[a,b]$ or $[b,c]$ is not unit.
\item Neither $[a,b]$ nor $[b,c]$ is unit.
\item $[a,b]\neq [b,c]$.
\end{enumerate}
\end{exercise}
\begin{proposition}
All unit intervals are congruent to each other.
\end{proposition}
\begin{proof}
We assume that $a\sim b$ and $c\sim d$. We prove that $[a,b]\cong [c,d]$. By the Axiom of Motion there is a sound $d'$ such that $[a,b]\cong [c,d']$. So $a\sim b$ implies that $c\sim d'$ by means of Corollary 2.4 and the Axioms of Regularity. It follows that $d\sim d'$ (transitivity of identity) and then $[c,d]=[c,d']$ (by definition). Finally, $[a,b]\cong [c,d]$.
\end{proof}
\begin{corollary}[Generalization of Axioms of Regularity] $\,\,\,\,\,\,\,\,\,\,\,\,\,\,\,$
\begin{enumerate}
\item Let $R\in \{*,*^{-1},\sim \}$ and $[a,b]\cong [c,d]$. Then $aRb$ if and only if $cRd$.
\item $[a,b]\cong [c,d]$ if and only if $[a,b]^{-1}\cong [c,d]^{-1}$.
\end{enumerate}
\end{corollary}
\paragraph*{}
The contrapositive of Corollary 3.12 in comparison with Corollary 3.13 may look more stimulative to ask whether a given non-unit interval is intuitively larger or its conversion. From now on, we try to found some new axiom with the aim of defining an order on the sets $\mathbb{I}^{>1}$ and $\mathbb{I}^{<1}$ based on the notion of congruence.
\begin{definition}
For every pair of sounds $p$ and $q$, the \emph{segment} $\overline{pq}$ \emph{in $(\mathbb{S},*)$} is defined as
\begin{center}
$\{x\in \mathbb{S}:x\sim p \,\,\text{or}\,\, x\sim q \,\,\text{or}\,\, \text{x is between-pitched p and q}\}.$
\end{center}
Any segment in $(\mathbb{S},*)$ which is a (proper) subset of a given segment is called a (\emph{proper}) \emph{subsegment} of that. The set of all sounds lower (or higher) in pitch than a given sound $s$ is called an \emph{initial \emph{(or} final\emph{)} segment in $(\mathbb{S},*)$}. A set of sounds $F$ is said to be \emph{convex} if $\overline{pq}\subseteq F$ whenever $p,q\in F$. The \emph{convex hull} of an $F\subseteq \mathbb{S}$ is the smallest convex superset of $F$. For every $F\subseteq \mathbb{S}$, $x\in F$ is called an \emph{extreme sound} of $F$ if x does not belong to any segment $\overline{pq}$ included in $\bigcup F/\sim \,=\bigcup _{s\in F} \tilde{s}$ unless x is identical in pitch with p or q; more precisely,
\begin{center}
$\forall p,q\in \mathbb{S}(x\in \overline{pq}\wedge \dfrac{\overline{pq}}{\sim}\subseteq \dfrac{F}{\sim}\Rightarrow x\in \tilde{p}\cup \tilde{q}).$
\end{center}
\end{definition}
\begin{remark}[Notation]
The convex hull of an $F\subseteq \mathbb{S}$, the set of extreme sounds of $F$, and the set of all segments in $(\mathbb{S},*)$ are respectively denoted by $Con(F)$, $Ext(F)$, and $Seg(\mathbb{S},*)$.
\end{remark}
\paragraph*{}
In the following propositions the main properties of segments are listed, not deeply related to the matter, whose proofs are entirely left to the reader as an exercise.
\begin{proposition}
Given $F\subseteq \mathbb{S}$. If $p$ is an extreme sound of $F$, then so is every sound identical-pitched with $p$.
\end{proposition}
\begin{proposition}
Let $\overline{rs}\subseteq \overline{pq}$. $\overline{rs}$ is a proper subsegment of $\overline{pq}$ iff one of the extreme sounds of $\overline{rs}$ (equivalently r or s) is between-pitched $p$ and $q$.
\end{proposition}
\begin{proposition}
All segments together with initial and final ones are convex. For any $s\in \mathbb{S}$, $\overset{\leftarrow}{s}$ is the low side of $s$ and $\overset{\rightarrow}{s}$ is the high side of $s$ ($\overset{\leftarrow}{s}\stackrel{\circ}{\cup} \tilde{s}\stackrel{\circ}{\cup}\overset{\rightarrow}{s}=\mathbb{S}$). Moreover,
\begin{center}
$\overset{\leftarrow}{s}=\underset{p*s}{\bigcup}\overline{ps}-\tilde{s},\,\,\,\,\,\overset{\rightarrow}{s}=\underset{s*p}{\bigcup}\overline{ps}-\tilde{s}$.
\end{center}
\end{proposition}
\begin{proposition}[Well-definedness Property of Convex Hull]
The intersection of every family of convex sets is convex. Furthermore, The convex hull of any $F\subseteq \mathbb{S}$ is the intersection of all convex sets of sounds containing $F$ (which exists uniquely by the elementary axioms of set theory).  
\end{proposition}
\begin{proposition}
$F\subseteq \mathbb{S}$ is convex if and only if $F=Con(F)$.
\end{proposition}
\begin{proposition}[Monotonicity of Convex Hull]
For every $F\subseteq \mathbb{S}$ we have the following:
\begin{center}
$Con(F)=\underset{p,q\in F}{\bigcup}\overline{pq}$.
\end{center}
\end{proposition}
\begin{remark}
One can characterize all convex sets on the base of segments, initial segments and final segments. In this context, a very special case of Krein-Milman theorem holds (the interested reader should search in functional analysis textbooks such as \cite{C3} and \cite{R4} for more information), stating that for every nonempty convex subset of $\mathbb{S}$ like $F$ which does not contain any initial or final segment in $(\mathbb{S},*)$, the equality $F=Con(Ext(F))$ is satisfied.
\end{remark}
\begin{proposition}
For every $p,q\in \mathbb{S}$, the following conditions hold:
\begin{enumerate}
\item $\overline{pq}=\overline{qp}.$
\item $\overline{pq}=\tilde{p}\Leftrightarrow p\sim q.$
\item $\bigcup \frac{\overline{pq}}{\sim}=\overline{pq}.$
\item $Ext(\overline{pq})=\tilde{p} \cup \tilde{q}.$
\item $Con(\{p,q\})=\overline{pq}.$
\end{enumerate}
\end{proposition}
\begin{proposition}
For every $p,q,r,s\in \mathbb{S}$, the following possibilities are equivalent:
\begin{enumerate}
\item $\overline{pq}=\overline{rs}.$
\item $Con(\overline{pq})=Con(\overline{rs}).$
\item $Ext(\overline{pq})=Ext(\overline{rs}).$
\item $[p,q]=[r,s]$ or $[p,q]=[r,s]^{-1}$.
\item $(p\sim r\wedge q\sim s)$ or $(p\sim s\wedge q\sim r)$.
\end{enumerate}
\end{proposition}
\begin{theorem}
There are unique one-to-one correspondences
\begin{center}
$\varphi : \mathbb{I}-\mathbb{I}^{<1}\rightarrow Seg(\mathbb{S},*),\,\,\,\,\, \psi : \mathbb{I}-\mathbb{I}^{>1}\rightarrow Seg(\mathbb{S},*)$
\end{center}
such that $Ext(\varphi ([a,b]))=\tilde{a}\cup \tilde{b}$ and $Ext(\psi ([c,d]))=\tilde{c}\cup \tilde{d}$ for any $[a,b]\in \mathbb{I}-\mathbb{I}^{<1}$  and any $[c,d]\in \mathbb{I}-\mathbb{I}^{>1}$. Moreover, $\varphi$ and $\psi$ are compatible; i.e., they coincide on the set of all unit intervals.
\end{theorem}
\begin{proof}
Define naturally $\varphi ([a,b])=\overline{ab}$ for every $[a,b]\in  \mathbb{I}-\mathbb{I}^{<1}$. $\varphi$ is well-defined by Proposition 3.28 ($4\Rightarrow 1$). $\varphi$ is surjective (or onto) because for any segment $\overline{ab}$, we may assume $a*b$ or $a\sim b$ without loss of generality (by item 1 of Proposition 3.27). So, clearly $\varphi ([a,b])=\overline{ab}$. $\varphi$ is injective (or one-to-one) by Proposition 3.28 ($1\Rightarrow 4$) together with Lemma 2.9 and item 2 of Proposition 3.27. Therefore, $\varphi$ is bijective. Also, Proposition 3.27 (item 4) simply yields $Ext(\varphi ([a,b]))=\tilde{a}\cup \tilde{b}$ for every $[a,b]\in \mathbb{I}-\mathbb{I}^{<1}$. Uniqueness of $\varphi$ with the latter property is straightforward by Proposition 3.28 ($3\Rightarrow 1$). The same definition does work for $\psi$ as well. Eventually, the compatibility of $\varphi$ and $\psi$ with these features will become obvious (item 2 of Proposition 3.27).
\end{proof}
\paragraph*{}
Although Theorem 3.29 tells us that the concept of \textit{interval} could be introduced in terms of \textit{segment}, we did not so. If this was the case, our motivational fashion of doing music would be tainted. We emphasize that we are no more interested in proceeding with segments in our music, because it basically specifies an alternative method of defining the notion of interval which is customary in the foundations of geometry toward line segments but hinders us from achieving prospective musical purposes in this ambience. On the other extreme, we could have considered interval as an additional primitive concept, axiomatically endowed with just the right properties, no more and no less. Of course, the mathematician's choice is pretty clearly a matter of taste, but preferable is to pay for conceptual economy while being equipped with the language of sets (see \cite{H1} p. 25).
\paragraph*{}
One could be contrapositively thinking of the conclusion of Lemma 3.8 in such a way that if $a*c$, provided that $[a,b]\cong [c,d]$, then either $b*d$ or $d*b$. Of course, we would intuitively like to get the first result, but there is no assurance about what exactly occurs in general. Applying the language of segments, `Is there any segment $\overline{ab}$ with a proper subsegment $\overline{cd}$ so that $\varphi ^{-1}(\overline{ab})\cong \varphi ^{-1}(\overline{cd})$?' one could ask. Take notice of the following remark.
\begin{remark}
Let $a*b*c$. One can easily observe that it is impossible to have $[a,b]\cong [a,c]$ (Lemma 3.8). Now, let $a*b*c*d$ ($\overline{bc}$ is a proper subsegment of $\overline{ad}$). Is it possible to prove or disprove that $[b,c]$ is congruent to $[a,d]$? As expected, the corresponding interval of a given segment must not be congruent to the corresponding interval of any proper subsegment of the given segment. Unfortunately, this expectation is not provable in the axiomatic system so far constructed and we need a fresh axiom to substantiate it.
\end{remark}
\begin{axiom}[Interval Addition Axiom]
Assume that $a*b*c$.
\begin{enumerate}
\item If $[a,b]\cong [a',b']$ and $[b,c]\cong [b',c']$, then $[a,c]\cong [a',c']$.
\item If $[a,b]\cong [a',b']^{-1}$ and $[b,c]\cong [b',c']^{-1}$, then $[a,c]\cong [a',c']^{-1}$.
\end{enumerate}
\end{axiom}
\begin{proposition}[Generalization of Interval Addition Axiom]
The Interval Addition Axiom holds for arbitrary $a,b,c,a',b',c'\in \mathbb{S}$.
\end{proposition}
\begin{proof}
We first check the case $c*a*b$ for item 1 of the Interval Addition Axiom. There is an $x\in \mathbb{S}$ such that $[a,c]\cong [a',x]$. So $[c,a]\cong [x,a']$ (item 2 of Corollary 3.17). Since $[a,b]\cong [a',b']$ (assumption), we conclude that $[c,b]\cong [x,b']$ by the first item of the axiom. It follows that $[b,c]\cong [b',x]$. On the other hand, we have $[b,c]\cong [b',c']$ which implies that $[b',x]\cong [b',c']$ (Corollary 3.6). Lemma 3.8 yields $x\sim c'$ and in turn $[a',x]=[a',c']$. Thus $[a,c]\cong [a',c']$. We now consider the case $a*b$ and $b\sim c$. From $[b,c]\cong [b',c']$ it follows that $b'\sim c'$ (item 1 of Corollary 3.17). We obtain $[a,c]=[a,b]\cong [a',b']=[a',c']$. The other cases such as $b*c*a$, $a\sim b \wedge b*c$, $a\sim b \wedge b\sim c$, and so on are similar to check and left to the reader. For the second item, we deal only with the case $a*c*b$. There exists a sound $y$ so that $[a,c]\cong [c',y]$. It means that $[a,c]\cong [y,c']^{-1}$. Based on the hypothesis and item 2 of the Interval Addition Axiom we get $[a,b]\cong [y,b']^{-1}$. But the hypothesis ($[a,b]\cong [a',b']^{-1}$) necessitates $a'$ and $y$ being identical-pitched, in which case we conclude $[a,c]\cong [a',c']^{-1}$, as desired.
\end{proof}
\begin{proposition}[Interval Subtraction]
Given $a,b,c,a',b',c'\in \mathbb{S}$,
\begin{enumerate}
\item if $[a,b]\cong [a',b']$ and $[a,c]\cong [a',c']$, then $[b,c]\cong [b',c']$.
\item if $[a,b]\cong [a',b']^{-1}$ and $[a,c]\cong [a',c']^{-1}$, then $[b,c]\cong [b',c']^{-1}$.
\end{enumerate}
\end{proposition}
\begin{proof}[Sketch of Proof]
Apply the Generalization of Interval Addition Axiom.
\end{proof}
\begin{proposition}
$[a,b]\cong [a',b']$ iff $[a,a']\cong [b,b']$.
\end{proposition}
\begin{proof}
It suffices to show that just one hand side implies the other. Suppose $[a,b]\cong [a',b']$. So $[a,b]\cong [b',a']^{-1}$. We know $[b,a']\cong [a',b]^{-1}$ (Axiom 6). From the Generalization of Interval Addition Axiom, it immediately follows that $[a,a']\cong [b',b]^{-1}=[b,b']$.
\end{proof}
\begin{corollary}[Generalization of Lemma 3.8]
Let $R\in \{*,*^{-1},\sim \}$ and $[a,b]\cong [c,d]$. Then $aRc$ iff $bRd$.
\end{corollary}
\begin{proof}[Sketch of Proof]
Compare the statement with item 1 of Corollary 3.17.
\end{proof}
\begin{remark}
Corollary 3.34 simply explains why it is impossible to have $[b,c]\cong [a,d]$ under the condition that $a*b*c*d$ (mentioned in Remark 3.30).
\end{remark}
\begin{exercise}
Without using the Axiom of Motion prove that all unit intervals are congruent.
\end{exercise}
\paragraph*{}
Notice Exercise 3.36 exhibits a theoretical justification for compatibility of facts listed in the axiom system. Let us move in on geometry.
\begin{lemma}
Let $[a,b]\cong [c,d]$ and $a\nsim b$. For every sound $x$ between-pitched $a$ and $b$, there exists only one sound $y$ up to identity between-pitched $c$ and $d$ so that $[a,x]\cong [c,y]$.
\end{lemma}
\begin{proof}
We first assume that $a*b$ for the sake of simplicity. For every sound $x$ with the property $a*x*b$, there exists a unique sound $y$ up to identity such that $[a,x]\cong [c,y]$ (Axiom of motion). It suffices to show that $y*d$. This is done by item 1 of the Interval Subtraction from which it follows that $[x,b]\cong [y,d]$ and the desired result would be achieved. Now, we assume that $b*a$. By hypothesis we have $[b,a]\cong [d,c]$. According to what is proved, for every sound $x$ between-pitched $b$ and $a$, there is only one sound $y$ up to identity which is between-pitched $d$ and $c$ so that $[b,x]\cong [d,y]$. Again by item 1 of the Interval Subtraction, we conclude $[a,x]\cong [c,y]$.
\end{proof}
\begin{proposition}[Generalization of Lemma 3.37]
Let $R\in \{*,*^{-1},\sim \}$ and $[a,b]\cong [c,d]$. For every sound $x$ with the property $xRb$, there exists only one sound $y$ up to identity with the property $yRd$ such that $[a,x]\cong [c,y]$.
\end{proposition}
\begin{proof}[Sketch of Proof]
Regarding the Generalization of Axioms of Regularity, use the Axiom of Motion and the Interval Subtraction.
\end{proof}
\begin{lemma}
Suppose $[a,b]\cong [c,x]$ and $[a,y]\cong [c,d]$.
\begin{enumerate}
\item $x$ and $d$ are identical-pitched if and only if so are $y$ and $b$; namely,
\begin{center}
$x\sim d \Leftrightarrow y\sim b.$
\end{center}
\item $x$ is between-pitched $c$ and $d$ if and only if $b$ is between-pitched $a$ and $y$; more strongly,
\begin{center}
$(a*b*y\Leftrightarrow c*x*d)\wedge (y*b*a\Leftrightarrow d*x*c).$
\end{center}
\end{enumerate}
\end{lemma}
\begin{definition}
Let $[a,b],[c,d]\in \mathbb{I}^{>1}$. We say $[a,b]$ is \emph{less than} $[c,d]$, written $[a,b]<[c,d]$, if there exists a sound $x$ lower in pitch than $d$ such that $[a,b]\cong [c,x]$; equivalently (item 2 of Lemma 3.39), there exists a sound $y$ higher in pitch than $b$ such that $[a,y]\cong [c,d]$. Naively, in this case we also say that $[c,d]$ is \emph{greater than} $[a,b]$ and occasionally write $[c,d]>[a,b]$; that is, the inverse of $<$ is called \emph{greater than} and denoted by $>$.
\end{definition}
\begin{proposition}[Interval Ordering]
For any $[a,b],[c,d],[e,f]\in \mathbb{I}^{>1}$,
\begin{enumerate}
\item only one of the following holds: $[a,b]<[c,d]$, $[a,b]\cong [c,d]$, or $[a,b]>[c,d]$.
\item if $[a,b]<[c,d]$ and $[c,d]\cong [e,f]$, then $[a,b]<[e,f]$.
\item if $[a,b]>[c,d]$ and $[c,d]\cong [e,f]$, then $[a,b]>[e,f]$.
\item if $[a,b]<[c,d]$ and $[c,d]<[e,f]$, then $[a,b]<[e,f]$.
\end{enumerate}
\end{proposition}
\begin{proof}
For item 1, we first show that $[a,b]<[c,d]$ and $[c,d]<[a,b]$ can never hold simultaneously (asymmetry). If this happens, then based on Definition 3.40 there would be sounds $x$ and $y$ so that $[a,b]\cong [c,x]$ and $[c,d]\cong [a,y]$ and so that $x*d$ and $y*b$. Moreover, a sound $z$ exists so that $[x,d]\cong [b,z]$. According to the Generalization of Interval Addition Axiom we get $[a,z]\cong [c,d]$. Thus $[a,z]\cong [a,y]$ which implies that $z\sim y$. But this is not possible because we have $y*b*z$. Returning to the main issue, there exists a sound $\theta$ such that $[a,b]\cong [c,\theta]$. Hence $c*\theta$. Now, if $\theta \sim d$, then $[a,b]\cong [c,d]$ and by virtue of Corollary 2.4 the other two cases do not hold. If $\theta *d$, then $[a,b]<[c,d]$ and likewise the second case does not hold as well as the third one based on what is proved in the first step. If $d*\theta$, then $[a,b]>[c,d]$ and by the same argument as before, the other cases fail to hold. For item 2, there is some $\alpha \in \mathbb{S}$ such that $[a,b]\cong [c,\alpha]$. Since $[c,d]\cong [e,f]$, from Lemma 3.37 it follows that there is some $\beta \in \mathbb{S}$ so that $[c,\alpha]\cong [e,\beta]$ and so that $e*\beta *f$. We deduce that $[a,b]\cong [e,\beta]$ and $\beta *f$, and then $[a,b]<[e,f]$. Item 3 is merely a matter of inspection. In item 4, some $\gamma ,\delta \in \mathbb{S}$ would be found such that $[a,b]\cong [c, \gamma]$ and $[c,d]\cong [e,\delta]$ and such that $\gamma *d$ and $\delta *f$. Also by Lemma 3.37, there is some $\zeta \in \mathbb{S}$ so that $[c,\gamma]\cong [e,\zeta]$ and $e*\zeta *\delta$. Thus, $[a,b]\cong [e,\zeta]$ and $\zeta *\delta * f$ which implies that $[a,b]<[e,f]$.
\end{proof}
\begin{corollary}
$(\mathbb{I}^{>1},<)$ is a strictly ordered set.
\end{corollary}
\begin{definition}
We define the ordering $<$ on $\mathbb{I}^{<1}$ (not to be confused with the same symbol as the ordering on $\mathbb{I}^{>1}$) as follows:
\begin{center}
$\forall [a,b],[c,d]\in \mathbb{I}^{<1}([a,b]<[c,d]\Leftrightarrow [c,d]^{-1}<[a,b]^{-1}).$
\end{center}
One may reformulate this in terms of congruence alike to Definition 3.40.
\end{definition}
\paragraph*{}
Having clarified the reason for naming in Definition 3.9, one can intuitively comprehend the naive extension of the ordering $<$ on whole $\mathbb{I}$. However, this will formally be done later (see Definition 3.66).

\subsection*{Continuity}
\paragraph*{}
It is time to approach measuring the sizes of intervals adjusted to the treatment of physics of sound concerning the notion of frequency. But there is a problem. No guarantee is given that there does not exist any gaps in pitch of our music, or in other words, the sounds in $\mathbb{S}$ are together arranged continuously in pitch. For more clarification carry on with the following.
\begin{definition}
A sound $m$ is called a \emph{midsound} of $[a,b]$ if $[a,m]\cong [m,b]$. Any sound of $\tilde{a}$ and $\tilde{b}$ is called an \emph{endsound} of $[a,b]$.
\end{definition}
\begin{proposition}
Any interval and its conversion have the same midsounds. Every midsound of a non-unit interval is between-pitched any two non-identical endsounds of that interval.
\end{proposition}
\begin{proof}[Sketch of Proof]
Done by the Generalization of Axioms of Regularity.
\end{proof}
\begin{proposition}
Every interval has at most one midsound up to identity.
\end{proposition}
\begin{proof}
Let m and $m'$ be midsounds of $[a,b]$. We then have $[a,m]\cong [m,b]$ and $[a,m']\cong [m',b]$. Item 2 of the Interval Subtraction yields $[m,m']\cong [m',m]$. From Corollary 3.13 it follows that $m\sim m'$.
\end{proof}
\paragraph*{}
Although a sound may be the midsound of infinitely many intervals (see Lemma 5.18), (conversely thinking) there is unfortunately no idea for proving the existence of midsounds in music thus far constructed. We require one axiom of continuity. Alas, we are not endowed with the beautiful notion of \textit{circle} (and other efficient geometric shapes in the framework of the foundations of geometry) because pitch in one-dimensional as was already mentioned, and any attempt to produce such an object in our music system does not disclose new information and will finally fail to work. So we have no choice of continuity except the most bristling one proposed by Dedekind (\cite{G1} 4th ed. p. 134).
\begin{proposition}
Suppose that $\mathbb{S}$ is the union of two nonempty subsets. The following are equivalent:
\begin{enumerate}
\item No sound of one subset is lower-pitched than a sound of the other.
\item No sound of either subset is between-pitched two sounds of the other.
\end{enumerate}
Moreover, in each of the items above, either the intersection of the two subsets is empty or it consists of exactly one sound up to identity.
\end{proposition}
\begin{proof}
We suppose $\mathbb{S}=P\cup Q$ and $P,Q\neq \emptyset$. Deducing item 2 from item 1 is simply done by a reductio ad absurdum. Conversely, we assume item 2. We choose sounds $x\in P$ and $y\in Q$ such that $x\nsim y$. Without loss of generality we may also assume that $x*y$. Now we verify that for every $a\in P$ and every $b\in Q$, either $a*b$ or $a\sim b$. Suppose not. So $b*a$ for some $a\in P$ and $b\in Q$. Applying item 1 of the Pitch Directness Property to $x$ and $b$ leads to a contradiction. Note that by the first part of this argument, it is impossible to replace the word ``one'' by the word ``either'' within item 1. Finally, suppose that $x,y\in P\cap Q$. It follows that $\neg(x*y)$ and $\neg(y*x)$, which means $x\sim y$.
\end{proof}
\begin{remark}
The two distinct set of sounds satisfying each item of Proposition 3.47 enjoy a kind of convexity in the context of segments. Actually, if we redefine the notion of convex set as follows, then each of the foregoing sets will appear in the form of convex subsets of $\mathbb{S}$; ``$F\subseteq \mathbb{S}$ is \emph{convex} if for every $x,y\in F$, any sound between-pitched $x$ and $y$ belongs to $F$.''
In fact, this new definition is weaker than convexity introduced in Definition 3.18. Therefore, in this sense each item of Proposition 3.47 is equivalent to convexity of both subsets except probably one equivalence class under $\sim$.\\
\textbf{Warning.} Do not assume the sets $P$ or $Q$ in the proof of Proposition 3.47 are in the form of initial or final segments respectively. This matter is practically equivalent to continuity axioms. Geometrically speaking, the only thing guaranteed here is that for every $x\in P$ and every $y\in Q$, $P$ contains the initial segment $\overset{\leftarrow}{x}$ and $Q$ contains the final segment $\overset{\rightarrow}{y}$; additionally, $P$ and $Q$ respectively contain no final and initial segments.
\end{remark}
\begin{definition}
Any partition $\{P,Q\}$ of $\mathbb{S}$ satisfying one of the equivalent conditions of Proposition 3.47 is called a \emph{cut} of $\mathbb{S}$.
\end{definition}
\begin{lemma}
Let $\Sigma =\{P,Q\}$ be a partition of $\mathbb{S}$. The following are equivalent:
\begin{enumerate}
\item $\Sigma$ is a cut of $\mathbb{S}$.
\item All sounds lower in pitch than a sound of one element of $\Sigma$ belong to the same element of $\Sigma$; i.e., $\forall s\in \mathbb{S}((s\in P\Rightarrow \overset{\leftarrow}{s}\subseteq P)\vee (s\in Q\Rightarrow \overset{\leftarrow}{s}\subseteq Q))$.
\item All sounds higher in pitch than a sound of one element of $\Sigma$ belong to the same element of $\Sigma$; i.e., $\forall s\in \mathbb{S}((s\in P\Rightarrow \overset{\rightarrow}{s}\subseteq P)\vee (s\in Q\Rightarrow \overset{\rightarrow}{s}\subseteq Q))$.
\end{enumerate}
\end{lemma}
\begin{proof}[Sketch of Proof]
Immediately done by a reductio ad absurdum.
\end{proof}
\begin{proposition}
Let $\Sigma$ be a cut of $\mathbb{S}$. For each sound $\theta$ the following conditions are equivalent:
\begin{enumerate}
\item If $\theta$ is between-pitched two sounds, then the two sounds are in different sets of $\Sigma$.
\item If a sound is between-pitched $\theta$ and another sound, then the two sounds are in the same set of $\Sigma$.
\item No sound of one of the sets of $\Sigma$ is higher in pitch than $\theta$ and no sound of the other set of $\Sigma$ is lower in pitch than $\theta$.
\item Any sound lower-pitched than $\theta$ belongs to one of the sets of $\Sigma$ and any sound higher-pitched than $\theta$ belongs to the other set of $\Sigma$.
\end{enumerate}
\end{proposition}
\begin{proof}[Sketch of Proof]
All implications of the process $1\Rightarrow 2\Rightarrow 3\Rightarrow 4\Rightarrow 1$ are trivial unless $2\Rightarrow 3$ (looking a bit less obvious) that is simply done by assuming $\Sigma =\{P,Q\}$ and that any sound of $P$ is lower-pitched than or identical-pitched with any sound of $Q$ (item 1 of Proposition 3.47).
\end{proof}
\begin{remark}
Let us do a little logic. In Proposition 3.51, the contrapositive of item 3, which is in fact the very same item 4, is describing item 1 in a somewhat more apparent manner. Item 2 in comparison with item 4 is discussing the same matter more precisely; if $\theta *x*y$, then $x$ and $y$ are in one element of the cut $\Sigma$; if $x'*y'*\theta$, then $x'$ and $y'$ are certainly in the other element of $\Sigma$. More surprisingly (from logical standpoint), the last two items are converse to each other; regarding what is hinted in the proof, item 3 states that $P-\tilde{\theta}\subseteq \overset{\leftarrow}{\theta}$ and $Q-\tilde{\theta}\subseteq \overset{\rightarrow}{\theta}$, whereas conversely item 4 states that $\overset{\leftarrow}{\theta}\subseteq P-\tilde{\theta}$ and $\overset{\rightarrow}{\theta}\subseteq Q-\tilde{\theta}$. Thus, having rephrased the first two items as follows, the two ``if'' can be replaced by ``if and only if'':
\begin{enumerate}
\item Two sounds are on opposite sides of $\theta$ if and only if they are in different elements of $\Sigma$.
\item Two sounds are on the same sides of $\theta$ if and only if they are in the same element of $\Sigma$.
\end{enumerate}
This matter automatically provides a necessary and sufficient condition for the phenomenon involved in Proposition 3.51 leading to the following definition.
\end{remark}
\begin{definition}
Let $\Sigma$ be a cut of $\mathbb{S}$. Each sound $\theta$ satisfying one of the equivalent conditions of Proposition 3.51 is said to be a \emph{cut sound} of $\Sigma$.
\end{definition}
\begin{proposition}
For any cut $\Sigma$ of $\mathbb{S}$, a cut sound of $\Sigma$, if it exists, is unique up to identity.
\end{proposition}
\begin{proof}
Let sounds $\theta _1$ and $\theta _2$ with $\theta _1*\theta _2$ satisfy the equivalent conditions of Proposition 3.51. Sounds $x$ and $y$ would be found such that $x*\theta _1*y*\theta _2$. It follows that $x$ and $y$ are in different elements of the cut and in the same element of that simultaneously --- a contradiction.
\end{proof}
\begin{remark}
Let $\Sigma =\{P,Q\}$ be a cut of $\mathbb{S}$. Without loss of generality suppose that no sound of $P$ is higher in pitch than a sound of $Q$. Now, consider the quotient sets $P/\sim$ and $Q/\sim$ in the ordered set $(\mathbb{S}/\sim \,,*')$. Thus, for every $\tilde{x}\in P/\sim$ and every $\tilde{y}\in Q/\sim$, one has $\tilde{x}*'\tilde{y}$. Every element of $P/\sim$ is a lower bound of $Q/\sim$ and any element of $Q/\sim$ is an upper bound of $Q/\sim$. If the supremum of $P/\sim$ or the infimum of $Q/\sim$ exists, then they are equal (both exist) since $(\mathbb{S}/\sim \,,*')$ is dense (item 2 of Axiom 5) and $\{P,Q\}$ constitutes a partition of $\mathbb{S}$. So if $\Theta =\sup P/\sim \,=\inf Q/\sim$ does exist, then it must belong to $P/\sim$ or $Q/\sim$. On the other hand, by the last part of Proposition 3.47, the intersection of the two quotient sets is either empty or singleton. If $(P/\sim )\cap (Q/\sim )=\emptyset$, then only one of the elements $\max P/\sim$ and $\min Q/\sim$ exists and equals $\Theta$, otherwise the unique element of $(P/\sim )\cap (Q/\sim )$ is equal to $\Theta$, in which case
\begin{center}
$\Theta =\sup \dfrac{P}{\sim}=\inf \dfrac{Q}{\sim}=\max \dfrac{P}{\sim}=\min \dfrac{Q}{\sim}.$
\end{center}
Notice any sound of $\Theta$ would be a cut sound of $\Sigma$, but there is no assurance that $\Theta$ exists in general. This is in fact equivalent to Dedekind's Axiom of Continuity which is supplied as follows.
\end{remark}
\begin{axiom}[Dedekind's Continuity Axiom]
Every cut of $\mathbb{S}$ has a cut sound.
\end{axiom}
\begin{remark}
Having collected all matters of continuity thus far discussed, Dedekind's Continuity Axiom states summarily that if $\mathbb{S}$ is partitioned into two subsets $P$ and $Q$ in such a way that no elements of $P$ is higher-pitched than an element of $Q$, then there exists a unique sound $\theta$, up to identity, such that $\theta$ is between-pitched two sounds if and only if the two sounds are not identical-pitched with $\theta$, and one of them belongs to P and the other one belongs to $Q$, namely,
\begin{center}
$P-\tilde{\theta}=\overset{\leftarrow}{\theta},\,\,\,\,\,Q-\tilde{\theta}=\overset{\rightarrow}{\theta}.$
\end{center}
In other words,
\begin{center}
$\forall s\in \mathbb{S}((s*\theta \Leftrightarrow s\in P\wedge s\nsim \theta) \wedge (\theta *s \Leftrightarrow s\in Q \wedge s\nsim \theta )).$
\end{center}
\end{remark}
\paragraph*{}
Dedekind's Continuity Axiom is the only axiom of music theory which essentially involves the notion of a set of sounds. The set-theoretic moral of this axiom is that the supremum and infimum elements mentioned in Remark 3.55 (i.e. $\Theta$) certainly exist(s). From geometric point of view, it states that all sounds of music are continuously, with respect to pitch, organized in $\mathbb{S}$. The music theory viewpoint on Dedekind's Continuity Axiom is the following. We realize by the Sound Separation Property that every sound $\theta$ naively gives a cut of $\mathbb{S}$ each of whose elements contains just one side of $\theta$. In such a manner, $\theta$ separates all sounds that are non-identical with itself. The axiom says that, conversely, each separation of sounds (say cut) is specified
by a unique sound up to identity. Therefore, the converse to the Sound Separation Property holds in this sense.
\begin{proposition}[Existence of Midsound]
Every interval has a midsound.
\end{proposition}
\begin{proof}
For unit intervals there is nothing to prove. Let $[a,b]$ with $a*b$ (Proposition 3.45). We define a cut of $\mathbb{S}$ as follows:
\begin{align*}
P&=\{s\in \mathbb{S}:s*a \vee s\sim a \vee [a,s]<[s,b] \vee [a,s]\cong [s,b]\},\\
Q&=\{s\in \mathbb{S}:s*^{-1}b \vee s\sim b \vee [a,s]>[s,b]\}.
\end{align*}
Speaking summarily, $P$ is the set of all sounds $s$ that $[a,s]$ is less than or congruent to $[s,b]$, and $Q$ is the set of all sounds $t$ that $[a,t]$ is greater than $[t,b]$. Item 1 of the Interval Ordering guarantees $\{P,Q\}$ is a partition of $\mathbb{S}$. No sound of $P$ is lower in pitch than a sound of $Q$; suppose that $x\in P$ and $y\in Q$ and moreover $x$ and $y$ are between-pitched $a$ and $b$. So $[a,x]<[x,b]$ or $[a,x]\cong [x,b]$, and also $[a,y]>[y,b]$. Based on definition, there would be $\alpha ,\beta \in \mathbb{S}$ such that $[a,\alpha]\cong [x,b]$ and $[a,\beta]\cong [y,b]$ and such that $\neg(\alpha *x)$ and $\beta *y$. By item 2 of the Generalization of Interval Subtraction we obtain $[\alpha ,\beta]\cong [y,x]$. Thus, $y*x$ implies that $\beta *\alpha$ which contradicts item 1 of the Axioms of Regularity. Therefore $\{P,Q\}$ is a cut of $\mathbb{S}$ (a more convenient way to prove this is to apply Lemma 3.50; assuming $s\in P$ and $a*x*s*b$, based on the Interval Ordering we get $[a,x]<[a,s]<or\cong [s,b]<[x,b]$ which yields $x\in P$, and hence $\overset{\leftarrow}{s}\subseteq P$). By Dedekind's Continuity Axiom, there exists a sound $\theta$ such that
\begin{center}
$P-\tilde{\theta}=\{s\in \mathbb{S}:s*\theta\}=\overset{\leftarrow}{\theta},\,\,\,\,\,Q-\tilde{\theta}=\{s\in \mathbb{S}:\theta *s\}=\overset{\rightarrow}{\theta}.$
\end{center}
We claim that $[a,\theta]\cong [\theta ,b]$. Suppose that $[a,\theta]<[\theta ,b]$. The case $[a,\theta]>[\theta ,b]$ is similar and left to the reader. We have by definition $[a,\theta]\cong [\theta ,c]$ (1) for some $\theta \in \mathbb{S}$ with $c*b$. Consider a sound $x$ between-pitched $c$ and $b$. By item 1 of the Interval Ordering we have either $[c,x]<[x,b]$, $[c,x]\cong [x,b]$, or $[c,x]>[x,b]$. We assume that one of the first two cases holds (the third one needs a little bit more intervalkeeping). Hence there are some $z\in \mathbb{S}$ so that $[c,x]\cong [x,z]$ (2) and $\neg (b*z)$. Based on the Axiom of Motion, some $y\in \mathbb{S}$ would be found such that $[\theta ,y]\cong [c,x]$ (3). Based on item 1 of the Generalization of Interval Addition Axiom, from (1) and (3) it follows that $[a,y]\cong [\theta ,x]$ (4). On the other hand, from (2) and (3) (by transitivity) it follows that $[\theta ,y]\cong [x,z]=[z,x]^{-1}$. Also $[y,x]\cong [x,y]^{-1}$. From item 2 of the Generalization of Interval Addition Axiom we deduce that $[\theta ,x]\cong [z,y]^{-1}=[y,z]$. From (4) (again by transitivity) it follows that $[a,y]\cong [y,z]$. Since either $z*b$ or $z\sim b$, it follows that either $[a,y]<[y,b]$ or $[a,y]\cong [y,b]$ which implies that $y\in P$. But $y$ is higher in pitch than $\theta$ (why?) and then it must belong to $Q$. Therefore, $y\in P\cap Q$ --- a contradiction.
\end{proof}
\paragraph*{}
The following kind of definition is a direct usage of the recursion theorem of set thory which is customary to introduce recursively some abstract entities in mathematics.
\begin{definition}
For each interval $[a,b]$ and every natural number $n$, we define the \emph{product} $n\cdot[a,b]$ by induction as follows; we define $1\cdot[a,b]$ to be $[a,b]$ and assuming $n\cdot[a,b]=[a,b_n]$, we define $(n+1)\cdot[a,b]$ to be $[a,b_{n+1}]$ such that $[b_n,b_{n+1}]\cong [a,b]$. We also define $0\cdot[a,b]$ to be $[a,a]$ and for every negative integer $n$, we extend the notion of product of a natural number with an interval as $n\cdot[a,b]=(-n)\cdot[a,b]^{-1}$. We define $(1/2)\cdot[a,b]$ to be $[a,m]$, where $m$ is a midsound of $[a,b]$. By induction, for every $k\in \mathbb{N}$, $(1/2^{k+1})\cdot[a,b]$ is defined as $(1/2)\cdot((1/2^k)\cdot[a,b])$.
\end{definition}
\begin{remark}[Notation]
The set of all dyadic numbers is denoted by $\mathbb{D}$; i.e., $\mathbb{D}:=\{n/2^k: n\in \mathbb{Z}, k\in \mathbb{N}\}$.
\end{remark}
\begin{proposition}
For every interval $[a,b]$, any integers $m, n$, and any natural numbers $k, l$, we have the following:
\begin{enumerate}
\item $m\cdot(n\cdot[a,b])=n\cdot(m\cdot[a,b])=(mn)\cdot[a,b].$
\item $\frac{1}{2^k}\cdot(\frac{1}{2^l}\cdot[a,b])=\frac{1}{2^l}\cdot(\frac{1}{2^k}\cdot[a,b])= \frac{1}{2^{k+l}}\cdot[a,b].$
\item $2^k\cdot(\frac{1}{2^k}\cdot[a,b])=\frac{1}{2^k}\cdot(2^k\cdot[a,b])=[a,b].$
\item $m\cdot(\frac{1}{2^k}\cdot[a,b])=\frac{1}{2^k}\cdot(m\cdot[a,b]).$
\item If $a\sim b$, then $m\cdot[a,b]=(1/2^k)\cdot[a,b]=[a,b]$.
\item If $a*b$ and $m$ and $n$ have the same sign, then
\begin{center}
$m<n\Leftrightarrow m\cdot[a,b]<n\cdot[a,b],\,\,\,\,\,k<l\Leftrightarrow \dfrac{1}{2^l}\cdot[a,b]<\dfrac{1}{2^k}\cdot[a,b].$
\end{center}
\item If $b*a$ and $m$ and $n$ have the same sign, then
\begin{center}
$m<n\Leftrightarrow n\cdot[a,b]<m\cdot[a,b],\,\,\,\,\,k<l\Leftrightarrow \dfrac{1}{2^k}\cdot[a,b]<\dfrac{1}{2^l}\cdot[a,b].$
\end{center}
\end{enumerate}
\end{proposition}
\begin{corollary}
The product $\cdot:\mathbb{D}\times \mathbb{I}\rightarrow \mathbb{I}$ defined by
\begin{center}
$\dfrac{n}{2^k}\cdot[a,b]=n\cdot(\dfrac{1}{2^k}\cdot[a,b])$
\end{center}
is well-defined.
\end{corollary}
\begin{proposition}
For every interval $[a,b]$ and for any positive dyadic numbers $w$ and $v$,
\begin{enumerate}
\item $w\cdot(v\cdot[a,b])=v\cdot(w\cdot[a,b])=(wv)\cdot[a,b]$.
\item if $a\sim b$, then $w\cdot[a,b]=[a,b]$.
\item if $a*b$, then $w<v$ if and only if $w\cdot[a,b]<v\cdot[a,b]$.
\item if $b*a$, then $w<v$ if and only if $v\cdot[a,b]<w\cdot[a,b]$.
\end{enumerate}
\end{proposition}
\begin{proposition}[Archimedean Property of Intervals]
If $[a,b]$ and $[c,d]$ are any intervals greater than unit, then there exists a natural number $n$ such that either $[a,b]\cong n\cdot[c,d]$ or $[a,b]<n\cdot[c,d]$.
\end{proposition}
\begin{proof}
Let $[a,b]$ and $[c,d]$ be greater than unit. Set
\begin{center}
$P=\{b'\in \mathbb{S}:\exists n\in \mathbb{N}(\exists s\in \mathbb{S}([a,s]\cong n\cdot[c,d]\wedge (b'*s\vee b'\sim s)))\}\cup \tilde{a}\cup \overset{\leftarrow}{a}.$
\end{center}
Thus, $P$ consists of all sounds lower-pitched than or identical-pitched with $a$ in addition to all sounds $b'$ so that the interval $[a,b']$ is less than or congruent to $n\cdot[c,d]$ for some natural number $n$. Now set $Q=P^c$ (the complement of $P$). If we prove that $Q=\emptyset$, we may in particular conclude that $b\in P$ and then we are done. So we assume the contrary. Since $\{P,Q\}$ constructs a partition of $\mathbb{S}$, we use item 2 of Lemma 3.50 to show that $\{P,Q\}$ is a cut of $\mathbb{S}$. If $p$ is any sound of $P$ and $x$ is lower in pitch than $p$, then obviously $x$ belongs to $P$ as well (the same $n\in \mathbb{N}$ and $s\in \mathbb{S}$ do work). Now by Dedekind's Continuity Axiom, a cut sound $\theta$ of $\{P,Q\}$ exists. Clearly, $P-\tilde{\theta}$ and $Q-\tilde{\theta}$ respectively consist of all sounds lower-pitched and higher-pitched than $\theta$ because $P$ contains an initial segment. There are only two cases. If $\theta \in P$, then by the definition of $P$ there are $n\in \mathbb{N}$ and $s\in \mathbb{S}$ such that $[a,s]\cong n\cdot[c,d]$ and $\neg(s*\theta)$. By the Axiom of Motion we obtain some $s'\in \mathbb{S}$ so that $[a,s']\cong (n+1)\cdot[c,d]$. Since $[s,s']\cong [c,d]$ by the Interval Subtraction, it follows that $s*s'$ (item 1 of the Axioms of Regularity). Thus $\theta *s'$, which implies $s'\in Q$, whereas $s'\in P$ by definition ($s'\sim s'$) --- a contradiction. If $\theta \in Q$, we pick a sound $t$ so that $[c,d]\cong [t,\theta]$ (by motion). So $t*\theta$ and in turn $t\in P$. Therefore, $m\in \mathbb{N}$ and $u\in \mathbb{S}$ would be found such that $[a,u]\cong m\cdot[c,d]$ and $\neg(u*t)$. On the other hand, there is a sound $u'$ so that $[a,u']\cong (m+1)\cdot[c,d]$. From the Interval Subtraction we deduce that $[u,u']\cong [c,d]$. By transitivity it follows that $[u,u']\cong [t,\theta]$. Based on Corollary 3.34, $\neg(u*t)$ yields $\neg(u'*\theta)$. Hence, $\theta \in P$ by the definition of $P$ --- a contradiction.
\end{proof}
\paragraph*{}
The following lemma essentially expresses the denseness of $\mathbb{D}\cdot \mathbb{I}$ in $\mathbb{I}$. For the sake of simplicity, an analogous proof will be supplied in Lemma 3.72, when we are subsequently equipped with the arithmetic of free intervals.
\begin{lemma}
Let $[a,b]$ and $[c,d]$ be greater than unit.
\begin{enumerate}
\item There is a natural number $n$ such that
\begin{center}
$(n-1)\cdot[c,d]<[a,b]\wedge ([a,b]<n\cdot[c,d]\vee [a,b]\cong n\cdot[c,d]).$
\end{center}
\item There is an integer $k$ such that
\begin{center}
$\dfrac{1}{2^k}\cdot[c,d]<[a,b]\wedge ([a,b]<\dfrac{1}{2^{k-1}}\cdot[c,d]\vee [a,b]\cong \dfrac{1}{2^{k-1}}\cdot[c,d]).$
\end{center}
\item If $[a,b]<[c,d]$, then for each non-unit interval $[x,y]$ there is a dyadic number $w$ such that $[a,b]<w\cdot [x,y]$ and $w\cdot [x,y]<[c,d]$.
\end{enumerate}
\end{lemma}
\begin{remark}
One can naturally generalize the Archimedean Property of Intervals in the following way; for every pair of non-unit intervals $[a,b]$ and $[c,d]$, there exists a nonzero integer $n$ such that $[a,b]<n\cdot[c,d]$.\\
Notice the relation $<$ is here considered between two intervals either greater than unit or less than unit which might intuitively be generalized for any two arbitrary non-unit intervals. Also, a non-uint analogue of Lemma 3.64 will automatically be satisfied, mutatis mutandis.
\end{remark}
\paragraph*{}
In order to more conveniently define a measure of intervals in a modern way (not to be confused with its intellectual parallel for integration in the context of measure theory \cite{R3}), we continue with a new technical concept supplied as follows instead of individual intervals, mainly because employing mathematical instruments is indeed the only recourse for conversational economy.
\begin{definition}
Every equivalence class of $\mathbb{I}/\cong$, equivalent to the set of all intervals congruent to a given interval $[a,b]$, is called a \emph{free interval}, denoted $\overline{[a,b]}$ (it is ordinarily designated by upper-case letters). We say that $\overline{[a,b]}$ is \emph{less than} $\overline{[c,d]}$, written $\overline{[a,b]}\prec \overline{[c,d]}$, whenever one of the following possibilities holds:
\begin{enumerate}
\item If $[c,d]$ is greater than unit, and $[a,b]$ is either less than unit, or unit, or else less than $[a,b]$.
\item If $[c,d]$ is unit and $[a,b]$ is less than unit.
\item If both $[a,b]$ and $[c,d]$ are less than unit and $[a,b]<[c,d]$.
\end{enumerate}
We say that $\overline{[a,b]}$ is \emph{less than or equal to} $\overline{[c,d]}$, written $\overline{[a,b]}\preceq \overline{[c,d]}$, if we have $\overline{[a,b]}\prec \overline{[c,d]}$ or $[a,b]\cong [c,d]$.
\end{definition}
\begin{corollary}
\begin{enumerate}
\item The relations $\prec$ is a well-defined strict ordering on $\mathbb{I}/\cong$.
\item $(\mathbb{I}/\cong \,,\preceq)$ is a linearly ordered set.
\end{enumerate}
\end{corollary}
\begin{proof}[Sketch of Proof]
Based on the Interval Ordering.
\end{proof}
\paragraph*{}
Having been inspired by the Interval Addition Axiom, one may define an operation on $\mathbb{I}$ as follows.
\begin{definition}
We define the \emph{sum} of free intervals $\overline{[a,b]}$ and $\overline{[c,d]}$, denoted $\overline{[a,b]}+\overline{[c,d]}$, to be the free interval $\overline{[x,y]}$ if there exists a sound $z$ such that $[x,z]\cong [a,b]$ and $[z,y]\cong [c,d]$.
\end{definition}
\begin{remark}
Having noticed from the algebraic point of view (read \cite{H6}), one can check that $+:(\mathbb{I}/\cong)\times(\mathbb{I}/\cong)\rightarrow \mathbb{I}/\cong$ is a binary operation on $\mathbb{I}/\cong$. Moreover, $(\mathbb{I}/\cong,+)$ is a commutative (or abelian) group whose unique identity element is the class of unit intervals and the inverse element of every $\overline{[a,b]}$ is the class of all intervals congruent to the conversion of $[a,b]$, i.e. $\overline{[b,a]}$. As expected, $n\cdot\overline{[a,b]}=\overline{n\cdot[a,b]}$ for any integer $n$. According to Corollary 2.18, the order of $(\mathbb{I}/\cong,+)$ is at least countable. On the other hand, $(\mathbb{D},+,.)$ is a commutative ring with identity $1$ (and zero element $0$) which is also an integral domain but not a division ring (and a fortiori a field) because the set of all invertible elements is $\{\pm2^k:k\in \mathbb{Z}\}$ strictly contained in $\mathbb{D}-\{0\}$. Now, we define the function $\cdot:\mathbb{D}\times (\mathbb{I}/\cong)\rightarrow \mathbb{I}/\cong$ by $w\cdot \overline{[a,b]}=\overline{w\cdot [a,b]}$. One can check that $(\mathbb{I}/\cong \,,+,\cdot)$ is a unitary left (and right) $\mathbb{D}$-module but not a vector space. We shall prove this $\mathbb{D}$-module is isomorphic to $(\mathbb{R}^+,\cdot)$ as an $\mathbb{R}^+$-module. As a result, the cardinality of $\mathbb{I}/\cong$ will have been equal to $2^{\aleph_0}$ (Proposition 3.81).
\end{remark}
Based on Definition 3.66, we have $\overline{[a,b]}+\overline{[b,c]}=\overline{[a,c]}$ as a correct equation for any sounds $a$, $b$, and $c$, which actually exhibits the Generalization of Interval addition Axiom. Although we are deprived of such an interesting fact known as the \textit{triangle inequality} in this context (because of one-dimensionality of pitch), the following propositions establish basic relationships between the order $\prec$ and the operation $+$ making proper provisions for producing a nice idea of measuring musical intervals.
\begin{proposition}
For any positive dyadic numbers $w$ and $v$ and for any $A$, $B$ and $C$ that $0\prec A$ and $0\prec B$,
\begin{enumerate}
\item $w<v$ iff $w\cdot A\prec v\cdot A$.
\item $A\prec B$ iff $w\cdot A\prec w\cdot B$ iff $A+C\prec B+C$.
\end{enumerate}
\end{proposition}
\begin{proposition}
For any $A, B, C, D\in \mathbb{I}/\cong$,
\begin{enumerate}
\item $A\prec B$ iff there is an interval $[a,b]$ greater than unit so that $A+\overline{[a,b]}=B$.
\item if $A\prec B$ and $C\prec D$, then $A+C\prec B+D$.
\end{enumerate}
\end{proposition}
\begin{lemma}
Let $A$ and $B$ be any free intervals with $0\prec A$ and $0\prec B$.
\begin{enumerate}
\item There is a natural number $n$ such that
\begin{center}
$(n-1)\cdot B\prec A\preceq n\cdot B.$
\end{center}
\item There is an integer $k$ such that
\begin{center}
$\dfrac{1}{2^k}\cdot B\prec A\preceq \dfrac{1}{2^{k-1}}\cdot B.$
\end{center}
\item If $A\prec B$, then for each nonzero free interval $X$ there is a dyadic number $w$ such that
\begin{center}
$A\prec w\cdot X\prec B$.
\end{center}
\end{enumerate}
\end{lemma}
\begin{proof}
According to the well-ordering principle of $(\mathbb{N},\leq)$ and the Archimedean Property of Intervals, item 1 is done (take the minimum element of the set of all natural numbers $n$ so that $A\preceq n\cdot B$). For item 2, if $A=B$, then $k=0$ does the job. Otherwise, by the first item there is a natural number $n$ such that $(n-1)\cdot B\prec A\preceq n\cdot B$. We take $k=\min\{k\in \mathbb{Z}:n\leq 2^k\}$. If $n>1$, by item 1 of Proposition 3.70 we obtain $2^{k-1}\cdot B\prec A\preceq 2^k\cdot B$. So the desired integer is $1-k$ in this case. If $n=1$ (equivalently, $A\preceq B$), we may get a natural number $m$ so that $(m-1)\cdot A\preceq B\prec m\cdot A$ (a modification of item 1). We then take $k'=\min\{k'\in \mathbb{Z}:m\leq 2^{k'}\}$. By item 1 of Proposition 3.70, from $B\prec m\cdot A$ it follows that $(1/2^{k'})\cdot B\prec A$, and since $m>1$, we obtain $2^{k'-1}\cdot A\preceq (m-1)\cdot A\preceq B$ from which it follows that $A\preceq (1/2^{k'-1})\cdot B$. Hence $k'$ does work in this case. For item 3, we first assume that $0\prec X$ and consider the free interval $B-A=B+(-A)$. Since $0\prec B-A$ (item 1 of Proposition 3.71), from the second item it follows that there is an integer $k$ such that $(1/2^k)\cdot X\prec B-A$. Also, by the first item there is a natural number $n$ such that $(n-1)\cdot ((1/2^k)\cdot X)\prec A\preceq n\cdot ((1/2^k)\cdot X)$. Finally, by item 2 of Proposition 3.71 we obtain
\begin{center}
$A\preceq n\cdot (\dfrac{1}{2^k}\cdot X)=(n-1)\cdot (\dfrac{1}{2^k}\cdot X)+\dfrac{1}{2^k}\cdot X\prec A+(B-A)=B.$
\end{center}
So the dyadic number we are seeking is $w=n/2^k$. If the assumption $X\prec 0$ holds, this same argument gives a dyadic number $w$ for $0\prec -X$ and $-w$ does the job obviously.
\end{proof}
\paragraph*{}
Now, in order to introduce the notion of \textit{measure}, for the sake of simplicity we shall deal with positive free intervals, i.e. free intervals whose representatives are greater than unit (see item 1 of Proposition 3.71), and start with $\mathbb{I}^{>1}/\cong$ as a semigroup.
\begin{definition}
By a \emph{semi-measure of free intervals} we mean any function $\phi$ from $\mathbb{I}^{>1}/\cong$ to $\mathbb{R}^{>1}$ which satisfies the following property:
\begin{center}
$\forall A,B\in \dfrac{\mathbb{I}^{>1}}{\cong}(\phi(A+B)=\phi(A)\phi(B)).$
\end{center}
\end{definition}
\begin{proposition}
If $\phi$ is a semi-measure of free intervals, then for any free intervals $A$ and $B$ and for every positive dyadic number $w$,
\begin{enumerate}
\item $A\prec B$ if and only if $\phi(A)<\phi(B)$.
\item $\phi(w\cdot A)=\phi(A)^w.$
\end{enumerate}
\end{proposition}
\begin{proof}[Sketch of Proof]
Item 1 is straightforward according to item 1 of Proposition 3.71 and Corollary 3.67. Item 2 follows by induction and according to the definition of $\phi$.
\end{proof}
\begin{corollary}
Every semi-measure of free intervals is one-to-one.
\end{corollary}
\begin{proposition}
Every semi-measure of free intervals is onto.
\end{proposition}
\begin{proof}
Let $\phi:\mathbb{I}^{>1}/\cong \,\rightarrow \mathbb{R}^{>1}$ be a semi-measure of free intervals. We first note that for each $r\in \mathbb{R}^{>1}$ there are $A,B\in \mathbb{I}^{>1}/\cong$ such that $\phi(A)>r$ and $\phi(B)<r$. To prove this, suppose that for some $r$ and for each $A$ we have $\phi(A)<r$. Since $\lim_{n\to \infty}\phi(A)^n=\infty$, it follows that there is a natural number $n$ so that $\phi(A)^n>r$. We conclude $\phi(n\cdot A)>r$ --- a contradiction. In a similar manner, the assumption $\phi(A)>r$ cannot hold ($\lim_{k\to \infty}\phi(A)^{2^{-k}}=1$). Returning to the main problem, let $r\in \mathbb{R}^{>1}$ be given, fix a sound $x$, and set
\begin{center}
$P=\{y\in \mathbb{S}:x*y\wedge \phi(\overline{[x,y]})<r\}\cup \tilde{x}\cup \overset{\leftarrow}{x},\,\,\,\,\,Q=P^c.$
\end{center}
Based on the first step of the proof, $\{P,Q\}$ partitions $\mathbb{S}$. If $y\in Q$ and $y*z$, then $\phi(\overline{[x,z]})=\phi(\overline{[x,y]})\phi(\overline{[y,z]})>r$, which implies $z\in Q$. Regarding item 3 of Lemma 3.50, $\{P,Q\}$ is a cut of $\mathbb{S}$. Applying Axiom 11 leads to the existence of a cut sound $\theta$. Clearly no sound of $P$ ($Q$) is higher (lower) in pitch than $\theta$. We claim that $\phi(\overline{[x,\theta]})=r$. Once this is done, a contradiction is derived completing the proof. If $\phi(\overline{[x,\theta]})<r$, then according to the first step of the proof, there is a free interval in $\mathbb{I}^{>1}/\cong$, say $\overline{[\theta,a]}$ (for some $a\in \mathbb{S}$ higher in pitch than $\theta$), such that $\phi(\overline{[\theta,a]})<r/\phi(\overline{[x,\theta]})$. We then obtain $\phi(\overline{[x,a]})<r$ by the definition of $\phi$, which implies $a\in P$, whereas $a$  belongs to $Q$ --- a contradiction. Likewise, one can show that $\phi(\overline{[x,\theta]})>r$ cannot happen.
\end{proof}
\begin{proposition}
If $\phi$ is any semi-measure of free intervals, then every function of the form $\phi ^r$ (defined as $\phi ^r(A)=(\phi(A))^r$), where $r$ is any positive real number, is a semi-measure of free intervals.
\end{proposition}
\begin{proposition}
If $\phi_0$ is any semi-measure of free intervals, then for every semi-measure $\phi$ of free intervals there is a positive real number $r$ so that $\phi=\phi_0^r$.
\end{proposition}
\begin{proof}
We fix a positive free interval $X$ and let $r=\log _{\phi_0(X)}\phi(X)>0$ (1). We claim that $\phi=\phi_0^r$. In order to prove this claim, we conversely suppose that $\phi(A)<\phi_0^r(A)$ for some $A\in \mathbb{I}^{>1}/\cong$. We may thus choose a natural number $k$ large enough such that $\phi(A)\phi(X)^{2^{-k}}<\phi_0^r(A)$ (2). Now we take $B=A+(1/2^k)\cdot X$ (3). From item 1 of Proposition 3.71 it follows that $A\prec B$. By item 3 of Lemma 3.72, there is a dyadic number $w$ so that $A\prec w\cdot X\prec B$ (4). Since $w>0$ (why?), based on Proposition 3.74 and the definition of semi-measure, from formulas (2) to (4), we obtain
\begin{center}
$\phi(X)^w=\phi(w\cdot X)<\phi(B)=\phi(A+(1/2^k)\cdot X)=\phi(A)\phi(X)^{2^{-k}}<\phi_0^r(A)$ (5),
\end{center}
but on the other hand, $\phi_0(X)^w=\phi_0(w\cdot X)>\phi_0(A)$, from which it follows, by (1), that $\phi(X)^w>\phi_0^r(A)$ contradicting formula (5). By a similar argument, the inequality $\phi(A)>\phi_0^r(A)$ leads to contradiction too.
\end{proof}
\begin{proposition}
For every fixed $X\in \mathbb{I}^{>1}/\cong$ and every $x\in \mathbb{R}^{>1}$, there exists just one semi-measure $\phi$ of free intervals such that $\phi(X)=x$.
\end{proposition}
\begin{proof}
From Proposition 3.78 it follows immediately that there exists at most one semi-measure $\phi$ of free intervals correlating the number $x$ with the free interval $X$. It therefore remains to construct such a semi-measure $\phi$. For this purpose, let $A$ be any positive free interval. Set $\alpha =\sup\{w\in \mathbb{D}:w\cdot X\prec A\}$ and define $\phi(A)=x^\alpha$. The first two items of Lemma 3.72 respectively ensure that the set of all dyadic numbers $w$ with the property $w\cdot X\prec A$ is bounded above and nonempty. So the supremum exists in $\mathbb{R}^+$, as known from the theory of real numbers. Thus, $\phi$ is well-defined. Now, consider any two positive free intervals $A$ and $B$ and let $\alpha =\sup\{w\in \mathbb{D}:w\cdot X\prec A\}$ and $\beta =\sup\{w\in \mathbb{D}:w\cdot X\prec B\}$. Suppose that $\phi(A+B)=x^\gamma$. We must prove that $\alpha +\beta =\gamma$. For any $w,v\in \mathbb{D}$ with $w\cdot X\prec A$ and $v\cdot X\prec B$, by item 2 of Proposition 3.71 and Remark 3.69 we get $(w+v)\cdot X\prec A+B$. Thus $w+v\leq \gamma$. Taking the supremum over $w$ and $v$, we obtain $\alpha +\beta \leq \gamma$. Now suppose that $\alpha +\beta <\gamma$. Since $\mathbb{D}$ is dense in $\mathbb{R}$, there is a dyadic number $w$ so that $\alpha +\beta <w<\gamma$. From the arithmetic of real numbers it is known that the dyadic number $w$ (with the property $\alpha +\beta <w$) can be represented as the sum of two dyadic numbers, $w=w_1+w_2$, such that $\alpha <w_1$ and $\beta <w_2$ (consider two decreasing sequence of dyadic numbers converging to $\alpha$ and $\beta$). From $w<\gamma$ it follows that there is a dyadic number $w'$ so that $w<w'$ and $w'\cdot X\prec A+B$. By item 1 of Proposition 3.70 we obtain $w\cdot X\prec A+B$ (1). On the other hand, from $\alpha <w_1$ it follows that $A\preceq w_1\cdot X$, and  from $\beta <w_2$ it follows that $B\preceq w_2\cdot X$. We then conclude that $A+B\preceq w\cdot X$ which contradicts formula (1). Therefore the function $\phi$ is a semi-measure of free intervals. It remains to show that $\phi(X)=x$; equivalently, $\lambda =1$ where $\lambda =\sup\{w\in \mathbb{D}:w\cdot X\prec X\}$. By item 1 of Proposition 3.70 we have got $\lambda \leq 1$ at once. Assuming $\lambda <1$, there is a natural number $k$ so that $(1/2^k)<1-\lambda$. But we have $(1-1/2^k)\cdot X\prec X$, which yields $1-1/2^k\leq \lambda$ --- a contradiction.
\end{proof}
\begin{definition}
By a \emph{measure of free intervals} we understand any function $\Phi:\mathbb{I}/\cong \,\rightarrow \mathbb{R}^+$ (correlating with every free interval a positive real number) which satisfies the following two properties:
\begin{center}
$\forall A,B\in \dfrac{\mathbb{I}}{\cong}(\Phi(A+B)=\Phi(A)\Phi(B))\wedge \exists A\in \dfrac{\mathbb{I}}{\cong}(0\prec A\wedge \Phi(A)>1).$
\end{center}
\end{definition}
\begin{proposition}
For every positive free interval $I$ and for every positive real number $r\neq 1$, there exists only one measure $\Phi$ of free intervals satisfying the following properties:
\begin{enumerate}
\item For any free intervals $A$ and $B$, $\Phi(A+B)=\Phi(A)\Phi(B)$.
\item For any free intervals $A$ and $B$, $A\prec B$ if and only if $\Phi(A)<\Phi(B)$.
\item If $r>1$, then $\Phi(I)=r$, otherwise $\Phi(I)=1/r$.

Furthermore,
\item for any free interval $A$ and dyadic number $w$, $\Phi(w\cdot A)=\Phi(A)^w$.
\item $\Phi$ is bijective (one-to-one and onto).
\end{enumerate}
\end{proposition}
\begin{proof}[Sketch of Proof]
Based on Proposition 3.79, there exists just one semi-measure $\phi$ of free intervals such that $\phi(I)=\max\{r,1/r\}$. Extend $\phi$ uniquely over the whole $\mathbb{I}/\cong$ in the following manner:
\begin{equation*}
\Phi(A)=\left\{
{\begin{array}{*{20}{c}}
{\phi(A)}&{}&{,0\prec A}\\
{1}&{}&{,A=0}\\
{1/\phi(A)}&{}&{,A\prec 0}
\end{array}} \right.
\end{equation*}
\end{proof}
\begin{corollary}
The function $\Phi$ existent in Proposition 3.81 is an isomorphism between the two ordered sets $(\mathbb{I}/\cong \,,\prec)$ and $(\mathbb{R}^+,<)$.
\end{corollary}
\begin{theorem}[Interval Measure]
Let $[x_0,y_0]$ be a given interval greater than unit and $\lambda$ be a fixed real number greater than one. There is a unique way of assigning to each interval $[a,b]$ a real number $|[a,b]|$ in such a way that
\begin{enumerate}
\item $|[a,b]|$ is a positive real number and $|[x_0,y_0]|=\lambda$.
\item $|[a,b]|<1$ iff $b*a$, $|[a,b]|=1$ iff $a\sim b$, and $|[a,b]|>1$ iff $a*b$.
\item $|[a,b]|=|[c,d]|$ if and only if $[a,b]\cong [c,d]$.
\item $|[a,b]|<|[c,d]|$ if and only if $\overline{[a,b]}\prec \overline{[c,d]}$.
\item for any $a,b,c\in \mathbb{S}$, $|[a,c]|=|[a,b]|.|[b,c]|$.
\item for any dyadic number $w$, $|w\cdot [a,b]|=|[a,b]|^w$.
\item for any positive real number $r$, there exists a unique interval $[a,b]$ up to congruence such that $|[a,b]|=r$.
\end{enumerate}
\end{theorem}
\begin{proof}[Sketch of Proof]
Define naturally $|[a,b]|=\Phi(\overline{[a,b]})$, where $\Phi$ is the unique measure of free intervals satisfying all items of Proposition 3.81 with the property that $\Phi(\overline{[x_0,y_0]})=\lambda$.
\end{proof}
\begin{definition}
Any function $|\,|:\mathbb{I}\rightarrow \mathbb{R}^+$ as mentioned in Theorem 3.83 is called a \emph{measure of intervals}, and $|[a,b]|$ is called the \emph{measure} of $[a,b]$.
\end{definition}
\begin{remark}
By means of a measure of intervals $|\,|$, one can extend the order $<$ to the whole $\mathbb{I}$ as follows:
\begin{center}
$[a,b]<[c,d]\Longleftrightarrow |[a,b]|<|[c,d]|$.
\end{center}
Also $[a,b]\leq [c,d]$ iff either $[a,b]<[c,d]$ or $[a,b]\cong [c,d]$. Thus the relation $\leq$ is a total ordering on the set of all intervals.
\end{remark}
\paragraph*{}
About the interval $[x_0,y_0]$ and the number $\lambda$ fixed in Theorem 3.83, it is worth noting that we have a great plan to identify these indeterminates in the future, and the reader is proposed to have patience. At this point, we assume they are given once and for all and the function $|\,|$ is then the unique measure of intervals in the axiom system.
\paragraph*{}
If we review what we did in this subsection, we realize that by assuming continuity in pitch, say Dedekind's Continuity Axiom, we have proved continuity of congruence, say Archimedean Property of Intervals, on the base of which we deduced continuity of intervals, say Interval measure, and found a way to measure them. We think this progress is more in accordance with musical intuition, though the first and the last continuations are logically equivalent and stronger than the middle one (for surjectivity of measures of intervals Dedekind's Continuity Axiom is essentially required (Proposition 3.76)). The following theorem is explicitly concerned with continuity in pitch in modern terminology matching the physics of sound.
\begin{theorem}
Given an arbitrary sound $\theta$ and a positive real number $f_0$. There is a unique way to assign a real number $f(x)$ to each sound $x$ such that the following possibilities occur:
\begin{enumerate}
\item $f(x)$ is positive and $f(\theta)=f_0$.
\item $f(x)=f(y)$ if and only if $x\sim y$.
\item $f(x)<f(y)$ if and only if $x*y$.
\item For any $x,y\in \mathbb{S}$, $|[x,y]|=f(y)/f(x)$.
\item For every positive real number $r$, there exists a unique sound $x$ up to identity such that $f(x)=r$.
\end{enumerate}
\end{theorem}
\begin{proof}[Sketch of Proof]
Define naively $f(x)=f_0\cdot |[\theta ,x]|$.
\end{proof}
\begin{definition}
For any pair of a sound $\theta$ and a positive real number $f_0$, the function $f:\mathbb{S}\rightarrow \mathbb{R}^+$ satisfying all items of Theorem 3.86 is called the \emph{frequency of sounds} in base $(\theta ,f_0)$. The ordered pair $(\theta ,f_0)$ is called the \emph{base} of the frequency of sounds. For every $x\in \mathbb{S}$ the value of the function $f$ at $x$, i.e. $f(x)$, is called the \emph{frequency} of $x$.
\end{definition}
\begin{corollary}
The two ordered sets $(\mathbb{S}/\sim \,,*')$ and $(\mathbb{R}^+,\leq)$ are isomorphic.
\end{corollary}
\begin{proof}[Sketch of Proof]
$fo\Omega:\mathbb{S}\rightarrow \mathbb{R}$ is the desired isomorphism, where $f$ is the frequency of sounds in a base and $\Omega :\mathbb{S}/\sim \,\rightarrow \mathbb{S}$ is a choice function for $\mathbb{S}/\sim$.
\end{proof}
\begin{corollary}
$\mathbb{S}$, $*$, $*'$, and $\mathbb{S}/\sim$ are uncountable. Moreover,
\begin{center}
$Card(\mathbb{S}/\sim)=Card(*')=Card(\mathbb{I})=2^{\aleph_0}.$
\end{center}
\end{corollary}
\paragraph*{}
Similar to measure of intervals, we assume that the base $(x_0,f_0)$ of the unique frequency of sounds $f$ is given once and for all in this axiom system (let $\theta$ be equal to $x_0$).
\paragraph*{}
Altogether, having considered the frequency of sounds is one-to-one up to identity, what Theorem 3.86 states intuitively is that there is no difference between the usual order on the real line and the pitch of our music system.

\section{Loudness and Colour}
\paragraph*{}
In this section, we want to speak about the sounds of a class $\tilde{s}$ for a given sound $s$. For this purpose, we adopt the acoustic approach to the sounds of our music system. From physical point of view sound waves have three main attributes affecting the way they are perceived by the ear \cite{B1}. The first is ``pitch'' measured by the physical term \textit{frequency} which is inversely proportional to the wavelength of sound. The second is ``loudness'' related to the analogous physical quantity \textit{intensity} (or sound pressure) which is proportional to the square of the wave amplitude of sound. The third is ``colour'' or ``timbre'' corresponding to the concept of \textit{spectrum} which represents the quality of the wave form of sound. We have already treated pitch axiomatically to achieve its measurement. We are going to do so as for loudness and color. Notice that pitch, loudness, and timbre are actually subjective terms and are not to be equated with their physical analogues in acoustics \cite{E2}. We have not defined the notion of pitch and we did not have such a plan. We just considered it as an intuitive tool used in axioms to access the corresponding physical quantity. We will adopt this same method to the other two attributes of sounds. (\cite{R1} is recommended for strictly physical study of the characteristics of sound waves.)
\begin{axiom}[Loudness Postulate]
Given any positive number $\iota_0$, there is a unique way to assign a positive real number $\iota(x)$ to each sound $x$ with the property $\iota(x_0)=\iota_0$ and such that for every positive real number $r$ and for every sound $s$ there exists a sound $x$ identical in pitch with $s$ so that $\iota(x)=r$; more precisely, for a given $\iota_0\in \mathbb{R}^+$ there exists a function $\iota:\mathbb{S}\rightarrow \mathbb{R}^+$ such that $\iota(x_0)=\iota_0$, and moreover,
\begin{center}
$\forall r\in \mathbb{R}^+(\forall s\in \mathbb{S}(\exists x\in \mathbb{S}(x\in \tilde{s}\wedge \iota(x)=r))).$
\end{center}
\end{axiom}
\begin{definition}
For every positive number $\iota_0$, the existing function $\iota:\mathbb{S}\rightarrow \mathbb{R}^+$ in the Loudness Postulate is said to be the \emph{intensity of sounds} in base $(x_0,\iota_0)$. The ordered pair $(x_0,\iota_0)$ is called the \emph{base} of $\iota$. For every $x\in \mathbb{S}$ the value of the function $\iota$ at $x$, i.e. $\iota (x)$, is called the \emph{intensity} of $x$.
\end{definition}
\paragraph*{}
We do assume that the number $\iota_0$ in Definition 4.1 is given once and for all.
\begin{remark}
Replacing the first two universal quantifier of the symbolic part of the Loudness Postulate, it is established that $\tilde{s}$ is uncountable for every $s\in \mathbb{S}$, and $\iota |_{\tilde{s}}$ is onto. It particularly implies that the relation $\sim$ does not coincide with $=$. Somewhat surprisingly, there are uncountably many (even greater than $2^{\aleph_0}$) sets of representatives for the equivalence relation of identity, and the relation $\sim$ is uncountable.
\end{remark}
\begin{remark}
We denote the set of all nonconstant periodic continuous functions on $\mathbb{R}$ by $C_p(\mathbb{R})$. The period of a $\Gamma \in C_p(\mathbb{R})$ is denoted by $T_\Gamma$ and is defined to be the minimum positive real number $T$ which satisfies the equation $\Gamma(x+T)=\Gamma(x)$ for every $x\in \mathbb{R}$. We also denote by $A_\Gamma$ the diameter of the range of a $\Gamma \in C_p(\mathbb{R})$. Thus, the elements
\begin{align*}
T_\Gamma &=\min \{T\in \mathbb{R}^+:\forall x\in \mathbb{R}(\Gamma(x+T)=\Gamma(x))\},\\
A_\Gamma &=\max \{|\Gamma(x)-\Gamma(y)|:x,y\in [0,T_\Gamma)\}
\end{align*}
exist and are positive for every $\Gamma \in C_p(\mathbb{R})$. From set theory we know that the cardinality of $C_p(\mathbb{R})$ equals $2^{\aleph_0}$.
\end{remark}
\begin{axiom}[Colour Postulate]
For every sound $x$ there exists a unique nonconstant periodic continuous real-valued function $\Gamma(x)$ such that $T_{\Gamma(x)}=f(x)$ and $A_{\Gamma(x)}=\iota(x)$.
\end{axiom}
\begin{definition}
The function $\Gamma:\mathbb{S}\rightarrow C_p(\mathbb{R})$ with the same properties as in the Colour Postulate is said to be the \emph{timbre}. For every $x\in \mathbb{S}$ the value of $\Gamma$ at $x$, i.e. $\Gamma(x)$, is said to be the \emph{spectrum} of $x$.
\end{definition}
\begin{remark}[Notation]
We denote by $C_p^*(\mathbb{R})$ the range of the timbre; i.e., $C_p^*(\mathbb{R}):=\Gamma (\mathbb{S})$.
\end{remark}
\paragraph*{}
The reason for phrasing the Colour Postulate so is that no really matter whether the function $\Gamma:\mathbb{S}\rightarrow C_p(\mathbb{R})$ can be obtained uniquely or not at the first look. Also, we were not so abstinent to allocate by $\Gamma$ the frequency and the intensity of sounds respectively to the wavelength and the amplitude of their spectra, since the maps $t\mapsto 1/t$ and $t\mapsto t^2$ act bijectively on $\mathbb{R}^+$. The crux of the matter is actually that whether or not $\Gamma$ gives us a simple characterization of the sounds of our music. As expected from the physical point of view, this point is guaranteed by the following axiom.
\begin{axiom}[Axiom of Extensionality]
For every sound $x$ and for every sound $y$, $\Gamma(x)=\Gamma(y)$ implies $x=y$.
\end{axiom}
\paragraph*{}
Physically speaking, the Axiom of Extensionality states that a sound is uniquely determined by its extension to the spectrum. In other words, if two sounds are equated in frequency, intensity, and spectrum (which in fact yields the equality of the first two), they are equal. However, we shall not be going to make use of such a physical identification untill Section 6, when required.
\begin{corollary}
$\Gamma$ is a one-to-one correspondence between $\mathbb{S}$ and $C_p^*(\mathbb{R})$.
\end{corollary}
\begin{corollary}
$\mathbb{S}$, $*$, $\sim$ are equipotent to the continuum $\mathbb{R}$; i.e.,
\begin{center}
$Card(\mathbb{S)}=Card(*)=Card(\sim)=2^{\aleph_0}.$
\end{center}
\end{corollary}
\paragraph*{}
One may understand that the subjective term \emph{timbre} is already defined as follows: $\Gamma(x)=\Gamma(y)$ if and only if $x=y$. Although this conclusion is true, `What is the definition of sound?' one another may ask. This is in fact the philosophy of the axiomatic method! We did suppose sound to be an undefined concept axiomatically equipped with the desired musical properties, just like the concept of \textit{ordered pair} in the context of set theory which is not actually able to be defined in terms of sets in some mathematicians' opinion and for this reason it is assumed as an additional primitive concept besides \textit{set}.
\paragraph*{}
`Did the system really entail a lot of talking about pitch to approach the notion of frequency unlike the other two attributes of sounds?' one could ask. It is worth stating that frequency is the single most important characteristic of sounds from musical point of view whose introduction involves another basic musical notion, namely interval. The lone significance of frequency will be clarified in the next section, when we develop the purely musical framework of our axiom system. For example, having introduced the notion of \textit{melody}, one can practically construct different melodies with fixing all characteristic of sounds other than frequency, whereas this is impossible even by virtue of the set of all sounds with a fixed frequency. However, the most interesting attribute of sounds from practical viewpoint is certainly intensity by which we can produce many attractive dynamics and various rhythms (Section 6). Notice that timbre is the most complicated attribute of sounds from physical viewpoint and actually a cause for the external appearance of different musical instruments that is essentially left to be dealt with in this paper.
\paragraph*{}
We denote the axiom system for music theory constructed so far by $\mathcal{M}(PI)$, and all axioms involved are from now on referred to as the $PI$\textit{-axioms}. The following remark suggests consistency of $\mathcal{M}(PI)$ by presenting an analytically geometric model for it.
\begin{remark}
The Cartesian model of $\mathcal{M}(PI)$ is constructed as follows.  We interpret \emph{sound} as an ordered triple of positive real numbers and consider the relation $*$ to be the usual strict ordering $<$ on the first coordinate of sounds, namely, for any $x=(x_1,x_2,x_3),y=(y_1,y_2,y_3)\in {\mathbb{R}^+}^3$, we have $x*y$ if and only if $x_1<y_1$. We interpret the relation $\cong$ as follows: $[x,y]\cong [z,t]$ whenever $y_1/x_1=t_1/z_1$. We define the function $\iota :\mathbb{S}\rightarrow \mathbb{R}^+$ simply by $\iota(x)=x_2$ (more precisely, multiplied by the scalar $\frac{\iota_0}{(x_0)_2}$), and let $\Gamma(x)$ be the function $g_x:\mathbb{R}\rightarrow \mathbb{R}$ defined by $g_x(t)=\frac{1}{2} x_2\sin (\frac{2\pi}{x_1}t)+x_3$ (corresponding to the physical interpretation of sounds as sinusoidal waves). One can patiently check the correctness of the $PI$-axioms and enjoy the Cartesian interpretation of all other technical concepts through such an elucidation of the musical world.
\end{remark}

\section{Tone Music}
\begin{center}
`\textbf{Let no one ignorant of music enter!}'
\end{center}
\paragraph*{}
The above entrance to this section, alike to its geometric analogue of Plato's academy, means we have already done quite a little music but not all of it. From now on, we go into the world of basic music theory with the aim of introducing key notions including \textit{scale}, \textit{note}, and \textit{melody} to develop musical aspects of our axiom system. Within such a framework, we consider two more undefined concepts ``equitonal'' and ``harmonic'', the first of which is a relation between sounds, written ``$x$ is equitonal to $y$'' and denoted $x\simeq y$, and the second one is a property of intervals, written ``$[a,b]$ is harmonic''. We do not interpret the musical meaning of these two primitive terms for now inasmuch as we intend to maintain the excitement of the story and present a mysterious scenario to clarify the matter while faithfully preserving the spirit of Euclid. Besides, they are open to interpretation from the model theory perspective without missing their own exciting genuine meanings (\cite{T1, P2} and of course \cite{M3}). (Another preferred reference to music theory other than those alluded before is \cite{S1}.)
\begin{axiom}[Euclidean Postulates of Tonality]$\,\,\,\,\,\,\,\,\,$
\begin{enumerate}
\item Every pair of identical sounds are equitonal to each other; i.e.,
\begin{center}
$\forall x,y\in \mathbb{S}(x\sim y\Rightarrow x\simeq y).$
\end{center}
\item Sounds equitonal to the same sound are equitonal to each other; i.e.,
\begin{center}
$\forall x,y,z\in \mathbb{S}((x\simeq z\wedge y\simeq z)\Rightarrow x\simeq y).$
\end{center}
\end{enumerate}
\end{axiom}
\begin{corollary}
The binary relation $\simeq$ is an uncountable equivalence on $\mathbb{S}$.
\end{corollary}
\begin{corollary}
Whenever $[a,b]=[c,d]$, then $a\simeq b$ iff $c\simeq d$. In other words,
\begin{center}
$\forall a,a',b,b'\in \mathbb{S}((a\simeq b\wedge a\sim a' \wedge b\sim b')\Rightarrow a'\simeq b').$
\end{center}
\end{corollary}
\begin{remark}[Notation]
For every $s\in \mathbb{S}$, we denote by $\hat{s}$ the equivalence class of all sounds equitonal to $s$.
\end{remark}
\begin{definition}
For every sound $s$, a minimal element of $\hat{s}\cap \overset{\rightarrow}{s}$ with respect to the order relation $*$, when it exists, is denoted by $s^*$, and $[s,s^*]$ is called an \emph{octave} interval.
\end{definition}
\begin{remark}
We know from set theory that for every sound $s$ a minimal element of the strict ordered set $(\hat{s}\cap \overset{\rightarrow}{s},*)$ is a sound $m\in \hat{s}\cap \overset{\rightarrow}{s}$ which satisfies the following:
\begin{center}
$\forall x\in \mathbb{S}((s\simeq x\wedge s*x)\Rightarrow \neg(x*m)).$
\end{center}
From this symbolic formulation, it becomes clear that the minimal element $s^*$ is unique up to identity, and in turn the octave interval $[s,s^*]$ is well-defined. Notice that $s^*$ (if it exists) is a sound higher in pitch than and equitonal to $s$, and further, there is no sound between-pitched them being equitonal to $s$ or $s^*$.
\end{remark}
\paragraph*{}
With regard to Definition 5.4, a musician partly understands what we are talking about in the present section.
\begin{lemma}
Let $s^*$ exist. For every sound $t$ identical-pitched with $s$, $t^*$ exists.
\end{lemma}
\begin{lemma}
Let $s$ and $t$ be any sounds and let $s^*$ and $t^*$ exist. Then $s\sim t$, iff $s^*\sim t^*$.
\end{lemma}
\begin{proof}
The ``only if'' part is obvious by Remark 5.5. The ``if'' part needs a reductio ad absurdum. Suppose that $s*t$ and $s^*\sim t^*$. Since $t*t^*$, it follow that $t*s^*$. On the other hand, we have $t\simeq t^*$ and $s\simeq s^*$. So the Euclidean Postulates of Tonality imply that $t\simeq s$. Hence we get $s*t*s^*$ and $t\in \hat{s}\cap \overset{\rightarrow}{s}$, which means that $t$ is a sound belonging to the set $\hat{s}\cap \overset{\rightarrow}{s}$ while being lower-pitched than its minimal element $s^*$ --- a contradiction.
\end{proof}
\begin{lemma}
If $[a,b]$ is an octave interval, then $a^*$ exists and is identical-pitched with $b$.
\end{lemma}
\begin{proof}
Suppose that $[a,b]=[s,s^*]$ for some sound $s$. Since $a\sim s$ and $s^*$ exists, it follows that $a^*$ exists (Lemma 5.6) and is identical-pitched with $s^*$ (Lemma 5.7). Finally, $a^*$ is identical-pitched with $b$ because $s^*\sim b$.
\end{proof}
\paragraph*{}
Based on the axiom so far stated about tonality, i.e. the Euclidean Postulates of Tonality, we do logically not observe any difference between the two binary relations $\simeq$ and $\sim$. Is there an octave interval, or at least, is there any pair of non-identical equitonal sounds? Intuitively, we do need the fact that for every sound there is an equitonal sound higher-pitched than that, but this statement will be redundant as a new axiom. We must deeply care about postulating statements involved in the existential quantification. At the other extreme, for all that has been said so far, we might have been operating in a vacuum. Thus, to give the discussion some substance and to prevent reducibility of our axiom set, we officially assume that `there exists an octave interval' at present, based on which we will be better able to hypothesize one basic axiomatic property of tonality entitled \textit{uniformity}. Notice that since later on we shall formulate a logically stronger and (of course) more efficient existential assumption, this one plays a temporary role only.
\begin{remark}
The assumption above is explicitly stated as follows:
\begin{center}
$\exists \vartheta ,\vartheta '\in \mathbb{S}(\vartheta'=\vartheta^*)$.
\end{center}
The octave interval $[\vartheta ,\vartheta^*]$ is then produced. So up to this point, there is at least one octave interval.
\end{remark}
\paragraph*{}
Since the ground of music theory, in practice, contains many octave intervals, a new axiom is required to generate them as instinctively as follows.
\begin{axiom}[Axiom of Uniformity]
Every interval congruent to a given octave interval is an octave interval; that is,
\begin{center}
$\forall s\in \mathbb{S}(\forall I\in \mathbb{I}(I\cong [s,s^*]\Rightarrow \exists t\in \mathbb{S}(I=[t,t^*]))).$
\end{center}
\end{axiom}
\paragraph*{}
The musical reason for the name of Axiom 16 is that an octave interval can move along pitch (by the Axiom of Motion) without any deformation of tonality of its endsounds. The moral is that there are infinitely many (in fact $2^{\aleph_0}$) octave intervals in the system. As expected from physical viewpoint, we shall see the Axiom of Uniformity indirectly tends to declare that tonality of our music is dependent on pitch.
\begin{lemma}[Generalization of Lemma 5.8]
If $[a,b]$ is congruent to an octave interval, then $a^*$ exists and is identical-pitched with $b$. In particular, $[a,b]$ is equal to the octave interval $[a,a^*]$.
\end{lemma}
\begin{corollary}
For all $s\in \mathbb{S}$, $s^*$ exists.
\end{corollary}
\begin{proof}
Using the Axiom of Motion, there is a sound $x$ such that $[s,x]\cong [\vartheta ,\vartheta^*]$ (see Remark 5.9). Now Lemma 5.10 ensures the existence of $s^*$.
\end{proof}
\paragraph*{}
Though $s^*$ is not unique, it leads us to the following effective musical conception whose importance will be attached, having guaranteed it is well-defined by Corollary 5.11 and Lemma 5.7.
\begin{definition}
The function $O^+:\mathbb{S}/\sim \,\rightarrow \mathbb{S}/\sim$ defined by $O^+(\tilde{s})=\tilde{s^*}$ is called the \emph{forward octave function}.
\end{definition}
\begin{exercise}
For any pair of sounds $s$ and $t$, the following are equivalent:
\begin{enumerate}
\item $s\sim t.$
\item $\overset{\rightarrow}{s}=\overset{\rightarrow}{t}.$
\item $\hat{s}=\hat{t}$ and $\overset{\rightarrow}{s}=\overset{\rightarrow}{t}.$
\item $\hat{s}\cap \overset{\rightarrow}{s}=\hat{t}\cap \overset{\rightarrow}{t}.$
\item $s^*\sim t^*.$
\item $O^+(\tilde{s})=O^+(\tilde{t}).$
\end{enumerate}
\end{exercise}
\begin{proof}[Hint]
Follow the direction $1\Rightarrow 2\Rightarrow 3\Rightarrow 4\Rightarrow 5\Rightarrow 6\Rightarrow 1$ and see Lemma 5.7.
\end{proof}
\begin{proposition}[Euclid's Tonality Postulate]
All octave intervals are congruent to each other.
\end{proposition}
\begin{proof}
For any sounds $s$ and $t$, by the Axiom of Motion we have $[s,s^*]\cong [t,\alpha]$ for some sound $\alpha$. Based on Lemma 5.10 we obtain $[t,\alpha]=[t,t^*]$. It follows that $[s,s^*]\cong [t,t^*]$.
\end{proof}
\begin{remark}
Here the independence of the Axiom of Uniformity would be under discussion. Consider the same Cartesian model of $\mathcal{M}(PI)$ as mentioned in Remark 4.8. Let the relation $\simeq$ be the union of the relation $\sim$ and the set $\{(x,y)\in \mathbb{S}\times \mathbb{S}: f(x),f(y)\in \{2,4\}\}$. One can observe that Axiom 15 together with the assumption of Remark 5.9 are satisfied, but the only existing octave interval, i.e. $[2,4]$ (indeed $[(2,0,0),(4,0,0)]$ ), cannot move in the sense that Axiom 16 fails to hold. It is interesting that Euclid's Tonality Postulate is still true in this interpretation. Therefore, in opposition to the grandeur of Euclid's Tonality Postulate from geometric point of view, the Axiom of Uniformity would not be inferred from it.
\end{remark}
\begin{remark}[Notation]
Let us temporarily denote by $k$ the measure of octave intervals. It is clear that $k>1$.
\end{remark}
\begin{remark}
Let $s\in \mathbb{S}$. Applying the operator $^*$ to $s$ gives $s^*$. We then have $|[s,s^*]|=k$. Pursuing the same procedure starting with $s^*$, we obtain a sound $s^{**}$ equitonal to $s^*$ and a fortiori to $s$ such that $|[s^*,s^{**}]|=k$. If we repeat this process ad infinitum, we will achieve arbitrary numbers of sounds higher-pitched than $s$ which are equitonal to that; furthermore, the measure of the resultant intervals produced by sequential sounds in each step are all equal to $k$. Now, `What happens to equitonal sounds lower-pitched than $s$?' one can ask. To answer this question, we continue with the following lemma which is satisfied in $\mathcal{M}(PI)$ but needed here.
\end{remark}
\begin{lemma}
For every sound $m$ and for every sound $a$ there is just one sound $b$ up to identity such that $m$ is a midsound of $[a,b]$.
\end{lemma}
\begin{proof}[Sketch of Proof]
Done by the Axiom of Motion.
\end{proof}
\begin{proposition}
The forward octave function $O^+$ is an isomorphism on $(\mathbb{S}/\sim ,*')$.
\end{proposition}
\begin{proof}
Based on Lemma 5.18 and Proposition 3.45, for every class $\tilde{s}$ we can choose a sound $t$ so that $s$ is a midsound of $[t,s^*]$. It means that $[t,s]\cong [s,s^*]$, which yields $O^+(\tilde{t})=\tilde{t^*}=\tilde{s}$ by Lemma 5.10. Thus $O^+$ is onto. Since $t$ is unique up to identity, $O^+$ is one-to-one (Lemma 5.7 independently shows $O^+$ is injective). Finally, by Euclid's Tonality Postulate, the two intervals $(\tilde{x},O^+(\tilde{x}))$ and $(\tilde{y},O^+(\tilde{y}))$ are congruent for any sounds $x$ and $y$. From Proposition 3.33 it follows that $[x,y]\cong (O^+(\tilde{x}),O^+(\tilde{y}))$. By Corollary 3.34, $\tilde{x}*'\tilde{y}$ if and only if $O^+(\tilde{x})*'O^+(\tilde{y})$.
\end{proof}
\begin{definition}
We define the \emph{backward octave function} $O^-:\mathbb{S}/\sim \,\rightarrow \mathbb{S}/\sim$ to be the inverse of the forward octave function, that is, $O^-=(O^+)^{-1}$.
\end{definition}
\begin{exercise}
Prove that $[x,y]$ is an octave interval if and only if one of the following equivalents occurs:
\begin{enumerate}
\item $x\simeq y$ and $x*y$, and for every sound $s$, $x*s*y$ implies $s\not\simeq x$.
\item $x^*\sim y.$
\item $O^+(\tilde{x})=\tilde{y}.$
\item $O^-(\tilde{y})=\tilde{x}$
\item $|[x,y]|=k.$
\end{enumerate}
\end{exercise}
\begin{remark}
Equivalent to Definition 5.20, $O^-(\tilde{s})$ is equal to the class $\tilde{t}$ of those sounds which make $[t,s]$ be an octave interval, namely $t^*\sim s$. Clearly, the backward octave function $O^-$ is an isomorphism on $(\mathbb{S}/\sim ,*')$ as well as the $O^+$. Going back to Remark 5.17, we may also obtain, by virtue of the backward octave function, infinitely many sounds equitonal to a given sound and lower-pitched than that. The Axiom of Uniformity implies that all equitonal sounds could be obtained in such a way working with the two octave functions. This claim is proven in the following.
\end{remark}
\begin{theorem}[Tonality Representation Theorem]
For every sound $s$, the bisequence $\{{O^+}^n(\tilde{s})\}_{n\in \mathbb{Z}}$ is a partition of $\hat{s}$.
\end{theorem}
\begin{proof}
The proof is based on the principle of mathematical induction. Clearly $(O^+)^n(\tilde{s})\in \mathbb{S}/\sim$ is nonempty for every $n\in \mathbb{Z}$. We notice the bisequence is strictly increasing with respect to $*'$; this is immediate by Proposition 5.19 and Definition 5.12. Hence, if ${O^+}^m(\tilde{s})$ intersects ${O^+}^n(\tilde{s})$ which implies that ${O^+}^m(\tilde{s})={O^+}^n(\tilde{s})$, then $m$ necessarily equals $n$. Finally, we must check the equality $\hat{s}=\bigcup_{n\in \mathbb{Z}} {O^+}^n(\tilde{s})$. Trivially, $\tilde{s}\subseteq \hat{s}$, and for every sound $x$ equitonal to $s$ we have $O^+(\tilde{x})\subseteq \hat{s}$ and $O^-(\tilde{x})\subseteq \hat{s}$ (transitivity). We deduce from induction that ${O^+}^n(\tilde{s})\subseteq \hat{s}$ for every integer $n$. Conversely, suppose that $x\in \hat{s}$. If $x\sim s$, we are done. Otherwise, depending on which side of $s$ contains $x$, by applying the forward or backward octave function, the well-ordering property of natural numbers guarantees the existence of an integer $n$ so that ${O^+}^n(\tilde{s})*'\tilde{x}$ and $\tilde{x}*'{O^+}^{n+1}(\tilde{s})$. Since $({O^+}^n(\tilde{s}),{O^+}^{n+1}(\tilde{s}))$ is an octave interval, it follows that $x$ belongs to either ${O^+}^n(\tilde{s})$ or ${O^+}^{n+1}(\tilde{s})$ (see item 1 of Exercise 5.21) which concludes the proof.
\end{proof}
\paragraph*{}
Having replaced the forward octave function by the backward one, another equivalent representation of tonality can appear in Theorem 5.23.
\begin{corollary}
For every sound $s$, the ordered set $(\hat{s}/\sim \,,*')$ is isomorphic to $(\mathbb{Z},\leq)$.
\end{corollary}
\begin{proof}[Sketch of Proof]
Define naturally the function $\mathbb{Z}\rightarrow \hat{s}/\sim$ as $n\mapsto {O^+}^n(\tilde{s})$.
\end{proof}
\paragraph*{}
The moral is that $(\hat{s}/\sim \,,*')$ inherits all set-theoretic features of $(\mathbb{Z},\leq)$. In particular, the relation $*'$ is a well-ordering on $(\hat{s}\cap \overset{\rightarrow}{s})/\sim$; i.e., every nonempty subset of $(\hat{s}\cap \overset{\rightarrow}{s})/\sim$ has its minimum element. On the other extreme, If we consider $\Omega :\mathbb{S}/\sim \,\rightarrow \mathbb{S}$ to be a choice function for $\mathbb{S}/\sim$, we then understand that the map $n\mapsto \Omega({O^+}^n(\tilde{s}))$ from $\mathbb{Z}$ to $\mathbb{S}$ injects all integers into the class $\hat{s}$ consisting of all sounds equitonal to $s$. Therefore, corresponding to each integer we may obtain a unique sound, up to identity, of different frequency from the others, which is equitonal to $s$. Having made some adjustment of this musical phenomenon, an alternate consequence is obtained as follows.
\begin{corollary}
Given a sound $s$, every family of all pairwise non-identical sounds which are equitonal to $s$ (and thus to each other) is countable.
\end{corollary}
\begin{remark}
Entering the world of dynamical systems (see \cite{H4, B2, A2} for deep study), if we equip the ordered set $(\mathbb{S}/\sim,*')$ with the order topology inherited from $\mathbb{R}^+$ (see Corollary 3.88), we will observe that the dynamics of the forward octave function $O^+:\mathbb{S}/\sim \,\rightarrow \mathbb{S}/\sim$ (as a homeomorphism) is very simple. Since the set of all orbits under $O^+$ makes a partition of $\mathbb{S}/\sim$, this material is compatible with the fact that the set of all classes $\hat{s}$ partitions $\mathbb{S}$ plainly, as mentioned in the Tonality Representation Theorem;
\begin{center}
$\mathbb{S}=\underset{s\in \frac{\mathbb{S}}{\sim}}{\stackrel{\circ}{\bigcup}} \underset{n\in \mathbb{Z}}{\stackrel{\circ}{\bigcup}}{O^+}^n(s).$
\end{center}
One can see that this dynamical system has no periodic point and every class $\tilde{s}$ is wandering under $O^+$. Hence every invariant set is expressed as the union of some orbits, and especially $O^+(\hat{s})=\hat{s}$. Note that the phase space $\mathbb{S}/\sim$ is not compact, thus every minimal set of $(\mathbb{S}/\sim \,,O^+)$ (a nonempty closed invariant set having no proper subset with these three properties) is of the form $\hat{s}$ which is equivalent to the orbit of $s\in \mathbb{S}$ under $O^+$. In fact, this system is topologically conjugate to the system $(\mathbb{R}^+,x\mapsto kx)$ under the conjugacy $fo\Omega$, where $\Omega:\mathbb{S}/\sim \,\rightarrow \mathbb{S}$ is a choice function. Fortunately, we are not on the circle to get involved in chaos theory! Similar result holds for the dynamical system $(\mathbb{S}/\sim \,,O^-)$.
\end{remark}
\begin{theorem}[Tonality Characterization Theorem]
For every pair of sounds $s$ and $t$, $s\simeq t$ if and only if $|[s,t]|=k^n$ for some $n\in \mathbb{Z}$.
\end{theorem}
\begin{proof}
Based on the Tonality Representation Theorem, $t$ is equitonal to $s$ if and only if there is some integer $n$ such that $t\in {O^+}^n(\tilde{s})$. Thus, it suffices to show that the possibility $t\in {O^+}^n(\tilde{s})$ is equivalent to $|[s,t]|=k^n$ for any sounds $s$ and $t$. This is done by induction on $n$. Obviously, $[s,t]$ is unit if and only if $t\in \tilde{s}$. Let the statement $|[s,t]|=k^n$ be equivalent to $t\in {O^+}^n(\tilde{s})$. First suppose that $|[s,t]|=k^{n+1}$. Since $|[s,s^*]|=k$, by item 5 of Theorem 3.83 it means that $|[s^*,t]|=k^n$. By the inductive hypothesis, this means $t\in {O^+}^n(\tilde{s^*})$ which is equivalent to $t\in {O^+}^{n+1}(\tilde{s})$ by the definition of the forward octave function. In a similar way, having assumed $|[s,t]|=k^{n-1}$, we equivalently obtain $|[s,t^*]|=k^n$. This condition holds if and only if $t^*\in {O^+}^n(\tilde{s})$ (the inductive hypothesis) which means $t\in {O^+}^{n-1}(\tilde{s})$ using the backward octave function.
\end{proof}
\paragraph*{}
Therefore, two sounds are equitonal if and only if the ratio of their frequencies equals $k^n$ for some integer $n$. What happened? The subjective concept of tonality has been characterized in terms of the frequency of sounds. One may complain about such an artificial characterization of tonality and ask why it is not defined so as to have the aforementioned property, having economized on undefined concepts. Instead, we did consider an axiomatic alternative converging to the same sensational event as musical recognition leads, just like Euclid's treatment of geometry. This is why we acted in that way. It is a matter of authority to have got to either remember a few more postulates or forget a few more philosophical conceptions.
\paragraph*{}
The following consequence is a kind of generalization of uniformity which indicates instinctively that tonality of our music is invariant under congruence.
\begin{corollary}[Generalization of Corollary 5.2]
Let $[a,b]\cong [a',b']$. Then $a\simeq b$ iff $a'\simeq b'$.
\end{corollary}
\paragraph*{}
A more interesting expression of such a generalization of uniformity is
\begin{corollary}
Let $[a,b]\cong [a',b']$. Then $a\simeq a'$ iff $b\simeq b'$.
\end{corollary}
\begin{proof}[Sketch of Proof]
Based on Proposition 3.33.
\end{proof}
\begin{remark}
This is as good a time as any to specify the constant $k$ on the base of the material treated on tonality in the current section. Recall the interval $[x_0,y_0]$ and the constant $\lambda >1$ in the Interval Measure (Theorem 3.83). We simply take $[x_0,y_0]$ to be an octave interval; equivalently, $y_0=x_0^*$ and $\lambda =2$. So, from now on we have got $k=2$. This would be realistic because from physics point of view the octave refers to a doubling in frequency (\cite{L1}, p. 10). Therefore, working with the frequency of sounds, we trivially obtain $f(\hat{s})=\{2^nf(s):n\in \mathbb{Z}\}$ for every sound $s$.
\end{remark}
\paragraph*{}
To conclude the discussion about tonality, it is useful to give an intuitive description of this notion. We have already known $\simeq$ is an equivalence relation on $\mathbb{S}$. In practice, when we increase the frequency of a sound $\alpha$ by an electronic tuner, the musical quality of resultant sounds, say their tonality, varies till we reach a sound of double frequency, i.e. $2f(\alpha)$. These two sounds have the same musical quality in the sense that they are equitonal to each other. Notice that iterating the process started with the new sound or inversely implementing it (by decreasing the frequency) will create the same musical phenomenon. Therefore, tonality of our music, which exhibits a type of musical similarity among sounds of the system, varies sinusoidally with respect to the frequency in this sense. Having introduced the notion of \textit{note}, we observe that equitonal sounds have the same note name, and roughly speaking, they are musically equivalent. In fact, \textit{tonality} is the first essentially musical fundamental we confront in this axiom system, peerless and without any analogue in the language of any mathematical theory thus far axiomatized.

\subsection*{Harmony}
\paragraph*{}
the concept of harmonic intervals has already been introduced as a primitive term at the beginning of this section. Nevertheless, we have merely taken action with tonality up to this point. This is mainly because we planned to proceed with the least and simplest musical equipments. From now onwards, we are no longer forced into mentioning such a musical caution. It is worth noting another recommended term which is intuitively appropriate to allocate for harmonic intervals is \textit{well-heard} that we are not interested in using at all.
\begin{axiom}[Elementary Principles of Harmony]$\,\,\,\,\,\,\,\,\,\,$
\begin{enumerate}
\item Let $a*b$ and $[a,b]\cong [c,d]$. If $[a,b]$ is a harmonic interval, then so is $[c,d]$.
\item The conversion of any non-unit harmonic interval is a harmonic interval.
\end{enumerate}
\end{axiom}
\begin{remark}[Notation]
We denote by $\mathbb{I}^* (\subseteq \mathbb{I})$ the set of all harmonic intervals and refer to $|\mathbb{I}^*|$ as the range of the restriction of the measure of intervals to the set of harmonic intervals. Also, we use the notations $|\mathbb{I}^{<1*}|$ and $|\mathbb{I}^{>1*}|$ in the similar meaning for harmonic intervals less and greater than unit respectively. Speaking formally,
\begin{center}
$|\mathbb{I}^*|=\{|I|: I\in \mathbb{I}^*\}=|\mathbb{I}^{<1*}| \cup \{|I|: I\in \mathbb{I}^* \wedge |I|=1\}\cup |\mathbb{I}^{>1*}|$.
\end{center}
Notice we do still not understand a unit interval being harmonic based on the Elementary Principles of Harmony, but we will later on. Once this occurs, harmony of music is well-behaved up to congruence (or well-defined up to the measure of intervals) in the following sense;

`if one of two congruent intervals is harmonic, then so is the other one.'\,\,\,\,\,\,\,\,\,\\
Therefore, harmony is a matter of interval measure just like tonality which is a matter of frequency. This property of harmony may be known as the \emph{Invariance under Congruence}. However, thus far in our treatment of music theory we have tried to refrain from applying numbers as the sizes of intervals; this was in keeping with the traditional strategy of abstractly doing music. But for the sake of convenience, we shall not be so austere from now onwards. As a result of Axiom 17,
\begin{center}
$\forall r\in \mathbb{R}^+(r\in |\mathbb{I}^*| \Leftrightarrow 1/r \in |\mathbb{I}^*|)$.
\end{center}
\end{remark}
\begin{axiom}[Operational Principle of Harmony]
The sum of every pair of positive free harmonic intervals and that of every pair of positive and negative free harmonic intervals are free harmonic intervals; more simply,
\begin{center}
$\forall (r,r')\in |\mathbb{I}^*|^2((r>1\wedge r'>1)\vee (r>1\wedge r'<1)\Rightarrow rr' \in |\mathbb{I}^*|)$.
\end{center}
\end{axiom}
\begin{corollary}
\begin{enumerate}
\item The sum of any two free harmonic intervals is a free harmonic interval.
\item $[a,b]$ is harmonic if and only if so is $n\cdot [a,b]$ for all integers $n$.
\item $r\in |\mathbb{I}^*|$ iff $r^n\in |\mathbb{I}^*|$ for every $n\in \mathbb{Z}$.
\end{enumerate}
\end{corollary}
\begin{proof}[Sketch of Proof]
For item 1 the non-trivial part is when the summation acts on two negative free harmonic intervals, in which case item 2 of Axiom 17 does the job. Items 2 and 3 are straightforward.
\end{proof}
\paragraph*{}
Such a kind of formal expression of Axiom 18 has not only been to show higher precision, but rather we like axioms to be the strongest ones involved in the least concise implicational information, sufficiently weak in both predecessor and successor, (as for theorems improved by weakening their hypotheses and strengthening their conclusions) and logic to have the most usage in the discovery of stronger facts on the base of those idiomatically strong axioms. We did so because, generally, ordered pairs play weaker roles than unordered ones while lying in the predecessor part of implications, but they would be logically equivalent under some suitable conditions; e.g., the formula $\forall  \{x,y\} \, P(x,y)$ yields its weaker version $\forall (x,y) \, P(x,y)$, and the equivalence holds whenever $P$ is a symmetric property (the reader should guess what $P$ is in this background (Remark 5.33)). Also, take a notice of the logical point that `the weaker the predecessor or the stronger the successor, the stronger the implication will be'.
\begin{remark}
Having noted never getting something out of nothing, we seem not yet to make available even one harmonic interval. Once such an existence comes into happening, all unit intervals become harmonic. Then we obtain the commutative group $(|\mathbb{I}^*|,.)$ at once, where $.$ is the usual multiplication of real numbers. In fact, $(|\mathbb{I}^*|,.)$ will be a subgroup of $(\mathbb{R}^+,.)$ (isomorphic to $(\mathbb{I}/\cong ,+)$ formerly checked). We will prove this group is isomorphic to $(\mathbb{Z},+)$ as well as expected from diatonic theory point of view.
\end{remark}
The following axiom asserts that the set of all harmonic intervals greater than unit and the set of all positive free intervals are well-ordered with the linear orders $\leq$ and $\preceq$ respectively, as the symbolic formulation reveals, just like the homogeneous incident occurred at tonality, and it does definitely separate $\mathbb{I}^*$ from the whole $\mathbb{I}$ in the meanwhile ($\mathbb{I}^*\subsetneq \mathbb{I}$). From musical perspective, this axiom appoints a practically reasonable limitation on the set of harmonic intervals so that it is one of the most essential and inseparable ingredients of diatonic set theory. Of course, discussing the fundamental reason for the name of the axiom is actually outside the bounds of the paper and the interested reader is suggested to study \cite{J1}.
\begin{axiom}[Axiom of Diatonicism]
There exists an interval greater than unit which is less than or congruent to all harmonic intervals greater than unit; equivalently,
\begin{center}
$\exists \alpha \in \mathbb{R}^{>1}(\forall \beta \in \mathbb{R}^{>1}(\neg (\beta <\alpha \wedge \beta \in |\mathbb{I}^*|)))$.
\end{center}
\end{axiom}
\begin{remark}
First of all, scrutinizing the Axiom of Diatonicism provides no information suggesting whether or not the existent $\alpha$ is harmonic. Let us for a moment assume that there is some harmonic interval greater than unit. We set $\Lambda =|\mathbb{I}^{>1*}|=\{r\in |\mathbb{I}^*|: 1<r\}$. So $\Lambda$ is nonempty and, by the Axiom of Diatonicism, bounded from below. Hence, $\varepsilon =\inf \Lambda$ is well-defined and larger than or equal to $\alpha$ and a fortiori larger than one. We claim the $\varepsilon$ is certainly harmonic. Suppose not. From the analysis of real numbers we know there is a strictly decreasing sequence $\{r_n\}_{n=1} ^\infty$ in $\Lambda$ converging to $\varepsilon$ (by definition of infimum there is some $a_n\in \Lambda$ with $\varepsilon \leq a_n<\varepsilon +1/n$ for every $n\in \mathbb{N}$ and the axiom of choice guarantees the existence of the sequence $\{a_n\}_{n=1} ^\infty$. Clearly, $\lim_{n\to \infty} a_n=\varepsilon$ and $a_n\neq \varepsilon$ for all $n$. Define $b_n=\min \{a_k:1\leq k\leq n\}$ for each $n$. So the sequence $\{b_n\}_{n=1} ^\infty$ is non-increasing and converges to the same $\varepsilon$. Then let $n_1=1$ and $n_k=\min \{i\in \mathbb{N}: b_i<b_{n_{k-1}}\}$ for any $k>1$. In accordance with the well-ordering property of natural numbers such a recursive definition of $n_k$ makes sense and generates the subsequence $\{b_{n_k}\}_{k=1} ^\infty$ with the desired property.). We note that $r_n\neq \varepsilon$ for all $n\in \mathbb{N}$ because of the absurdity ($\varepsilon \notin \Lambda$). Also $\{r_n\}_{n=1} ^\infty$ is a Cauchy sequence because of convergence, so we have
\begin{center}
$\min \{|\frac{r_m}{r_n}-1|,|\frac{r_n}{r_m}-1|\}=|r_m-r_n|\min \{\frac{1}{r_n},\frac{1}{r_m}\}\longrightarrow 0$
\end{center}
as $m,n\to \infty$. Thus we can find some natural numbers $m$ and $n$ large enough so that $n<m$ and $r_m/r_n-1<\varepsilon -1$. In general, one may work with a comfortable metric on $\mathbb{R}^+$, defined by
\begin{center}
$d(x,y)=\min \{|\frac{x}{y}-1|,|\frac{y}{x}-1|\}=|x-y|\min \{\frac{1}{x},\frac{1}{y}\}$.
\end{center}
This new metric, of course, generates the same topology as the Euclidean metric does on $\mathbb{R}^+$ as required, but they are not strongly equivalent (however, these two topologically equivalent metrics enjoy strong equivalence too on the subspace $\mathbb{R}^{>\lambda}$ for arbitrary positive real numbers $\lambda$, especially in this context) \cite{A3}. Moving forward, we have obtained $1<r_m/r_n<\varepsilon$, contradicting the fact that $r_m/r_n$ is harmonic provided by item 2 of Axiom 17 and Axiom 18. Therefore, the $\varepsilon$ must be harmonic. On one hand, based on the Operational Principle of Harmony, all $\varepsilon ^n$'s are harmonic where $n$ varies in integers. On the other hand, the definition of $\varepsilon$ implies that no other harmonic intervals exist (why?). We then conclude
 \begin{center}
$|\mathbb{I}^*|=\{\varepsilon ^n:n\in \mathbb{Z}\}$.
\end{center}
In particular, $|\mathbb{I}^*|$ is countable.
\end{remark}

The philosophical point is that the Axiom of Diatonicism, beside the other principles of harmony, without supplying any harmonic intervals could characterize all of them in the preceding way. The moral is that when tonality changes, no harmonic intervals arise until any intervals of measure  $\varepsilon$ (or $\varepsilon^{-1}$) appears, and continuing the process again produces other harmonic intervals. So it becomes relevant to ask whether equitonal sounds make harmonic intervals.

The interval $\varepsilon$, which is the least harmonic interval greater than unit, is going to be the basis of all musical substances which will be posed from now on. It is clear that all concepts introduced in this way are well-defined if and only if so is the $\varepsilon$; equivalently, there exists a non-unit harmonic interval.
\begin{remark}
If there is some harmonic interval greater than unit, then the function $\phi :(\mathbb{Z},+)\to (|\mathbb{I}^*|,.)$ defined by $\phi (n)=\varepsilon ^n$ is an isomorphism between groups. This fact divulges an algebraic confirmation professing that every infinite cyclic group is isomorphic to the group of integers under addition. Thence, $|\mathbb{I}^*|=<\varepsilon>$ and the only two generators of this cyclic group are $\varepsilon$ and $\varepsilon ^{-1}$.
\end{remark}
\paragraph{}
What remains to determine in the axiom system so far constructed is just the magnitude of the $\varepsilon$. From now on, we intend to find out this indefinite value through the upcoming final axiom which must naturally seem to contain a powerful existential quantifier as intelligent readers anticipate. We will proceed with an unfeigned manner to characterize the concept of harmony in the context of the theory $\mathcal{M}(PI)$ (just as tonality) such that, with no direct utilization of the language of the numbers, the nature of $\varepsilon$ is naively manifested. In order to do this, we need to pass some probably-arduous musical preliminaries.
\begin{definition}
Every interval whose measure is equal to $\varepsilon$ is said to be a \emph{semitone} or \emph{half tone} or \emph{half step}, and every interval whose measure is equal to $\varepsilon ^2$ is said to be a \emph{whole tone} or \emph{whole step}.
\end{definition}
\paragraph{}
In contrast to the fashion at basic music theory we do not discriminate between diatonic and chromatic half tones. The philosophy of such an equalization, which is known as the \emph{equal temperament} constructively treated by J.S. Bach for the first time \cite{D2}, will be described later. Also we use the numbers $\varepsilon$ and $\varepsilon ^2$ as the meaning of semitones and whole tones in the sense of the measure of intervals once in a while.
\begin{definition}
A \emph{monad} is simply defined by any single $(x)$ in which $x$ is a sound. A \emph{dyad} is defined by any ordered pair $(x,y)$ where $x$ and $y$ are sounds such that $x*y$ and $[x,y]$ is a harmonic interval. A \emph{triad} is defined by any ordered triple $(x,y,z)$ of sounds that $x*y*z$ and the two intervals $[x,y]$ and $[y,z]$ are harmonic.
\end{definition}
\paragraph{}
It is worth stating that the concept of polyads, contemporarily supplied as above, is a especial species of the generalized analogue, called \textit{chords}, whose advanced inspection is postponed until we get equipped with the necessary theoretical facilities.
\begin{remark}
An alternate method of defining polyads can be applied as
\begin{itemize}
\item every singleton $\{x\}$ of sounds is said to be a \emph{monad};
\item every unordered pair of non-identical sounds constructing a harmonic interval is said to be a \emph{dyad};
\item every unordered triple of pairwise non-identical sounds which construct two harmonic intervals in pairs is said to be a \emph{triad}.
\end{itemize}
One can first show that this version of the notions of monad, dyad and triad are well-defined (how?), and then prove that it is equivalent to Definition 5.37 which we employ for the sake of convenience. It is worth noting that every sound and every harmonic interval respectively corresponds to a monad and a dyad, but the notion of triad is basically fresh and applicable.
\end{remark}
\begin{definition}
\begin{enumerate}
\item A dyad is called \emph{perfect} if its measure is a semitone greater than the measure of half an octave interval.
\item A triad $(x,y,z)$ is called \emph{major} (\emph{minor}) if the measure of $[x,y]$ is a semitone greater (less) than the measure of $[y,z]$, and moreover, the dyad $[x,z]$ is perfect.
\end{enumerate}
\end{definition}
\paragraph*{}
Alike to octave intervals, perfect dyad is a beautiful entity of music we encounter the next time as a \emph{fifth interval}. We notice that the concepts above are significant whether there exists a non-unit harmonic interval or not, although there is not yet any assurance to find a half tone.
\begin{definition}
An \emph{ascendent mode} is defined by any nonempty finite subset of $|\mathbb{I}^{>1*}|$ the multiplication of some natural power of whose all elements equals the measure of an octave interval; more accurately, an ascendent mode is a set of the form $\emptyset \neq \{r_i: i=1, ..., n\}\subseteq |\mathbb{I}^{>1*}|$ ($n\in \mathbb{N}$) that there is a subset $\{k_i:i=1, ..., n\}$ of natural numbers such that $\Pi _{i=1}^nr_i^{k_i}=2$. A \emph{descendent mode} is any nonempty finite subset of $|\mathbb{I}^{<1*}|$ the multiplication of some natural power of whose all elements equals the inverse of the measure of an octave interval. The \emph{order} of an ascendent (or descendent) mode is defined by the minimum of the sum of all the powers over all such sets of natural numbers, i.e. $\min \{\Sigma _{i=1}^nk_i:\Pi _{i=1}^nr_i^{k_i}=2 (or 1/2)\}$ with the used notation. A (descendent or ascendent) mode having just one member is called \emph{trivial}. A pair of modes $\Delta_1$ and $\Delta_2$ is said to be \emph{conjugate} if $Ord(\Delta_1).Ord(\Delta_2)=\log_{\varepsilon}2$.
\end{definition}
\paragraph*{}
Our method of defining modes is a bit different from what is conventional in basic music theory. The usual treatment leads us to consider this notion as an ordered tuple which we discarded for some reason. We also insist on the requirement that the whole intervals of a mode constitutes an octave and neglect other attitudes towards making modes exceeding the range of an octave interval (see \cite{B1} p. 221).
\begin{remark}[Notation]
We denote by $\Delta ^{+n}$ an ascendent mode $\Delta$ of order $n=Card(\Delta)$, and by $\Delta ^{-n}$ when it is descendent. The set of ascendent modes and the set of descendent modes are respectively denoted by $Mod^+$ and $Mod^-$. One can easily check that the order of a mode $\Delta$, denoted $Ord(\Delta)$, is a well-defined notion.
\end{remark}
\begin{proposition}
Let $r_i\in |\mathbb{I}^{>1*}|$ for $i=1, ..., n$ ($n\in \mathbb{N}$). The following are equivalent:
\begin{enumerate}
\item $\{r_i^{k_i}:i=1, ..., n\}$ is an ascendent mode for a set $\{k_i:i=1, ..., n\}$ of natural numbers.
\item $\{r_i:i=1, ..., n\}$ is an ascendent mode.
\item $\{r_i^{-1}:i=1, ..., n\}$ is a descendent mode.
\item $\{r_i^{-k_i}:i=1, ..., n\}$ is a descendent mode for a set $\{k_i:i=1, ..., n\}$ of natural numbers.
\end{enumerate}
\end{proposition}
\begin{definition}
Based on Proposition 5.42, every ascendent mode naturally gives a descendent one and vice versa.  If $\Delta$ is an ascendent (descendent) mode, we define the \emph{descendent} (\emph{ascendent}) \emph{version} of $\Delta$, denoted $\Delta ^{-1}$, to be the set $\{r^{-1}:r\in \Delta\}$.
\end{definition}
\begin{proposition}
The following conditions are equivalent:
\begin{enumerate}
\item There is a trivial mode.
\item There is a pair of conjugate modes.
\item $Mod^-\neq \emptyset$.
\item $Mod^+\neq \emptyset$.
\item $\varepsilon^n=2$ for some $n\in \mathbb{N}$.
\item $2\in |\mathbb{I}^*|$.
\item $2\varepsilon ^{-1}\in |\mathbb{I}^*|$.
\item $\varepsilon \leq 2$ and $\{2^n:n\in \mathbb{Z}\}\subseteq |\mathbb{I}^*|$.
\item $\{2^m\varepsilon^n :m,n\in \mathbb{Z}\}\subseteq |\mathbb{I}^*|$.
\item $2^m|I|\in |\mathbb{I}^*|$ for every $I\in \mathbb{I}^*$ and every $m\in \mathbb{Z}$.
\item $[s,s^*]\in \mathbb{I}^*$ for some or all $s\in \mathbb{S}$.
\item If $[x,y]\in \mathbb{I}^*$, then $x\simeq x'$ and $y\simeq y'$ imply $[x',y']\in \mathbb{I}^*$.
\end{enumerate}
\end{proposition}
\paragraph*{}
We discern how item 10 and item 12 of Proposition 5.44 state that harmony is well-behaved in regard to tonality, as expected from musical outlook. Having avoided making more statements to describe mathematical treatment of modes, we move on considering the musical properties to achieve what we have in mind as soon as possible.
\begin{definition}
Given a pair of ascendent (descendent) modes $\Delta _1$ and $\Delta _2$. We say that $\Delta_1$ is \emph{coarser} than $\Delta_2$, or $\Delta_2$ is \emph{finer} than $\Delta_1$, denoted $\Delta_1 \lhd \Delta_2$, if $Ord(\Delta_1)<Ord(\Delta_2)$.
\end{definition}
\begin{corollary}
$(Mod^+,\lhd)$ and $(Mod^-,\lhd)$ are strictly ordered sets isomorphic to each other.
\end{corollary}
\begin{remark}
Is there a mode? Is an octave interval harmonic at all? Is a half tone less than an octave interval? Is there any triad? One can prove that all items of Proposition 5.44 are equivalent to the matter that there are the coarsest and finest ascendent (descendent) mode. Formally speaking, each item of the proposition is logically equivalent to the existence of a minimal element of $Mod^+$ ($Mod^-$) with respect to the order $\lhd$ as well as the existence of a maximal element of it. One may also show these elements are unique, namely they are actually the minimum and maximum elements of the ordered set $(Mod^+,\lhd)$ ($(Mod^-,\lhd)$ ). Nevertheless, the only possible coarsest ascendent (descendent) mode will be the singleton $\{2\}$ ($\{2^{-1}\}$) and the single finest one will be $\{\varepsilon\}$ ($\{\varepsilon ^{-1}\}$), both trivial and of course conjugate. In addition, these two elements are equal iff $\varepsilon=2$, i.e. $|I^*|=\{2^n:n\in \mathbb{Z}\}$, or equivalently, $\{\varepsilon\}$ is the only existing ascendent mode, which is trivial and of order $1$. The order of the coarsest mode is naturally $1$ and that of the finest one is just the natural number $n$ satisfying item 5 of Proposition 5.44. This means the existence of mode is logically equivalent to finding one solution to the equation $\varepsilon ^n=2$ in the set of natural numbers. The parameter $n$ naively depends on the value of $\varepsilon$ which we plan to discover. It is worth noting every trivial mode would be of the form $\{\varepsilon ^k\}$ where $k$ is any integer dividing the $n$. We also notice that the existence of a major or minor triad plainly necessitates the existence of a perfect dyad which implies all items of Proposition 5.44, but the converse of this statement is not yet deduced. We wish the answers to the aforementioned questions all to be positive by musical intuition. Of course only one positive answer to either of the first two suffices theoretically, and the second two (i.e., whether $\varepsilon <2$ and the $n$ is even (equivalent to the existence of a pair of conjugate trivial modes one of which is of order $2$)) need more theoretical tools.
\end{remark}
\paragraph*{}
Now we intend to introduce the notion of \textit{maximal evenness} about a set of harmonic intervals (see \cite{J1} p. 27) in the language of our axiomatic system. Since we are not allowed, at the moment, to use any arbitrary musical tools (such as notes and so on) prevalent in diatonic theory which are not already delineated in the system, we try to define an appropriate copy of this notion on the base of other concepts thus far introduced, so that it is of course meaningful in this text adapting to the same phenomenon as created in diatonicism. In order to approach this purpose, the special well-known fact that, `a heptatonic scale is maximally even iff the harmonic intervals between any pair of its adjacent degrees are either half tone or whole tone', inspires us to the general case in such a manner that, a set of sounds is maximally even if, abstractly, each pair of adjacent sounds, say \emph{generic} interval, contains either a single number or two consecutive numbers of half steps, say \emph{specific} interval, and if, concretely, all sounds are spread out as much as possible within an octave interval . Naturally, a mode is maximally even when it generates a maximally-even set of sounds.
\begin{definition}
A (descendent or ascendent) mode $\Delta$ is said to be \emph{maximally even} if there is an integer $k$ such that $\Delta \subseteq \{\varepsilon^k,\varepsilon^{k+1}\}$.
\end{definition}
\begin{corollary}
An ascendent (descendent) mode is maximally even iff so is its descendent (ascendent) version.
\end{corollary}
\begin{corollary}
Every trivial mode is maximally even.
\end{corollary}
\begin{definition}
The ascendent mode $\{\varepsilon\}$ and its descendent version are said to be \emph{chromatic}. The ascendent mode $\{\varepsilon ,\varepsilon ^2\}$ and its descendent version are said to be \emph{diatonic}. The ascendent mode $\{\varepsilon^2\}$ and its descendent version are called the \emph{whole tone} modes.
\end{definition}
\paragraph*{}
From modality point of view, the above modes are three specific important types of maximally even ones two of which are trivial. Also, there are other kinds of translation of equivalent items provided in Proposition 5.44 into the language of these maximally even modes; the existence of the chromatic mode is equivalent to each of the items and particularly the existence of a natural number $n$ satisfying $\varepsilon^n=2$ (item 5); the existence of the whole tone mode is equivalent to the evenness of the natural number $n$; the existence of the diatonic mode is equivalent to the veracity of the inequality $n\geq 3$.
\paragraph*{}
For the time being we initiate the most important structural entity of the system that is fundamental in musical composition. Again, similar to modes, We take account of the condition that the whole interval of a scale constitutes an octave, despite what is customary with some alternative approaches permitting scales to exceed or even not to reach the expanse of an octave interval. We also supply an extra term concerning \emph{compatibility} of scales (whose ascendent and descendent version are practically adapted) for manifestation of interest in working with such a kind of them as well as music theorists do.
\begin{definition}
By a \emph{scale} we mean any ordered triple $(\hat{s},(r_i)_{i=1}^m,(r'_i)_{i=1}^n)$, where $s\in \mathbb{S}$, and $(r_i)_{i=1}^m$ and $(r'_i)_{i=1}^n$ are ordered tuples formed of harmonic intervals greater and less than unit respectively whose addition gives an octave interval; i.e., $(r_i)_{i=1}^m\in |\mathbb{I}^{>1*}|^m$ and $(r'_i)_{i=1}^n\in |\mathbb{I}^{<1*}|^n$ with $\Pi_{i=1}^mr_i=2$ and $\Pi_{i=1}^nr'_i=1/2$. We call a scale $(\hat{s},(r_i)_{i=1}^m,(r'_i)_{i=1}^n)$ \emph{compatible} whenever $m=n$ and $r'_i=r^{-1}_{n-i+1}$ for all $i=1, ..., n$, and denote it by $(\hat{s},(r_i)_{i=1}^n)$ for simplicity.
\end{definition}
\begin{remark}[Notation]
We denote by $\mathcal{S}$ the set of all scales, by $\mathcal{S}^c$ the set of all compatible scales, and by $\mathcal{S}_{\hat{s}}$ the set of all scales whose first coordinate is $\hat{s}$ for a class $\hat{s}$. Obviously, the set $\mathcal{S}_{\hat{s}}$ is finite for every $s\in \mathbb{S}$, and hence the two sets $\mathcal{S}$ and $\mathcal{S}^c$ are equipotent to (as large as) the $\mathbb{S}$.
\end{remark}
\begin{corollary}
For every scale $(\hat{s},(r_i)_{i=1}^m,(r'_i)_{i=1}^n)$, the two sets $\Delta =\{r_i:i=1, ..., m\}$ and $\Delta' =\{r'_i:i=1, ..., n\}$ are ascendent and descendent modes respectively so that $Ord(\Delta)\leq m$ and $Ord(\Delta')\leq n$.
\end{corollary}
\begin{definition}
For every scale $S=(\hat{s},(r_i)_{i=1}^m,(r'_i)_{i=1}^n)$, the two sets $\Delta =\{r_i:i=1, ..., m\}$ and $\Delta' =\{r'_i:i=1, ..., n\}$ are respectively called the \emph{ascendent} and \emph{descendent} mode of the scale $S$.
\end{definition}
\begin{corollary}
Given a compatible scale $S$. The ascendent (descendent) version of the descendent (ascendent) mode of $S$ is the ascendent (descendent) mode of $S$.
\end{corollary}
\begin{remark}
Corollary 5.54 states that any scale gives only one ascendent (descendent) mode, but the converse does not hold uniquely unless the mode is trivial. It means that each mode determines at least one scale in general (how?). This lack of symmetry in the relation between two concepts of mode and scale causes difficulty to define scale in terms of mode directly. On the other hand, we let the function $\chi :\mathbb{S}/\simeq \times \cup_{n\in \mathbb{N}}{\mathbb{R}^+}^n\to \cup_{n\in \mathbb{N}}(\mathbb{S}/\simeq)^n$ be acting as $(\hat{s},(r_i)_{i=1}^m,(r'_i)_{i=1}^n)\mapsto (\hat{s},(\hat{s_i})_{i=1}^{m},(\hat{s'_i})_{i=1}^{n})$, where $s_0=s'_0=s$, and $s_i$ and $s'_j$ are sounds (unique up to identity) with $f(s_{i})=r_if(s_{i-1})$ and $f(s'_{j})=r'_jf(s'_{j-1})$ for every $1\leq i\leq m$ and every $1\leq j\leq n$ ($f$ is the frequency of sounds). One can easily check $\chi$ is well-defined and onto as well, but it is not clearly one-to-one. Due to the Tonality Characterization Theorem, if we work with some suitable equivalence classes of measures of intervals as $[r]:=\{2^nr:n\in \mathbb{Z}\}$ ($r\in \mathbb{R}^+$), we will then get the function $\chi':\mathbb{S}/\simeq \times \cup_{n\in \mathbb{N}}(\mathbb{R}^+/[\,])^n\to \cup_{n\in \mathbb{N}}(\mathbb{S}/\simeq)^n$ acting as $(\hat{s},([r_i])_{i=1}^m,([r'_i])_{i=1}^n)\mapsto (\hat{s},(\hat{s_i})_{i=1}^{m},(\hat{s'_i})_{i=1}^{n})$ bijective. Using the surjective map $\jmath :\cup_{n\in \mathbb{N}}{\mathbb{R}^+}^n\to \cup_{n\in \mathbb{N}}(\mathbb{R}^+/[\,])^n$ defined by $(r_i)_{i=1}^n\mapsto ([r_i])_{i=1}^n$ together with the identity map $Id:\mathbb{S}/\simeq \to \mathbb{S}/\simeq$, we again obtain the same $\chi =\chi'o(Id\times \jmath)$ that is not one-to-one. However, we observe that the restriction of the $\chi$ to the set of all scales, namely $\chi |_\mathcal{S}^{\chi(\mathcal{S})}$, sets up a one-to-one correspondence (why?). This fact naturally gives another equivalent way to define scales founded on ordered tuples of equitonal classes of sounds that is treated in some alternative approaches. The point is that in this method of introducing the notion of scales as well as what is supplied in Definition 5.52, modality of our music will not collapse, and despite its misleading appearance, it will really preserve modality of scales in the sense of what we wished to take place and tended to be formulated within its definition, although the function $\chi$ does not necessarily advance on bijectivity in general (even its restriction to the significantly efficient set $\mathbb{S}/\simeq \times \cup_{n\in \mathbb{N}}|\mathbb{I}^*|^n$). It is worth noting that we apply the construction of the function $\chi$ for ease of some conceptual purposes.
\end{remark}
\begin{definition}
In virtue of material verified in Remark 5.57, the restricted function $\chi |_{\mathbb{S}/\simeq \times \cup_{n\in \mathbb{N}}|\mathbb{I}^*|^n}^{\chi(\mathbb{S}/\simeq \times \cup_{n\in \mathbb{N}}|\mathbb{I}^*|^n)}$ is called the \emph{module generation}, every member $\chi(x)=(\hat{s},(\hat{s_i})_{i=1}^{m},(\hat{s'_i})_{i=1}^{n})$ of its range is called the \emph{generated module} of $x$, the two ordered tuples $(\hat{s_i})_{i=0}^{m}$ and $(\hat{s'_i})_{i=0}^{n}$ are respectively called the \emph{ascendent module} and the \emph{descendent module} of $x$, the set of all coordinates of $\chi(x)$ is called the \emph{module components set} of $x$, and the two sets $\{\hat{s_i}:i=0, 1, ..., m\}$ and $\{\hat{s'_i}:i=0, 1, ..., n\}$ are respectively called the \emph{ascendent} and \emph{descendent} \emph{components set} of $x$ (where $x\in \mathbb{S}/\simeq \times \cup_{n\in \mathbb{N}}|\mathbb{I}^*|^n$).
\end{definition}
\paragraph*{}
In practice, the module components set of a scale specifies all claviers required to play a melody on the base of that scale by means of any keyboard. As an intuitive result, we have the following proposition.
\begin{lemma}
Let $S=(\hat{s},(r_i)_{i=1}^n,(r'_i)_{i=1}^n)$ be a compatible scale. Then
\begin{center}
$S=(\hat{s},(2\underset{j\neq n-i+1}{\Pi}r'_j)_{i=1}^n,(\frac{1}{2}\underset{j\neq n-i+1}{\Pi}r_j)_{i=1}^n)$.
\end{center}
\end{lemma}
\begin{proposition}
A given scale is compatible iff its module components set coincides with its ascendent and descendent components set.
\end{proposition}
\begin{proof}[Sketch of Proof]
It suffices to prove the coincidence of the ascendent and descendent components sets of a compatible scale by applying Lemma 5.59 for the ``only if'' part of the statement. Using induction on the cardinality of the (ascendent or descendent) components set  simplifies the ``if'' part proof.
\end{proof}
\begin{definition}
Given a scale $S=(\hat{s},(r_i)_{i=1}^m,(r'_i)_{i=1}^n)$. $S$ is said to be \emph{trivial}, or \emph{chromatic}, or \emph{diatonic}, if both its ascendent and descendent modes are trivial, or chromatic, or diatonic, accordingly. For every $1\leq k\leq m$ ($1\leq k\leq n$), the $k$-th coordinate of the ascendent (descendent) module of $S$ is called the \emph{$k$-th ascendent} (\emph{descendent}) \emph{degree} of the scale. Particularly, the first degree of $S$, namely $\hat{s}$, is called its \emph{tonic}. $S$ is called \emph{monotonic} if $m=n=1$, it is called \emph{pentatonic} if $m=n=5$, it is called \emph{heptatonic} if $m=n=7$, and it is called \emph{dodecatonic} if $m=n=12$ (other species are left for applicability defect in the framework).
\end{definition}
\begin{remark}
Given $S=(\hat{s},(r_i)_{i=1}^n)\in \mathcal{S}^c$. For every $1\leq k\leq n+1$, the $k$-th ascendent (descendent) degree of $S$ is equivalent to the $(n-k+2)$-th descendent (ascendent) degree of $S$ by the agreement that the $(n+1)$-th ascendent (descendent) degree is equated with the tonic $\hat{s}$, because the first and the last coordinates of the ascendent (descendent) module of any scale are actually the same.
\end{remark}
\begin{corollary}
The module components set of a scale is equal to the set of all its ascendent and descendent degrees.
\end{corollary}
\begin{corollary}
A (compatible) scale is monotonic iff its module components set is the singleton containing the tonic.
\end{corollary}
\begin{corollary}
Every scale whose ascendent and descendent modes are chromatic or whole tone is trivial.
\end{corollary}
\begin{corollary}
The existence of a scale is equivalent to each of the following:
\begin{enumerate}
\item $\mathcal{S}\neq \emptyset$.
\item There is a compatible scale; i.e., $\mathcal{S}^c\neq \emptyset$.
\item There is a trivial scale.
\item There is a chromatic scale.
\item There is a monotonic scale.
\end{enumerate}
Moreover, it is equivalent to each item of Proposition 5.44.
\end{corollary}
\begin{corollary}
\begin{enumerate}
\item There is at most one monotonic scale up to relativity;
\item There is at most one chromatic scale up to relativity;
\end{enumerate}
(see Remark 5.68) both of which are compatible and trivial. Moreover, the existence of a diatonic scale implies that of a monotonic or chromatic one but not necessarily vice versa.
\end{corollary}
\begin{remark}
In the literature, two scales having the same tonic are called \emph{parallel} which makes with the same class of scales as denoted by $\mathcal{S}_{\hat{s}}$ for a fixed sound $s$. Also, two scales with the same modes are called \emph{relative}. There is a practical manipulation of scales, possessing a significant algebraic property, that gives a natural relationship between two parallel scales. In practice, by virtue of any two given scales, we can obtain a unique productive scale whose degrees belong to the union of the module components sets of the given scales, having obeyed the usual ordering of their degrees within the octave of course. (We have refrained from supplying the formal definition for the nonce because of the complexity of its presentation, preferring instead to follow a simpler pattern.) We notice the necessary and sufficient condition for the product to be a scale is that the tonics of the two given scales construct a harmonic interval. For this reason, we introduce this binary operation on the some suitable set of scales, e.g., the set of all parallel scales to a given one. Let $s\in \mathbb{S}$ be fixed once and for all, and let scales $S_1=(\hat{s},(a_i)_{i=1}^m,(a'_i)_{i=1}^n)$ and $S_2=(\hat{s},(b_i)_{i=1}^k,(b'_i)_{i=1}^l)$ be given. We set
\begin{center}
$\,A=\{\overset{j}{\underset{i=1}{\Pi}}a_i:1\leq j\leq m\}\cup \{\overset{j}{\underset{i=1}{\Pi}}b_i:1\leq j\leq k\}$,\\
$B=\{\overset{j}{\underset{i=1}{\Pi}}a'_i:1\leq j\leq n\}\cup \{\overset{j}{\underset{i=1}{\Pi}}b'_i:1\leq j\leq l\}.\,\,\,$
\end{center}
Now, we recursively define the operation $\sqcup :\mathcal{S}_{\hat{s}}\times \mathcal{S}_{\hat{s}}\to \mathcal{S}_{\hat{s}}$ by
\begin{center}
$S_1\sqcup S_2=(\hat{s},(c_i)_{i=1}^p,(c'_i)_{i=1}^q)$,
\end{center}
where $c_1=\min A$, $c'_1=\max B$, and $c_i=(1/c_{i-1})\min (A-\{c_1, ..., c_{i-1}\})$, $c'_j=(1/c'_{j-1})\max (B-\{c_1, ..., c_{j-1}\})$
 for $1<i\leq p$ and $1<j\leq q$. The recursion will continue as far as the minimum and maximum elements exist, namely the underlying sets are nonempty, and the natural parameters $p$ and $q$ will be determined. This idea is inspired by the simple way to combine two strictly increasing (decreasing) sequences of real numbers so that we gain a unique strictly increasing (decreasing) sequence again. One can easily check that the binary operation $\sqcup$ is well-defined. In addition, the algebra $(\mathcal{S}_{\hat{s}},\sqcup)$ is associative and its identity element is the monotonic trivial scale $(\hat{s},(2),(1/2))$, say \emph{neutral} scale. Unfortunately, no elements are invertible unless the identity, but since any two scales in $\mathcal{S}_{\hat{s}}$ commute and every one is idempotent (that is, $S\sqcup S=S$), it follows that we deal with the algebraic structure $(\mathcal{S}_{\hat{s}},\sqcup)$ as a finite semilattice. We know from lattice theory that a partial ordering $\sqsubseteq$ is induced on the semilattice $(\mathcal{S}_{\hat{s}},\sqcup)$, defined by $S_1\sqsubseteq S_2$ iff $S_1\sqcup S_2=S_2$, but $(\mathcal{S}_{\hat{s}},\sqsubseteq)$ is not a totally ordered (and a fortiori a well-ordered) set, however, every pair of its elements $S_1$ and $S_2$ has a least upper bound $S_1\sqcup S_2$. In this case, $(\mathcal{S}_{\hat{s}},\sqcup)$ is referred to as a \emph{join semilattice} (or \emph{meet semilattice}, considering the dual order $\sqsubseteq ^{-1}$ on the $\mathcal{S}_{\hat{s}}$). Also, an interesting element exists in $(\mathcal{S}_{\hat{s}},\sqcup)$ (and in any finite semilattice) which plays a complementary role to the neutral scale, that is, $\sqcup \mathcal{S}_{\hat{s}}=\sqcup _{S\in \mathcal{S}_{\hat{s}}}S$, denoted $S^*$, having the property that $S^*\sqcup S=S^*$ for every $S\in \mathcal{S}_{\hat{s}}$. In fact, we have $S^*=(\hat{s},(\varepsilon)_{i=1}^n,(\varepsilon^{-1})_{i=1}^n)$ where $n$ is the natural number satisfying item 5 of Proposition 5.44, and whenever $S^*$ exists, it will be the unique chromatic scale of the system (Corollary 5.67) which is the greatest element of $\mathcal{S}_{\hat{s}}$ with respect to the order $\sqsubseteq$, while the unique monotonic (neutral) scale of the system is the least one. We occasionally use the name \emph{embedding} for the order $\sqsubseteq$, and $S_1\sqsubseteq S_2$ is read as $S_1$ is \emph{embedded} in $S_2$. One can also define a binary operation $\sqcap$ on $\mathcal{S}_{\hat{s}}$ taking the intersection of
the module components sets of two given scales, and construct the finite lattice $(\mathcal{S}_{\hat{s}},\sqcup ,\sqcap)$.
\end{remark}
\begin{exercise}
Define the \emph{inverse} of a scale $S=(\hat{s},(r_i)_{i=1}^m,(r'_i)_{i=1}^n)$ as the scale $(\hat{s},(r'^{-1}_{n-i+1})_{i=1}^n,(r^{-1}_{m-i+1})_{i=1}^m)$, denoted $S^{-1}$. Prove that a scale $S$ is compatible iff so is its inverse iff $S=S^{-1}$, and prove that the scale $S\sqcup S^{-1}$ is the least (with respect to $\sqsubseteq$) compatible scale that $S$ (and $S^{-1}$) is embedded in.
\end{exercise}
\begin{definition}
A scale $(\hat{s},(r_i)_{i=1}^m,(r'_i)_{i=1}^n)$ is said to be \emph{maximally even} if for every natural number $k$ there is a natural number $n_k$ such that
\begin{center}
$\forall i(i\in \mathbb{N}\Rightarrow \overset{i+k}{\underset{j=i}{\Pi}}r_{j (\mod m)}\in \{\varepsilon^{n_k},\varepsilon^{n_k+1}\}\wedge \overset{i+k}{\underset{j=i}{\Pi}}r'_{j (\mod n)}\in \{\varepsilon^{-n_k},\varepsilon^{-n_k-1}\})$.
\end{center}
Here, the index $x (\mod y)$ is considered as the unique natural number $x'$ less than or equal to $y$ so that $x$ and $x'$ are congruent modulo $y$ ($x\equiv x' (\mod y)$); i.e., $x$ and $x'$ have the same remainder when divided by $y$, equivalently, $x-x'$ is divisible by $y$ ($x'\leq y$ and $y|(x-x')$). In other words, $x (\mod y)$ is actually the remainder of the division algorithm in dividing $x$ by $y$, except when $x$ is a multiple of $y$ (i.e. $y|x$) in which case it is considered the modulus $y$ (not the remainder zero) as a matter of convention.
\end{definition}
\paragraph*{}
By a sort of philosophically thinking of Definition 5.70, it is implicitly announced that the cyclic nature of tonality of our music turns out to create modality fundamentally and absolutely. We should express that we planned on introducing the notion of maximal evenness for scales in accordance with diatonicism, in the sense that the multiplication of every number of adjacent intervals of a maximally even scale must be either one quantity of semitones or two consecutive quantities of them, and in addition, in such a way that it becomes well-behaved with respect to the analogous notion for modes, as we will disclose that the modes of each maximally even scale are maximally even too. This matter can also be viewed from the point that such a notion is essentially well-defined as anticipated. In the following, Proposition 5.71 intuitively states that it is necessary and sufficient to check the condition of being maximally even for a scale up to an octave interval.
\begin{proposition}
A scale $(\hat{s},(r_i)_{i=1}^m,(r'_i)_{i=1}^n)$ is maximally even if and only if for every $1\leq k\leq \max \{m,n\}$ there is an $n_k\in \mathbb{N}$ so that
\begin{center}
$\forall i(1\leq i\leq m\Rightarrow \overset{i+k}{\underset{j=i}{\Pi}}r_{j (\mod m)}\in \{\varepsilon^{n_k},\varepsilon^{n_k+1}\})$,\\
$\forall i(1\leq i\leq n\Rightarrow \overset{i+k}{\underset{j=i}{\Pi}}r'_{j (\mod n)}\in \{\varepsilon^{-n_k},\varepsilon^{-n_k-1}\})$.
\end{center}
\end{proposition}
\begin{proposition}
A compatible scale $(\hat{s},(r_i)_{i=1}^n,(r'_i)_{i=1}^n)$ is maximally even iff
\begin{center}
$\forall k\in \mathbb{N}^{<n}(\exists n_k\in \mathbb{N} (\forall i\in \mathbb{N} (i<n \Rightarrow \overset{i+k}{\underset{j=i}{\Pi}}r_{j (\mod n)}\in \{\varepsilon^{n_k},\varepsilon^{n_k+1}\})))$
\end{center}
if and only if
\begin{center}
$\forall k\in \mathbb{N}^{<n}(\exists n_k\in \mathbb{N} (\forall i\in \mathbb{N} (i<n \Rightarrow \overset{i+k}{\underset{j=i}{\Pi}}r'_{j (\mod n)}\in \{\varepsilon^{-n_k},\varepsilon^{-n_k-1}\})))$.
\end{center}
\end{proposition}
\begin{corollary}
Every trivial scale is maximally even. In particular, every monotonic scale and every chromatic scale are maximally even.
\end{corollary}
\begin{proposition}
The ascendent and descendent modes of every maximally even scale are maximally even.
\end{proposition}
\begin{proof}[Sketch of Proof]
Check the possibility mentioned in Proposition 5.71 for the natural numbers $k=m$ and $k=n$, and use $\Pi_{i=1}^{m}r_i=1/\Pi_{i=1}^{n}r'_i=2$.
\end{proof}
\begin{remark}
It is clear that if one of two relative scales is maximally even, then so is the other, but there is a more general and interesting fact. We define a \emph{rearrangement} of a scale $(\hat{s},(r_i)_{i=1}^m,(r'_i)_{i=1}^n)$ by any scale $(\hat{s},(r_{\tau(i)})_{i=1}^m,(r'_{\tau'(i)})_{i=1}^n)$, where $\tau$ and $\tau'$ are permutations (bijections) on the sets $\{1, ..., m\}$ and $\{1, ..., n\}$ respectively. A \emph{circular rearrangement} of $(\hat{s},(r_i)_{i=1}^m,(r'_i)_{i=1}^n)$ is defined by a rearrangement of that the order of whose ascendent and descendent modules is preserved up to modulus; in other words, any scale $(\hat{s},(r_{\tau(i)})_{i=1}^m,(r'_{\tau'(i)})_{i=1}^n)$ where $\tau$ and $\tau'$ are cyclic permutations on the sets $\{1, ..., m\}$ and $\{1, ..., n\}$ respectively (i.e., $\tau(i)=i+k(\mod m)$ and $\tau'(j)=j+k'(\mod n)$  for some $1\leq k\leq m$ and some $1\leq k'\leq n$). Thus, for a given scale $S$, the set of all rearrangements of $S$ is a subset of the set of all scales parallel to $S$. Notice that an arbitrary rearrangement of a maximally even scale is not necessarily maximally even (e.g., in case of diatonic scales), although so is every circular rearrangement of that by definition.
\end{remark}
\begin{definition}
Given a scale $S=(\hat{s},(r_i)_{i=1}^m,(r'_i)_{i=1}^n)$ with the generated module $\chi(S)=(\hat{s},(\hat{s_i})_{i=1}^{m},(\hat{s'_i})_{i=1}^{n})$.
\begin{enumerate}
\item The ascendent (descendent) module of $S$ is said to have the \emph{sensible} when $r_m$ ($r'^{-1}_n$) is a half step. In this case, the last or $m$-th ascendent ($n$-th descendent) degree of $S$ is named the \emph{ascendent} (\emph{descendent}) \emph{sensible}.
\item Every ordered triple of the forms $(r_{k-1},r_k,r_{k+1})$ and $(r'_{l-1},r'_l,r'_{l+1})$ for a $1<k<m$ and an $1<l<n$ is said to be a \emph{tetrachord}, and in this case we say $S$ has that tetrachord. Every pair of tetrachords are called \emph{disjoint} if their corresponding module components sets do not intersect except probably at the tonic of $S$.
\item A triad $(x,y,z)$ is said to be \emph{ascendently} (\emph{descendently}) \emph{scale-based} on the $k$-th ascendent (descendent) degree of $S$ ($1\leq k\leq m$ ($1\leq k\leq n$)) if the module components set of $(\hat{s_k}_{-1},(|[x,y]|,|[y,z]|))$ ($(\hat{s'_k}_{-1},(|[z,y]|,|[y,x]|))$) is a subset of the ascendent (descendent) components set of the scale $S$.
\end{enumerate}
\end{definition}
\paragraph*{}
The mathematically minded reader may repine to face such an extensive area of fresh musical concepts but should know this is not our fault. These are the most natural substances a musician is inseparably involved in and enjoys for musical intentions. For example, the sensible of a scale, also known as the leading note in music theory, refers to that degree of scale (the last degree in this situation) inducing the listener an incentive to expect the next degree (namely the gamut equivalent to the same tonic) \cite{M4}; speaking compositionally, the exorbitant stop on the sensible is not so pleasant because the human ear stays waiting to hear the next note. As another example, the application of tetrachords returns to that period of time this notion used to be accounted the basis for constructing various musical scales by theoreticians (take a look at \cite{S1} p. 235), which we adopted the reverse of this method. We recommend the reader to communicate with these subjects by a musically geometrical insight. For instance, a chord (or triad) is scale-based when its components lie on the degrees of scale; easy A!
\begin{remark}
\textbf{Warning.} A scale whose ascendent (descendent) module has the sensible need not have the descendent (ascendent) sensible.\\
By the notation used in Definition 5.76, what is meant by disjoint tetrachords in item 2 is the following; whenever, e.g., $(r_{k-1},r_k,r_{k+1})$ and $(r'_{l-1},r'_l,r'_{l+1})$ are two tetrachords, they are disjoint iff the intersection of the module components set of $(\hat{s_k}_{-1},(r_{k-1},r_k,r_{k+1}))$, i.e. $\{\hat{s_k}_{-2},\hat{s_k}_{-1},\hat{s_k},\hat{s_k}_{+1}\}$, and that of $(\hat{s'_l}_{-1},(r_{l-1},r_l,r_{l+1}))$, i.e. $\{\hat{s'_l}_{-2},\hat{s'_l}_{-1},\hat{s'_l},\hat{s'_l}_{+1}\}$, (where $\hat{s_0}=\hat{s'_0}=\hat{s}$) is either empty or the singleton $\{\hat{s}\}$. The lexicology of tetrachord as a quadruple of adjacent ascendent or descendent degrees of a scale is pretty clarified right here.\\
About triads scale-based on a degree of a given scale, we indicate and apply a briefer terminology when working on compatible scales; since there would be a one-to-one correspondence between the set of triads which are ascendently scale-based on the degrees of a compatible scale $S=(\hat{s},(r_i)_{i=1}^n,(r^{-1}_{n-i+1})_{i=1}^n)$ and that of descendently scale-based on the degrees of $S$ as
\begin{center}
$(\hat{s_k}_{-1},|[x,y]|,|[y,z]|)\longleftrightarrow (\hat{s'_n}_{-k+1},|[z,y]|,|[y,x]|)$
\end{center}
for every $1\leq k\leq n$ and using the notation employed in Definition 5.76, we simply say, in this case, that the triad $(x,y,z)$ is scale-based on the $k$-th (ascendent) degree of the scale $S$ insofar as it does not lead to ambiguity. 
\end{remark}
\begin{corollary}
Given a scale $S=(\hat{s},(r_i)_{i=1}^m,(r'_i)_{i=1}^n)$.
\begin{enumerate}
\item If $m+n\leq 3$, then $S$ and particularly every monotonic scale have no tetrachords. The number of tetrachords of $S$ is equal to $m+n-4$ on condition that $m+n>3$.
\item The number of (unordered) pairs of disjoint tetrachords of $S$ with $m,n\geq 7$ equals $(m-5)(m-6)+(n-5)(n-6)$. In particular, a heptatonic scale has four pairs of disjoint tetrachords.
\end{enumerate}
\end{corollary}
\begin{corollary}
Let $S=(\hat{s},(r_i)_{i=1}^m,(r'_i)_{i=1}^n)$ be a scale with the generated module $\chi(S)=(\hat{s},(\hat{s_i})_{i=1}^{m},(\hat{s'_i})_{i=1}^{n})$. A triad $(x,y,z)$ is ascendently (descendently) scale-based on the $k$-th ascendent (descendent) degree of $S$ iff there are natural numbers $p$ and $q$ with $k<p<q$ so that 
\begin{center}
$\chi((\hat{s_k},|[x,y]|,|[y,z]|))=(\hat{s_k}_{-1},\hat{s_p}_{-1 (\mod m)},\hat{s_q}_{-1 (\mod m)})$
($\chi((\hat{s_k},|[x,y]|,|[y,z]|))=(\hat{s'_k}_{-1},\hat{s'_p}_{-1 (\mod n)},\hat{s'_q}_{-1 (\mod n)})$).
\end{center}
\end{corollary}
\paragraph*{}
Having forbore from giving more elementary mathematical statements, we are ready to present the last and the most basal principle of the system as was previously mentioned based on all material collected up to this moment.
\begin{axiom}[Axiom of Gamme/Major Scale Axiom]
There exists a heptatonic compatible scale with a maximally even mode, whose ascendent module has the sensible but so does not the descendent one, having two disjoint and equal tetrachords, and such that a major triad is scale-based on its tonic.
\end{axiom}
\paragraph*{}
The Axiom of Gamme is the very same point the digit $7$, as the number of musical notes, essentially emerges from in the world of music. Also, it stealthily gives the real number $\varepsilon$ for the exact value of half steps as will follow. Since we aimed at philosophically hiding and logically revealing this value just as the number $7$ and we yearned to take this action in a realistic manner not seemed to be so artificial, we have borne the heavy encumbrance of all the new technical concepts recently introduced to formulize the axiom; otherwise, we could have postulated many other statements extremely simpler than one entitled ``Major Scale Axiom'' (see Remark 5.85). Note that we abstained from symbolic rephrasing of the axiom to exhibit its spiritual magnificence and, of course, to save paper.
\begin{remark}
The Axiom of Gamme is undoubtedly the most powerful axiom of our music system occasionally referred to as the \emph{Fundamental Principle of Music}. This axiom is musically fundamental for the following two basic reasons:
\begin{itemize}
\item It firstly states that there are equitonal sounds non-identical, making the point where we are no more in need of the assumption mentioned in Remark 5.9. If we ignored this property of music, then there would be infinitely many distinct musical notes making the performance practically impossible for the instrumentalist (as indicated in Figure 1)\footnote{\textsl{It is easy to play any musical instrument; all you have to do is touch the right key at the right time and the instrument will play itself.}}.
\item It secondly states that all intervals created by equitonal sounds are harmonic. More generally, all items of Proposition 5.44 are satisfied without which doing more music would be practically nonsense.
\end{itemize}
In particular, once such an axiom is established, the octave intervals are both existent and harmonic.
\end{remark}
\begin{figure}
\begin{center}
\includegraphics[width=12cm]{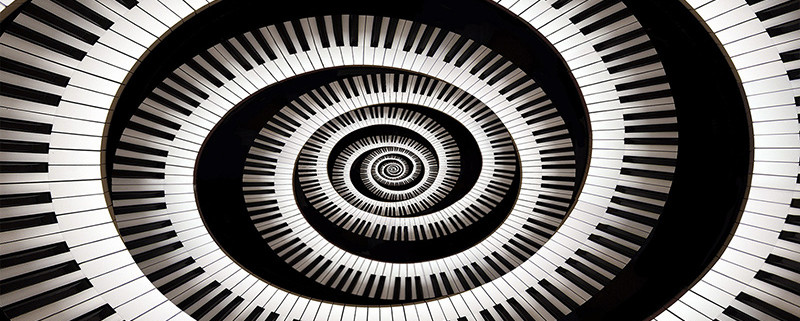}
\end{center}
\caption{\scriptsize Having had a keyboard of length $\infty$ in hand, the validity of the famous quotation exposed in the footnote, due to J.S. Bach, does happen to be distorted because of inaccessibility to all notes, in the sense that there is not enough time to touch an arbitrary key, let alone the right one!}
\end{figure}
\begin{theorem}
The $\varepsilon$ is well-defined and equal to $2^{1/12}$.
\end{theorem}
\begin{proof}[Sketch of Proof]
The existence of a harmonic octave interval followed by Remark 5.80 guarantees that $\varepsilon$ introduced in Remark 5.34 is well-defined. Using logic rules within indirect proof for several times turns out that the scale involved in Axiom 20 is of the form $(\hat{s},(\varepsilon^2,x,\varepsilon,y,\varepsilon^2,x,\varepsilon))$ for some $s\in \mathbb{S}$, controlled by three equations $\varepsilon^6x^2y=2$, $\varepsilon^2xy=2^{1/2}$, and $x=y$. It easily follows that $\varepsilon=2^{1/12}$. Moreover, $x=y=2^{1/6}$ that is a whole tone.
\end{proof}
\begin{corollary}[Representation of Harmony]
$|\mathbb{I}^*|=\{2^{n/12}:n\in \mathbb{Z}\}$.
\end{corollary}
\begin{remark}
All $\varepsilon$-based definitions such as half/whole tone are meaningful. The calculation of the number of all kinds of modes and scales existent in the system (which are all finite), i.e. $Card(Mod^{\pm})$, $Card(\mathcal{S}_{\hat{s}})$, and $Card(\mathcal{S}_{\hat{s}}^c)$ for a given $s\in \mathbb{S}$, is a matter of combinatorics. Also the classification of pairs of conjugate modes is easily possible left to the reader as a relieving exercise.\\
It is necessary to say that although the quantity $\varepsilon$ is already characterized, we still employ the notations $\varepsilon$ and $\varepsilon^2$ for half steps and whole steps respectively for ease of use, especially in the following literary definition which introduces two  species of diatonic scales which are circular rearrangement of each other and most significant of all.
\end{remark}
\begin{definition}
Any compatible scale of the form $(\hat{s},(\varepsilon^2,\varepsilon^2,\varepsilon,\varepsilon^2,\varepsilon^2,\varepsilon^2,\varepsilon))$ is said to be a \emph{major scale}, and any one of the form $(\hat{s},(\varepsilon^2,\varepsilon,\varepsilon^2,\varepsilon^2,\varepsilon,\varepsilon^2,\varepsilon^2))$ is said to be a \emph{minor scale} ($s\in \mathbb{S}$).
\end{definition}
\begin{remark}
All items of Corollary 5.66 are satisfied at the moment. In particular, the unique chromatic scale up to relativity is of the form $(\hat{s},(\varepsilon)_{i=1}^{12})$ that is dodecatonic; actually, a scale is chromatic iff it is dodecatonic at the present time. It is clear the second name of Axiom 20 is explained by its last part stating that the major triad is scale-based on the tonic, and the scale involved is actually a major scale due to Definition 5.84. This scale is diatonic as well, and hence by a little modification (specifically in the last part) of the axiom we could posit a minor version of that. Of course, we also could replace this axiom with the simplest one as follows:
\begin{center}
$\exists S\in \mathcal{S}(\exists s\in \mathbb{S}(S=(\hat{s},(\varepsilon)_{i=1}^{12},(\varepsilon^{-1})_{i=1}^{12})))$,
\end{center}
which means there exists a dodecatonic chromatic scale, say ``Axiom of Chromaticism/Chromatic Scale Axiom'', in which case we would obtain the same consequences; having compared this axiom and Axiom 20, we realize there are some substantial similarity between them such as compatibility and maximal evenness, and some structural differences stating that all the degrees of the chromatic scale are practically the sensible and the major and minor triads are scale-based on each degree. However, we did not consider such a brief axiom because it is accounted artificial or at least not so natural as the major scale is from musical philosophy perspective mentioned before. On the other hand, the marvel of the Major Scale Axiom whose harmonic intervals are uniquely determined is to logically characterize all harmonic intervals (Theorem 5.81) in a completely naive way and through such various complicated musical attributes as involved.
\end{remark}
\begin{remark}
From Proposition 5.74 we know that the modes of every maximally even scale are maximally even too, but the converse is not true; as a counterexample, we can artificially construct a (heptatonic) diatonic scale, having two half tones adjacent, whose other steps are whole tone. The following proposition, expressing how scale involved in the Major Scale Axiom is maximally even in the meanwhile, shall provide a necessary and sufficient condition for the satisfaction of a special case of the converse of Proposition 5.74 as well.
\end{remark}
\begin{proposition}
A heptatonic diatonic scale $(\hat{s},(r_i)_{i=1}^{7},(r'_i)_{i=1}^{7})$ is maximally even if and only if it is a circular rearrangement of the major scale $(\hat{s},(\varepsilon^2,\varepsilon^2,\varepsilon,\varepsilon^2,\varepsilon^2,\varepsilon^2,\varepsilon))$. In particular, every major scale and every minor scale are maximally even.
\end{proposition}
\begin{remark}
The study of those scales $(\hat{s},(r_i)_{i=1}^{m},(r'_i)_{i=1}^{n})$ whose ascendent (descendent) mode $\Delta$ satisfies the equality $Ord(\Delta)=m$ ($Ord(\Delta)=n$) is so interesting but out of the scope of the paper. It is clear that the equality is satisfied for trivial scales (and notably for monotonic and chromatic ones); in addition, for diatonic scales this condition is equivalent to the matter that the scale is heptatonic (especially major or minor).
\end{remark}
\begin{remark}
It is not ungraceful here to make some remarks on the Operational Principle of Harmony. This axiom, for its part, is sufficiently powerful and technical from logical point of view as much as it might be protested by musicians for the sake of having a purely mathematical nature. It is necessary to mention that we could posit another strictly weaker statement musically expressing the related algebraic property of harmony. An appropriate one in this context may roughly be presented as follows:
\begin{center}
`the sum of every number of adjacent intervals of a scale is harmonic'
\end{center}
(recall that a family of adjacent coordinates of a tuple $(r_i)_{i=1}^{n}$ is any set of the form $\{r_i, ..., r_{i+k(\mod n)}\}$ for $k\in \mathbb{N}$ and $1\leq i\leq n$), say ``Modal Principle of Harmony'', which seems more evident and arises out of the practical prospect of doing various melodies over scales within a music implementation by player. But one should pay close attention to its reasonable consequences and care to maintain the valuable musical purposes. In fact, not only is the Modal Principle of Harmony, contrary to its non-trivial appearance, the most elementary expectation of harmonic entities almost all musicians have, but it produces a widespread variety of musical modes and, in turn, an extended aria of musical scales on the base of which composers, in a remarkable position of authority, are able to create miscellaneous musical temperaments with more exceptional feedback through all types of music styles covering different genres from western to eastern ones. Of course, harmonic intervals of such new structural scales are not of the form $\varepsilon^n$ and this means the characterization affair becomes no longer practicable rationally. In other words, if it was the case, we would obtain only a proper subset of conclusions thus far discovered in our music theory. In conclusion, the axiomatic system will get involved in the musical gap not to have whole acquaintance with the science of \emph{counterpoint} as a knowledge investigating principles governing the harmonic world of musical chords in the light of polyphony, as we will fractionally have paid in the future and we had nurtured the dream of the portrayal of harmony in mind and we turned it into a musical intent, whatever the cost, for the sake of having the benefit of such a scientific blessing which determines and classifies all species of triads as the foundation of polyphonic music \cite{T2}. If only it could be possible to make at least one triad scale-based on each degree of all heptatonic eastern scales ...! But, combining and adapting the rules of eastern and western music concomitantly in a unique theoretical framework would be a futile effort. Having replaced the Operational Principle of Harmony by the modal version, we lose a part of facilities of counterpoint in practice (for instance in musical composition, arrangement and so forth) as well as our system misses its completeness in the absence of harmony characterization and hence it causes us to subconsciously enter the possible music universes in the sense that, theoretically speaking, we have to figure out the facts of the world by virtue of \emph{modal logic} \cite{C1, C2} (equipped with additional connectives such as necessity and possibility interrelated by Aristotle's classical attitude, i.e. $\square \varphi =\neg \lozenge \neg \varphi$) that changing the underlying logic of the axiom system from first-order into modal has not been of our intentions, and interpretatively speaking, we may construct abundant models within each of which harmony represents its own particular meaning not incidentally obeying general laws of counterpoint, having included a strange diversity of musical structures. After all, it is worth noting the Modal Principle of Harmony is consistent with all axioms of our music system, especially with the Axiom of Diatonicism, and indeed dependent upon the Operational Principle of Harmony, that we felt compelled not to apply in spite of the inherent desire. We extremely recommend readers interested in the study of eastern music (see \cite{F1}) discovering consequences deducible from using such an axiom and referring to Appendix \RNum{2} which supplies some oriental scales whose existence contradicts only the Operational Principle of Harmony, and instead, they are in agreement with the Modal Principle of Harmony and produced by a similar eastern copy of the Major Scale Axiom entitled ``Chaargah Scale Axiom''. From theoretical viewpoint, the heptatonic scale of Chaargah (not to be confused with the gypsy scale) is the most significant Iranian scale, including other sorts of harmonic intervals (not of the form $\varepsilon^n$ for $n\in \mathbb{Z}$), having a strange interesting analogy with the same as presented in Axiom 20, and containing some attributes of the minor and major scales simultaneously.
\end{remark}
\begin{remark}
Let $\varsigma \in \mathbb{S}$ be fixed and for all. Denote by $H_{\hat{\varsigma}}$ the largest (w.r.t. $\subseteq$) subset of $\mathbb{S}$ containing $\varsigma$ whose elements construct harmonic intervals, and name it as the \emph{harmonic set} balanced on $\varsigma$. The reason for such a notation is as clear as $H_{\hat{s}}=H_{\hat{t}}$ iff $s\simeq t$. One may easily observe and check that $f(H_{\hat{\varsigma}})=\{2^{n/12}f(\varsigma):n\in \mathbb{Z}\}$ which is obtained by a sort of combination of the two concepts of tonality and harmony for $\varsigma$. Once we construct the set $H_{\hat{\varsigma}}/\simeq$, a naive binary operation $\oplus$ on it is able to be considered as follows; given any $\hat{s},\hat{t}\in H_{\hat{\varsigma}}/\simeq$, we define $\hat{s}\oplus \hat{t}$ to be $\hat{u}$ where $f(u)=f(s)f(t)/f(\varsigma)$; in other words, a representative of $\hat{s}\oplus \hat{t}$ is with frequency $|[\varsigma ,s]|.|[\varsigma ,t]|f(\varsigma)$, clarifying why $\oplus$ is well-defined. Alternately and more accurately thinking, the harmonic set balanced on $\varsigma$ is exactly the module components set of the chromatic scale $(\hat{\varsigma},(\varepsilon)_{i=1}^{12})$ (the greatest scale w.r.t. $\sqsubseteq$) with right twelve members. Without loss of generality, we denote these members by $\mathcal{N}_i$ for $i=0, 1, ..., 11$ in their natural order of the ascendent module. Obviously, all degrees of any parallel scale (with tonic $\hat{\varsigma}$) must truthfully belong to the harmonic set $H_{\hat{\varsigma}}/\simeq =\cup_{i=0} ^{11}\mathcal{N}_i$. It is easily proved that $\mathcal{N}_i\oplus \mathcal{N}_j=\mathcal{N}_{i+j (\mod 12)}$ for any $0\leq i,j\leq 11$. Also, it is trivial to show that $H_{\hat{\varsigma}}/\simeq$ together with $\oplus$ constitutes a finite abelian group of order $Card(H_{\hat{\varsigma}}/\simeq)=12$ whose neutral element is $\mathcal{N}_0=\hat{\varsigma}$ and the inverse of an $\mathcal{N}_i$, denoted $\ominus \mathcal{N}_i$, is equal to $\mathcal{N}_{12-i}$ for each $i$. In addition, that is a cyclic group generated by $\mathcal{N}_1$. Therefore, the group $(H_{\hat{\varsigma}}/\simeq ,\oplus)$ is in fact isomorphic to the usual additive group $\mathbb{Z}_{12}$ ($\mathcal{N}_i\longleftrightarrow [i]$). We may also define $\mathcal{N}_i\ominus \mathcal{N}_j =\mathcal{N}_i\oplus (\ominus \mathcal{N}_j)$ as another binary operation on $H_{\hat{\varsigma}}/\simeq$. The \emph{accidentals} of western music are the unary operations $\sharp :H_{\hat{\varsigma}}/\simeq \to H_{\hat{\varsigma}}/\simeq$ and $\flat :H_{\hat{\varsigma}}/\simeq \to H_{\hat{\varsigma}}/\simeq$ respectively called the \emph{diese} and \emph{bemol} and defined by $\sharp \mathcal{N}=\mathcal{N}\oplus \mathcal{N}_1$ (raising the pitch of the note $\mathcal{N}$ by one half step) that is alternately denoted by $\mathcal{N}\sharp$, and $\flat \mathcal{N}=\mathcal{N}\ominus \mathcal{N}_1$ (lowering the pitch of $\mathcal{N}$ by a half step) that is alternately denoted by $\mathcal{N}\flat$.
\end{remark}
\paragraph*{}
Now, it is time to introduce the musical notes axiomatically approached, namely the most basic reason for speaking of the number $7$ concealed in the system. In order to further explore musical phenomena, the reader might simultaneously continue with the common music theory textbooks from now on. 
\begin{definition}
Given $\varsigma \in \mathbb{S}$. The group $(H_{\hat{\varsigma}}/\simeq ,\oplus)$ mentioned in Remark 5.90 is said to be the \emph{fundamental group of music} relative to the \emph{base sound} $\varsigma$. Every member of this group, i.e. $\mathcal{N}_i$ ($0\leq i\leq 11$), is called a \emph{note}. In particular, those notes belonging to the module components set of the parallel major scale are respectively named as \emph{do}, \emph{re}, \emph{mi}, \emph{fa}, \emph{sol}, \emph{la}, and \emph{si} (or \emph{ti}), and simply denoted by $C$, $D$, $E$, $F$, $G$, $A$, and $B$ in their natural order of the ascendent module. Any monotonic, or chromatic, or minor, or major scale with tonic $\mathcal{N}$ is called the \emph{$\mathcal{N}$ monotonic}, or \emph{$\mathcal{N}$ chromatic}, or \emph{$\mathcal{N}$ minor}, or \emph{$\mathcal{N}$ major} scale accordingly. The definitions of other common technical concepts conventional in the music theory language are optionally left to the reader.
\end{definition}
\begin{remark}
From now on, we fix such a sound $\varsigma$ in the system to keep using the same notations as employed in Remark 5.90 and Definition 5.91. It is worth noting that the reason for applying such notations, beginning with the third letter of the English alphabet, is essentially physical since it is customary to consider a sound with frequency $440\,Hz$ as the note la ($A$) not as the base sound for the sake of having the widest feasible tonal expanse of musical hearing from base frequencies to high ones and getting the best number of audible octaves most competent of all. However, there is no logical difference between these two methods.
\end{remark}
\begin{remark}
Why is the procedure we pursued for acquiring musical notes not basically in accordance with the ordinary classical approach in harmonic consonance theory originated from ancient Greeks by the Pythagoreans? Let us speak more in relation to the subject (the reader is initially recommended to seriously study related recourses like \cite{L1} p. 45-51). From the viewpoint of the theory of Fourier series (see \cite{B1}) in physics of sound waves, musical sounds (or tones) are consonant compound waves, having a \emph{fundamental} frequency $f_0$, and decomposable into (as a sum of) countably many sine waves with various frequencies in the form of natural multiples of the fundamental, namely $nf_0$ where $n\in \mathbb{N}$, say $n$-th \emph{harmonic}. So in the language of our axiom system, for a given $s\in \mathbb{S}$ with $f(s)=f_0$, the corresponding harmonic set, designated $H_{\hat{s}}$, must be considered as having the property $\{nf_0:n\in \mathbb{N}\}\subseteq f(H_{\hat{s}})$ in place of what we have worked on in Remark 5.90. In the event that desired harmony is supposed to be well-behaved with respect to tonality as we believe so (otherwise, our music goes down the drain), we should deal with a larger harmonic set containing $\{2^mnf_0: m\in \mathbb{Z}, n\in \mathbb{N}\}$. Whereas this set is dense in the set of positive real numbers especially in the open real interval $(1,2)$, so many harmonic intervals will be produced which are arbitrarily small, contradicting the Axiom of Diatonicism and subsequently devastating the current ground we have hitherto provided for music theory. Proceeding with a simple practicable example of such a traditional method, we suppose that $\hat{s}=C$. In practice, what is antecedently heard is the first harmonic, i.e. the fundamental frequency $f_0$, the second harmonic namely $2f_0$ making an octave interval and producing the same note as $C$ is, then the third harmonic with frequency $3f_0$ whose class is almost the note sol ($G$), and so forth. The point is that the intervals produced by consecutive harmonics tend towards the unit interval because the corresponding sequence of frequency ratio converges to one, and hence from a step onwards no harmonics will sound well-tuned, but since the intensity level of such harmonics converges to zero, consequently from a step onwards they are no more able to be heard practically, and in turn the false-intoned harmonics are not received at all even by well-equipped experienced sensitive ears. Indeed, this fact is the foundation of mix engineering to control different ranges of the sea of extra waste frequencies created by all components of contents composed in a piece of music, which plans to achieve a higher quality with a fair transparence. In summary, by means of the initial harmonics one may obtain some scales, e.g., the classic version of the major scale, i.e. the \emph{just intonation} scale, which is constructed as follows:
\begin{center}
$(\hat{s},(\frac{9}{8},\frac{10}{9},\frac{16}{15},\frac{9}{8},\frac{10}{9},\frac{9}{8},\frac{16}{15}))$.
\end{center}
But unfortunately, for that not all semitones, particularly inside an octave, are congruent in any possible way (the scale is idiomatically not \emph{equal-tempered}), which makes the underlying trend of doing music be entirely fucked up, we besides get into trouble with transposing arbitrary songs written in the key of $C$ missing transposability of melodies in turn.\\
Instead, We defined the notion of note on the base of the most natural musical structure adapted to our physical intuition of music and by virtue of the equal-tempered major scale, having replaced the traditional method of harmonic progression with the modern octave progression. We, meanwhile, emphasize that there are no rules restricting the use of different harmonics of a given sound and applying a certain collection of the initial harmonics with a fixed number of them. Aside from this, how can we be sure whether the frequencies obtained by harmonic progression relative to a given sound (like $(3/2)f_0$) are exactly corresponding to the same notes as our musical intuition may lead (the perfect fifth of $C$, namely $G$), and more important than it, how can we prove that whether this issue is not a tendency inherited from the history of music affected by the ancient Greeks thoughts (so far as we have been informed of course)? Nobody knows what the essence of a perfect fifth and the note sol are, let alone their measure and frequency, just like the unfamiliar reason why doubling the frequency leads to non-identical equitonal sounds with the same note name. It is apparent that physical experience cannot justify the relation between the nature of musical sounds and their frequency as decisively as mathematical logic does. It is equally true that the origin of our treatment enjoys both physics and history but that is not the only criterion to achieve the desired intentions. In this direction, we have entrusted ourselves to the rules of logic and wended our own way through this splendid beauteous land for meticulous discovery of musical facts in the language of mathematics. Do not ever misconstrue our opinion on logic because it is not everything. Similar to other branches of mathematics, logic is involved in some unsolved problems too which have remained still open, but there is nothing to be done except something better than anything; it is just a fact of life (see the Final Discourse). 
\end{remark}
\paragraph*{}
All students of music should be expected to prove the elementary properties of notes for themselves at least once in their life, some of which are as follows.\\
\begin{exercise}
$C\sharp =D\flat$, $(D\sharp)\sharp =E$, $E\sharp =F=(G\flat )\flat$, $A\oplus B=G\ominus B$, $B\oplus B=F\oplus F$ (equivalently, $B\ominus F=F\ominus B$), $(E\oplus E)\oplus E=F\sharp \oplus F\sharp =C\ominus C=C$, etc. 
\end{exercise}
\begin{remark}
This time we set-theoretically divide the harmonic set $H_{\hat{\varsigma}}$ balanced on $\varsigma$ by the all-purpose relation $\sim$, equip it with the appropriate version of the operation $\oplus$ endowing the fundamental group of music, and consider the restriction of the binary relation $*'$. Clearly, the resultant structure $(H_{\hat{\varsigma}}/\sim ,\oplus,*'|_{H_{\hat{\varsigma}}/\sim})$ is a countable abelian totally-ordered group whose identity element is $\tilde{\varsigma}$ and the inverse of an $\tilde{s}$ is an identity class whose representatives have a frequency of $(f(\varsigma))^2/f(s)$. Moreover, this group is cyclic and equal to $<\tilde{s}>$ where $s$ is any sound with frequency $2^{\pm 1/12}f(\varsigma)$, and hence it is isomorphic to the additive group of the integers $(\mathbb{Z},+)$. This is also able to be comprehended in the language of frequency:
\begin{center}
$f(\dfrac{H_{\hat{\varsigma}}}{\sim})=\{\{2^{\frac{n}{12}}f(\varsigma)\}:n\in \mathbb{Z}\}$.
\end{center}
\end{remark}
\begin{definition}
The ordered group $(H_{\hat{\varsigma}}/\sim ,\oplus,*')$ mentioned in Remark 5.95 is said to be the \emph{fundamental group of monophony} relative to the \emph{base sound} $\varsigma$. Every (finite) sequence in the set $H_{\hat{\varsigma}}/\sim$ is called a (\emph{finite}) \emph{melody}. The empty melody (i.e. the sequence $\emptyset$) is called \emph{silent} and every constant melody is called \emph{trivial}. The \emph{harmonic sequence} of a melody $\{\tilde{s_i}\}_{i=1}$ is defined by $\{[s_i,s_{i+1}]\}_{i=1}$ (whose terms are all harmonic). Two melodies $\{\tilde{s_i}\}_{i=1}$ and $\{\tilde{t_i}\}_{i=1}$ are called \emph{congruent} if their consecutive terms construct congruent intervals correspondingly; i.e., $(\tilde{s_i},\tilde{s_{i}}_{+1})\cong (\tilde{t_i},\tilde{t_{i}}_{+1})$ for each $i$.
\end{definition}
\begin{remark}[Notation]
We use that same old notation $\cong$ to express when melodies are congruent to each other. We denote by $\mathcal{M}$ the set of all melodies, and by $\mathcal{M}^f$ the set of all finite melodies.
\end{remark}
\begin{corollary}
$Card(\mathcal{M}^f)=\aleph_0$, and $Card(\mathcal{M})=2^{\aleph_0}$.
\end{corollary}
\begin{theorem}[Melody Transposition Theorem]
Given a melody $M=\{\tilde{s_i}\}_{i=1}$. There are countably many melodies congruent to $M$ (i.e. the family of congruent melodies is countable). In addition, for every harmonic interval $I$ there exists one and only one melody $M'=\{\tilde{t_i}\}_{i=1}$ congruent to $M$ such that the interval $(\tilde{s_1},\tilde{t_1})$ is congruent to $I$.
\end{theorem}
\begin{proof}[Sketch of Proof]
The first part is trivial. For the second part, by the Axiom of Motion take a sound $t_1$ satisfying $I\cong [s_1,t_1]$, and for each $i$ pick a sound $t_{i+1}$ with the property that $f(t_{i+1})=|[s_i,s_{i+1}]|.f(t_i)$.
\end{proof}
\begin{definition}
Given a note $\mathcal{N}$ and a melody $M=\{\tilde{s_i}\}_{i=1}$. $M$ is said to be \emph{ascendent} (\emph{descendent}) if all terms of its harmonic sequence are greater (less) than unit. Any subsequence of $M$ is said to be a \emph{submelody} of $M$. Any submelody of $M$ of the form $\{\tilde{s_k}\}_{k=i} ^j$ where $i<j$ (possibly $j=\infty$) is called \emph{continuous}. For a given scale $S\in \mathcal{S}_{\mathcal{N}}$, we say $M$ is \emph{based on} $S$ whenever all terms of every ascendent and descendent continuous submelody of $M$ are respectively contained in all ascendent and descendent degrees of $S$, meaning that firstly each ascendent (and descendent) degree $\hat{s}$ of $S$ includes a term $\hat{s_i}$ of some ascendent (and descendent) continuous submelody of $M$ (namely $\tilde{s_i}\subseteq\hat{s}$), and secondly each term of an arbitrary ascendent (and descendent) continuous submelody of $M$ is included in some ascendent (and descendent) degree of $S$; in this case, we also say that $S$ is the scale of the melody $M$.
\end{definition}
\begin{remark}[Notation]
The set of all (finite) melodies based on a given scale $S$ is denoted by $\mathcal{M}_S$ ($\mathcal{M}^f_S$).
\end{remark}
\begin{corollary}
For every $S\in \mathcal{S}$, $Card(\mathcal{M}^f_S)=\aleph_0$ and $Card(\mathcal{M}_S)=2^{\aleph_0}$.
\end{corollary}
\begin{corollary}
There is only one silent melody and it is (vacuously) both ascendent and descendent. Every trivial melody is neither ascendent nor descendent.
\end{corollary}
\begin{proposition}
Let a melody $M$ be based on a scale $S$ with tonic $\hat{s}$. Then $S$ is the least scale (w.r.t. $\sqsubseteq$ on $\mathcal{S}_{\hat{s}}$) whose ascendent and descendent degrees contain the terms of every ascendent and descendent continuous submelody of $M$ respectively.
\end{proposition} 
\paragraph*{}
There is a simpler description of a melody based on a compatible scale which is presented in the following, although its converse does not absolutely hold as expected.
\begin{proposition}
Given $S\in \mathcal{S}^c$. $M\in \mathcal{M}$ is based on $S$ if and only if all terms of the melody $M$ are contained in the degrees of the scale $S$.
\end{proposition}
\begin{proposition}[Existence of Melody Scale]
For every melody $M$ (not silent), there exists some scale on which $M$ is based.
\end{proposition}
\paragraph*{}
The following states that the act of transposing a melody is well-behaved with respect to the scale of the melody. 
\begin{proposition}
Let $M$ be a melody based on a scale $S$. For every melody $M'$ congruent to $M$ there is a relative scale (w.r.t. $S$) that $M'$ is based on.
\end{proposition}
\begin{remark}
Of course, a more general fact analogous to Proposition 5.107 would be under discussion;

`if $M$ and $M'$ are congruent melodies respectively based on scales $S$ and $S'$, then $S$ and $S'$ are circular rearrangements of each other with probably different tonics.'\\
Now, by a sort of conversely thinking of this reality and considering the Existence of Melody Scale, one should ask whether there is a practical way to find a scale a melody is based on or not. Notice that the scale of a melody is not necessarily unique as mentioned; by virtue of what we have from the recent fact, we understand that any circular rearrangement of the existing scale of a given melody (with the appropriate tonic obtained by the corresponding term of the melody) will be another scale of that melody. Although this result, namely the lack of uniqueness, does maybe not appear so agreeable for the reader at first sight, it is worth stating that, in practice, different musicians may correlate several distinct scales with the given melody too. In fact, that on which scale a melody factually lies is associated with the aesthetic of music. As an example, some music psychologists believe that the larger the musical interval, the pleasanter it would become. But having discarded such similar beliefs, we got ahead with another course for intervals and instead of setting up a binary relation among them, we utilized a kind of conception of harmony as a property of intervals. We adopted to put off such an aesthetic comparison to other phenomena such as melodies and chords as will follow. Let's move on. As another example in relation to the current discussion at monophonic music, most musicians believe that the tonic of a melody played in an unknown scale is uniquely determined by its melodic motion through the state of its descent; though the scale of the melody is not the only one, that scale on whose tonic the melody descends and finishes is certainly unique and the only place where the melody pleasantly glitters and its quality of movement looks coherent (many sources of this topic are published like \cite{N1}). On the other extreme, one may protest about sophisticated construction we defined for the cryptic concept of tonic. Why tonic at all? By a little consideration it is realized that we have not yet made a musically substantial use of this notion. It is imperative to mention that we tried to introduce only the mathematical structure of such an entity, just like the others, according to musical expectations, and we did never ever plan to analyse musical emotion exhibited in the form of the melodic demeanour or to explain its psychological reasons. Of course, we have also attained lots of wonderful conclusions in some rough drafts never referred to in the literature, but in order to retain the main framework of the paper we preferred to abstain from entering the profound subject of aesthetics of music and attempting at classifying species of melodies (we intentionally relinquished publishing results of working on this subject for some anonymous reasons). However, after introducing another concept of monophonic music, called \emph{modulation}, we conclude the discussion with a conjecture on music aesthetics in the direction of the classification of pleasant melodic movements at monophony.
\end{remark}
\paragraph*{}
In the following, the monophonic version of the notion of modulation is supplied whose polyphonic discernment will consequently be exploitable too on the base of the Fundamental Principle of Harmony, as well as chord tonicization totally which are left to the reader.
\begin{remark}
We define a binary operation $\otimes:\mathcal{M}^f\times \mathcal{M}^f\to  \mathcal{M}^f$ (on the set of all finite melodies) by $\{\tilde{s_i}\}_{i=1}^m\otimes \{\tilde{t_i}\}_{i=1}^n=\{\tilde{u_i}\}_{i=1}^{m+n}$, where
\begin{equation*}
\tilde{u_i}=\left\{
{\begin{array}{*{20}{c}}
{\tilde{s_i}}&{}&{,1\leq i\leq m}\\
{\tilde{t_i}_{-m}}&{}&{,m+1\leq i\leq m+n}.
\end{array}} \right.
\end{equation*}
It is clear that $\otimes$ is well-defined and $\otimes (M,M')$ is actually the unique shortest melody up to order (longer than $M$ and $M'$) that $M$ and $M'$ are two continuous submelodies of. One can easily check that $(\mathcal{M}^f,\otimes)$ constructs a non-commutative zerosumfree monoid whose left and right identities are equal to the silent melody. A mathematically interesting submonoid of this algebra is $(\mathcal{M}_S^f,\otimes|_{\mathcal{M}_S^f\times \mathcal{M}_S^f})$ for a given arbitrary scale $S$ with which we do not deal in this context, since musical interest we feel to address is just in the complements.
\end{remark}
\begin{definition}
Having set
\begin{center}
$\Delta (\mathcal{M}^f)=\{(M,M')\in \mathcal{M}^f\times \mathcal{M}^f: M \text{and}\, M' \text{are based on the same scale}\}$,
\end{center}
we call the restriction of the operation $\otimes$ on the set $\mathcal{M}^f\times \mathcal{M}^f-\Delta (\mathcal{M}^f)$ the \emph{modulation} and every member of its range a \emph{modulated} melody. In particular, for two distinct scales $S$ and $S'$ the function $\otimes |_{\mathcal{M}_S^f\times \mathcal{M}_{S'}^f}$ is called the \emph{modulation from} $S$ \emph{to} $S'$. For a given finite melody $M$ based on a scale $S$, the function $\otimes |_{\{M\}\times (\mathcal{M}^f-\mathcal{M}_{S}^f)}$ is called the \emph{modulation of} $M$.
\end{definition}
\paragraph*{}
In order to summarily present the formulation of our conjecture asserting the classification of melodies on the base of all equipments thus far provided (whose formal phrasing in the language of mathematics needs many more new technical concepts), by tradition, an additional primitive term whose rational properties are jammed into some axioms is essentially required. The difference between concepts (including all primitive ones) and axioms (regarded as theorems) is indeed that against theorems, a definition is exclusively an ascription of name to a conceptual entity which is declared by decree to increase our convenience but not our knowledge, whereas axioms and theorems are increasing both knowledge and inconvenience to discover new facts. Therefore, we consider an undefined strict ordering $\propto$ over the set of all melodies $\mathcal{M}$, called ``being pleasanter than'', and we say ``a melody $M_1$ is pleasanter than a melody $M_2$'', denoted $M_2\propto M_1$.
\paragraph*{}
\textbf{Conjecture.} \textit{Let $\mathcal{X}$ denote either the ordered set $(\mathcal{M}_S,\propto)$ or $(\mathcal{M}^f_S,\propto)$ for a given scale $S$. There exists a satisfiable axiom set describing simply acceptable properties of pleasantness of melodic motions such that for every given scale $S$ it determines a countable final segment $\mathcal{M}_S^*$ of $\mathcal{X}$, and makes the axiom system decide whether an arbitrary melody (based on $S$) belong to $\mathcal{M}_S^*$ or not. Moreover, there is a surjective function $\Xi_S:\mathcal{M}_S^*\to \mathbb{N}$ assigning $1$ to the silent melody and satisfying the basic condition that for every pair of melodies $M_1$ and $M_2$, $\Xi_S(M_1)=\Xi_S(M_2)$ iff $M_1$ and $M_2$ have the same continuous submelodies up to congruence, and $\Xi_S(M_1)<\Xi_S(M_2)$ iff $M_1\propto M_2$. Besides, focusing upon finite melodies, for any two distinct scales $S$ and $S'$ the functions $\Xi_S$ and $\Xi_{S'}$ can uniquely be extended to a surjection $\Xi:\otimes(\mathcal{M}_S^{f*}\times \mathcal{M}_{S'}^{f*})\to \mathbb{N}$ with the same basic properties and that $\Xi(M_1\otimes M_2)=\Xi_S(M_1)\Xi_{S'}(M_2)$.}
\begin{remark}
We truly confess full discussion around the Conjecture is beyond the scope of the paper, but some notes have the distinction of being under consideration. The first moral is adapted to the approval for that the select pleasant melodies are made constructively whose elements (and notes) are successively collected not in an accidental manner but rather explicitly clarified by the new axioms claimed in the Conjecture. We believe that the most complicated melodies from the aesthetic point of view are infinite whose more delicate classification would be possible in terms of ordinal numbers, and the pleasanter the melody, the more its length, but not conversely; they are able to be decomposed (maybe periodically) into finite continuous submelodies, tonicized within an octave and having the least pleasantness among the pleasantest melodies;  these less pleasant melodies of short length are directly supplied in or derived from the axioms. The point is that uncountably many of melodies based on a given scale $S$ have pathological motion and behave too chaotically to appear pleasant. In fact, the proper subset $\mathcal{M}_S^*$ of melodies based on $S$, which is certainly countable, contains the pleasantest melodies meaning that every melody of $\mathcal{M}_S^*$ is pleasanter than all others. So by a kind of sensitivity, one may say that $\mathcal{M}_S^*$ consists of all the \emph{pleasant} melodies and the rest are roughly \emph{unpleasant}. About the relation $\propto$, we first notice any pair of arbitrary melodies need not be comparable under pleasantness; we think just as comparing songs in different styles is unreasonable, so is comparison among melodies based on two distinct scales; this seems nonsense because various scales are essentially incomparable alike to various styles. Hence we avoided such a collation as far as possible up to where, if necessary, we hold some musical tools such as monophonic modulation. What is apparent inside the Conjecture is an intuitive way to construct a linear extension of the order $\propto$ on a set of appropriate classes of congruent melodies of $\mathcal{M}^*_S$; once it is done, the function $\Xi_S$ specifying the amount of pleasantness of elements of $\mathcal{M}^*_S$ guarantees that such an order is factually a well-ordering; perhaps this nice guy stands to benefit from some other good properties but its first value correlated with the empty melody (which is the endpoint of the final segment $\mathcal{M}^*_S$), namely $\Xi_S(\emptyset)=1$, gives an official translation of the fact that `the silence is better than the piffle'. It is worth noting the case of the Conjecture in relation to finite melodies is more applicable in practice; the extension $\Xi$ of $\Xi_S$ first establishes a correspondence between the range of the modulation restricted to the product of the finite pleasant melodies based on $S$ and $S'$ (i.e. $ran(\otimes |_{\mathcal{M}_S^{f*}\times \mathcal{M}_{S'}^{f*}})$) and the monoid of natural numbers under multiplication (i.e. $(\mathbb{N},\times)$); indeed, this extension characterizes the pleasant modulations from $S$ to $S'$ and additionally its formula declares intelligibly that the impact of pleasant modulation is extremely more than that of pleasant tonicization (made by the motion of pleasant melodic movements) in the direction of embellishment and beautification of melody.
\end{remark}
\paragraph*{}
Lest you simply judged the conjecture, we dared not opine on completeness, or independence, or even consistence of the resultant axiom system. It may involve many mathematicians' hair greying during a long time, just like the destination of parallelism in geometry; at least we think so. If one knew the satisfiable strange phenomena we have experienced on the Conjecture, he or she would get goose pimples undoubtedly. Just anyhow, as an example of a pleasant melody, we have prepared a simple pattern in eastern music (Appendix \RNum{3}) whose pulchritude is almost ascertained for us (as well as a song composed by L\o vland called ``Nocturne'' based on the pathetic minor scale).
\paragraph*{}
Returning to the main discourse, in order to finalize the monophonic part of basic music theory we present the following remark with the aim of merely introducing some musical notions and conventions significant and relevant. In turn, the philosophy of such naming as (first of all) octave interval, fifth interval, perfect dyad, minor and major triad, tonic, sensible, and so on will somewhat be clarified. Although we have no need to get onto the subject of consonance and dissonance in the field of music analysis (consonant intervals are special types of harmonic intervals usually described as agreeable, while dissonant intervals are those that cause tension to be resolved to consonant ones), it is graceful to proceed with a suitable musical structure, i.e. \emph{chord}, to exactly match the base of the science of polyphony. To deeper study we refer the interested reader to some related sources like \cite{T2}).
\begin{remark}
By tradition, the number of notes beginning with $C$ to $B$ included between endsounds distinguishes the music theoretic types of harmonic intervals from each other. In so treating, every unit interval is called the \emph{first} or \emph{prime} interval, every half tone and every whole tone are called the \emph{second} interval, any intervals of measure $2^{1/4}$ or $2^{1/3}$ and of measure $2^{5/12}$ (equivalent to five semitones) are respectively said to be the \emph{third} and \emph{fourth} interval, $2^{7/12}$ is also expected to be the \emph{fifth} interval by definition (the musical distance from $C$ to $G$ within an octave), the \emph{sixth} interval has measure $2^{2/3}$ or $2^{3/4}$, the measure of the \emph{seventh} equals $2^{5/6}$ or $2^{11/12}$, and finally the \emph{eight} is another name for the octave interval. We may naively generalize such a method of inductively naming harmonic intervals on the base of the Fundamental Group of Music, and by using the equivalence classes of tonality we can restrict the species of harmonic intervals to an octave interval from the first to the seventh, and again construct an objective group isomorphic to the same as $(\mathbb{Z},+)$. So the $n$-th intervals would be equivalent to some from first to seventh that is congruent to $n$ modulo $7$ ($n\in \mathbb{Z}$). Every harmonic interval whose measure oversteps the boundaries of an octave (namely $1\leq |I|<2$) just like the octave interval is known as a \emph{compound} interval, otherwise it is considered as a \emph{simple} interval. The \emph{inversion} of $[a,b]\in \mathbb{I}^*$ is customarily defined by the harmonic interval $[b,a^*]$ (or sometimes $(O^-(\tilde{b}),\tilde{a})$) which is actually the inverse of the interval $[a,b]$ as an element of the aforementioned group, and displays a kind of melodic conversion of the given interval with a similar musical property; that the class of all intervals which are pairwise inversions of each other must have the same harmonic behavior is evident, since the inversion has right twice the size of the conversion. Thus the inversion of an $n$-th interval is just a $(9-n)$-th interval ($1\leq n\leq 8$). It is worth noting that after octave intervals which produce unisons, the most agreeable interval is of the fifth type as theoretically as we did seek. For more precise differentiation of all species of harmonic intervals, theoreticians divide them into two general parts based on their quality; 1- \emph{threefold-quality} intervals such as \emph{perfect} ones including first, fourth, fifth, and octave; 2- \emph{fourfold-quality} intervals such as \emph{major} and \emph{minor} ones including second, third, sixth, and seventh (as expected, the size of minor intervals is just one semitone less than that of major ones). Two additional quality of harmonic intervals are named \emph{diminished} (having one semitone less than the main interval including perfect, major, or minor) and \emph{augmented} (having one semitone more than the main interval). In diatonic scales such as major and minor ones each (ascendent) degree plays a special musical role on the base of which they are relegated to some suitable names describing their position in the scale; the first degree is the \emph{tonic} as the basis of tonality of the scale construction, the second one is called the \emph{supertonic} that naturally comes after the tonic, the third is said to be the \emph{mediant} meaning that it is in the middle of the perfect fifth interval produced by the first and fifth (as the most important) degrees of the scale, the fourth is the \emph{subdominant} because it is under the next (i.e. the fifth) degree called the \emph{dominant} (making a perfect fifth by the tonic) which gives the most noticeable note after the tonic and shows itself more strongly than the others within the octave interval created by the tonic and its unison, the sixth one is known as both the \emph{superdominant} and the \emph{submediant} for that its position is firstly over the dominant and secondly symmetrically similar to the location of the supertonic laid below the mediant within the second tetrachord of the $C$ major scale, the seventh is of course the \emph{sensible} or \emph{leading} note having tendency to the next note as was previously mentioned, but in the minor scale and more general cases it is named the \emph{subtonic} for the sake of being before the unison of the tonic making the basic octave interval in the scale. Therefore, it becomes clear that the characteristic of tonality of a scale is exhibited by its first and fifth (whether ascendent or descendent) degrees in order of precedence and for this reason the three notes obtained, i.e. the tonic, the subdominant (the fifth descendent degree) and the dominant are well-known as the \emph{tonal} notes; furthermore, the nature of (the parallel major and minor) scale is determined by those degrees changed in the act of converting from one to another the most important of which are the third and sixth (and occasionally seventh) ascendent degrees, i.e. the submediant and the mediant (and the sensible) that are well-known as the \emph{modal} notes.
\end{remark}
\begin{remark}
About western scales, it is worth noting there are different types of minor scales common in music theory other than what is introduced in Definition 5.84 and is dubbed the \emph{natural} minor scale, but the existence of a whole tone between its seventh and eighth ascendent degrees causes the subtonic to lose its leading temperament to the unison and because of this musicologists artificially reduced the size of the mentioned interval to a half tone by applying the operator $\sharp$ to the seventh ascendent (equivalent to the second descendent) degree of the natural minor scale. The resultant heptatonic compatible scale which has the ascendent sensible is called the \emph{harmonic} minor; (again) but this scale used not to be practically handled to produce melodies so much throughout the history of music for the sake of its fairly large interval of one and a half steps between the sixth and seventh ascendent degrees that seems not to be intimate and agreeable; so by virtue of using the naturalizer $\natural$ in the descendent module of the harmonic minor scale if we turned the second descendent degree to the first situation (i.e. the case of the natural minor scale) which is equivalent to the seventh ascendent degree (that was counterfeitly elevated by a semitone to amplify its sensibility and, for now, there is no longer need for doing that) to get lowered by a half tone, then there would be found an additional artificial scale which is in turn not compatible --- the so-called \emph{melodic} minor scale. It is obvious the recent two kinds of minor scales are not maximally even. Also, a special kind of the natural minor scale is frequently used at realistic music, notably in the East, which is (of the form $(\hat{s},(\varepsilon,\varepsilon^2,\varepsilon^2,\varepsilon^2,\varepsilon,\varepsilon^2,\varepsilon^2)$) not a matter of modal changing, but is a circular rearrangement of the natural minor and major scales having the descendent sensible by means of which a despondent and doleful sensation is endued with the scale structure, despite the major scale having a jocund and tenacious construction through its ascendent sensible. The fresh scale so obtained is consequently still maximally even and we are interested in naming it the \emph{pathetic} minor scale because of inducing the cadences of melodies (remarkably descending ones) based on itself to manifest an occult grief.
\end{remark}
\paragraph*{}
Having passed the world of monophonic music, we tend to generalize the concept of triad motivated to present the ``Fundamental Theorem of Harmony'' which is in fact the mediator between monophony and polyphony as the basis of the art of counterpoint.
\begin{definition}
Every (ordered) tuple $(s_i)_{i=1}^n$ of sounds ($n\in \mathbb{N}$) in which for every $1=1, ..., n-1$, $[s_i,s_{i+1}]$ is a harmonic interval greater than unit is called a \emph{chord}, every interval $[s_i,s_j]$ ($1\leq i,j\leq n$) is called an \emph{interior} interval of the chord, and the sound $s_1$ is called the \emph{root} sound.
\end{definition}
\begin{remark}[Notation]
We denote by $\mathcal{C}$ the set of all chords.
\end{remark}
\begin{corollary}
$Card(\mathcal{C})=2^{\aleph_0}$.
\end{corollary}
\paragraph*{}
The main word of the following proposition is that all interior intervals of a chord are harmonic as expected.
\begin{proposition}
Let $(s_i)_{i=1}^n\in \mathbb{S}^n$ for $n\in \mathbb{N}$. The following are equivalent:
\begin{enumerate}
\item For all $1\leq i<n$, $s_i*s_{i+1}$ and $[s_i,s_{i+1}]\in \mathbb{I}^*$.
\item $(s_i)_{i=1}^n$ is a chord.
\item $[s_i,s_{i+1}]\in \mathbb{I}^{>1*}$ for every $1\leq i<n$.
\item $[s_i,s_j]\in \mathbb{I}^{>1*}$ for every $1\leq i<j\leq n$
\item For every $1\leq i<n$, $[s_i,s_j]\in \mathbb{I}^{<1*}$ for all $1\leq j<i$ and $[s_i,s_j]\in \mathbb{I}^{>1*}$ for all $i<j\leq n$; in particular, $[s_1,s_j]\in \mathbb{I}^{>1*}$ for any $1<j\leq n$.
\item $[s_i,s_j]\in \mathbb{I}^*$ for any $1\leq i,j\leq n$, and $s_i*s_j$ for any $1\leq i<j\leq n$.
\end{enumerate}
\end{proposition}
\begin{definition}
Every pair of chords are said to be \emph{enharmonic} whenever they have the same interior intervals; more formally, there exists a one-to-one correspondence between their (sets of) coordinates such that corresponding intervals are equal.
\end{definition}
\begin{remark}[Notation]
When a chord $X$ is enharmonic to a chord $Y$, we write $X\approx Y$. From musical point of view, the binary relation $\approx$ is expected to set up a harmonic equivalence over all chords; this is in agreement with the following.
\end{remark}
\begin{proposition}
Two chords $(s_i)_{i=1}^m$ and $(t_i)_{i=1}^n$ are enharmonic to each other iff $m=n$ and $[s_i,s_{i+1}]=[t_i,t_{i+1}]$ for every $1\leq i<m$ (the root sounds are identical in pitch when $m=n=1$) iff $m=n$ and $s_i\sim t_i$ for all $1\leq i<m$.
\end{proposition}
\begin{corollary}
$\approx$ is an equivalence relation on $\mathcal{C}$.
\end{corollary}
\paragraph*{}
The aim of presenting such a notion of harmonic equivalence on $\mathcal{C}$ as Definition 5.118 was right to preserve the substance of harmony institutionalized within the interior harmonic intervals of a constructed chord. We think the identity relation $\sim$ among sounds keeps the harmonic sensation induced by all enharmonic chords invariant and in practice one may find an appropriate element of an arbitrary class under $\approx$ resolving the initial chord so that it sounds more impressing, but any other variation in the construction of the initial chord, like increasing the number of its coordinates even by means of sounds equitonal to the previous ones, certainly changes the nature of chord harmony so that in no way it can look more prepossessing. Of course, it is necessary to note that the ideational use of this musical term we made in Definition 5.118 essentially differs from the usual sense. The difference between this approach and what is customary in the music theory is that enharmonic objects (such as notes, scales, and chords) only refer to diverse names of one subjective entity over there, whereas we give only one objective name to differently related harmonic subjects in this context (like distinct chords).
\paragraph*{}
For the sake of matching up the knowledge of counterpoint with monophony, we need to carry our chords in a standard framework including melodies and make their coordinates included in the harmonic set $H_{\hat{\varsigma}}$ balanced on $\varsigma$. 
\begin{remark}[Notation]
We take $\mathcal{C}^*:=\mathcal{C}\cap (\cup_{n\in \mathbb{W}}H_{\hat{\varsigma}}^n)$. So $Card(\mathcal{C}^*)=\aleph_0$.
\end{remark}
\begin{definition}
Every element of $\mathcal{C}^*$ is called a \emph{standard} chord (in base $\varsigma$).
\end{definition}
\begin{remark}
Let us construct the quotient set $\mathcal{C}^*/\approx$. Recall that a melody $\{\tilde{s_i}\}_{i=1}$ is ascendent iff $s_i*s_{i+1}$ for every $i$. We trivially observe that each class of standard chords under $\approx$ (with a representative $(s_i)_{i=1}^n$) corresponds to an ascendent melody ($\{\tilde{s_i}\}_{i=1}^n$) and vice versa. This reality is indeed the germ of counterpoint connecting polyphony to monophony. Just like scales, we may define a binary operation $\vee$ on $\mathcal{C}^*/\approx$ correlating with every pair of standard chords a new one whose coordinates consist of those of the two given chords (formal definition is left to the reader). Having attached the empty chord $\emptyset$ to $\mathcal{C}$, we may also define another operation $\wedge :\mathcal{C}^*/\approx\times \mathcal{C}^*/\approx \to \mathcal{C}^*/\approx$ which naturally takes the intersection of a given pair of chords and gives a new one whose coordinates belong to both of them. One can check the accuracy of commutative, associative, and absorption laws with respect to the binary operations $\vee$ and $\wedge$ and immediately conclude that the algebraic structure $(\mathcal{C}^*/\approx,\vee,\wedge)$ establishes a countable lattice which is admissible to be called the \emph{Fundamental Lattice of Polyphony}.
\end{remark}
\paragraph*{}
We now deal with the main intention for polyphony by presentation of the Fundamental Theorem of Harmony. The finite version of this theorem (item 1) is more practicable in music aesthetic versus the infinite case (item 2) whose idea comes from an abstract extension of finite melodies to infinite ones in virtue of some nice operators (similar to the $\otimes$ and of course with some special properties) to make compatibility with the finite version, albeit both are intuitively obvious for musicians and mathematicians. Try to visualize it by a musical insight.
\begin{theorem}[Fundamental Theorem of Harmony]$\,\,\,\,\,\,\,\,\,\,\,\,\,\,\,\,\,\,\,\,$
\begin{enumerate}
\item There exists a one-to-one correspondence between the set of  finite sequences of the enharmonic classes of standard chords $\cup_{n\in \mathbb{N}}(\mathcal{C}^*/\approx)^n$ and the set of finite sequences of finite melodies $\cup_{n\in \mathbb{N}}(\mathcal{M}^f)^n$.
\item There exists a one-to-one correspondence between the set of infinite sequences of the enharmonic classes of standard chords $(\mathcal{C}^*/\approx)^{\mathbb{N}}$ and the set of all sequences of melodies $\cup_{n\in \mathbb{N}}\mathcal{M}^n\cup \mathcal{M}^{\mathbb{N}}$.
\end{enumerate}
\end{theorem}
\paragraph*{}
The moral is that when a composer is juxtaposing a series of particular chords to make up or improve a piece of music, it looks as though he/she is arranging several parallel melodies with related melodic movements over each other in a specific manner for his/her song.
\paragraph*{}
There are some polyphonic expressions of the Conjecture on the base of the Fundamental Theorem of Harmony in which to study the harmonic behaviour of different melodies on a scale in relation to each other simultaneously, but its discussion gets too onerous to be involved in terms of music so far constructed. On the other hand, its detailed topic could be an absurd ideal in the vain hope such that one might not even be confident enough of provability of the monophonic version of that conjecture. It would inevitably need more advanced equipments ...
\begin{remark}
For every standard chord $(s_i)_{i=1}^n$, the note $\hat{s_1}$ is known as the \emph{root note}. In music theory, the most important chords are three-note ones which were previously dubbed triads. The triads $(x,y,z)$ give a natural basis to construct advanced chords and divide up into two general species. 1- Those contain agreeable chords (from the theory of harmony viewpoint) whose fifth interval is perfect, or equivalently speaking, the dyad $(x,z)$ is perfect; this species is known as the \emph{perfect} chords and divides into two groups; the first case are those in which the interval $[x,y]$ is a major third and in turn $[y,z]$ is a minor third, that are called \emph{major} (\emph{perfect}) chords, and the second group is the converse of the first case, called \emph{minor} (\emph{perfect}) chords. 2- These are accounted disagreeable, having imperfect fifth interval, and divided into two groups;
the first one contains those with diminished fifth and in turn consists of two (congruent) minor third intervals ($[x,y]$ and $[y,z]$) which are called \emph{diminished} chords, and the second case includes those with two major thirds and in turn making an augmented fifth which are called \emph{augmented} chords. It is worth noting the triad $(x,y,z)$ in its own initial case is said to be in the \emph{root position} and for the purpose of musical composition, it is customary to define some types of its inversion; the \emph{first inversion} of $(x,y,z)$ is defined by the triad $(y',x,z)$ (or $(y',z',x)$) where $y'^*=y$ (and $z'^*=z$) in which the first third gets the inverse and a sixth (or fourth) will appear by the initial root note; the \emph{second inversion} of $(x,y,z)$ is defined by the triad $(z',x,y)$ where $z'^*=z$ in which the fifth changes to its inversion (this kind is more applicable). The point that might be questioned about is whether replacing each component of a chord by another equitonal sound deforms the harmonic sensation of that chord or not. This is a matter of harmony that the resultant chord under such a replacement \emph{almost} keeps in harmony with the root position in practice, which naively gives a more extensive class of enharmonic chords not preferentially involved in. Further exploration of this aria of music theory including the analysis of the interior intervals of chords in their root position or first and second inversions and their relationships, involving in the four-note chords and other kinds of inversions and their classification, as well as discussing the agreeability of chords in relation to their components, and more chords and so on are all entrusted to the interested reader which may be recommended to read \cite{P1}.
\end{remark}
\paragraph*{}
Concluding the section, one may infer the two concepts of tonality and harmony (in terms, monophony and polyphony) in the world of music are incorporated into each other so as not undoubtedly be more separable than we handled. We may denote the so far constructed axiom system by $\mathcal{M}(PIT)$, and since the new primitive concepts of this section are already theoretically characterized by its axioms in the language of $\mathcal{M}(PI)$, we do not deal with a fresh theory except a mathematically extended one. Nevertheless, this glance is not much acceptable from the perspective of music.
\begin{remark}
One may naively extend the Cartesian model of $\mathcal{M}(PI)$ (recall Remark 4.8) to $\mathcal{M}(PIT)$ in the following way; $x\simeq y$ iff $x_1=2^ny_1$ for some $n\in \mathbb{Z}$, and the interval $[x,y]$ is harmonic iff $x_1=2^{n/12}y_1$ for some $n\in \mathbb{Z}$. The veracity of the axioms of Tone Music are clear because it is independent of the interpretation (characterizations of tonality and harmony are theoretically done).
\end{remark}

\section{Rhythm}
\paragraph*{}
To start this section, we should first mention that the sounds of $\mathcal{M}(PIT)$ are just abstract and not physical as the ear perceives in the material environment, and most important of all they have no time value. So it is not irrelevant to dub the axiom system $\mathcal{M}(PIT)$ \emph{free time} or \emph{abstract}. The point is that such kinds of sounds are clearly not suitable to be directly applied to introduce the notion of rhythm as a \emph{sound-valued} function in the real world of music, since in order to stabilize the framework on the ground of time we will innately require a peculiar concept of continuity on the nature of sounds. Hence, we generalize our sound in a natural way that it makes sense assigning a positive real number to its duration as the fourth characteristic of sound waves. In this direction, we make use of topological equipments whose technical study via topology textbooks (for instance \cite{M6}) is left to the reader.
\begin{remark}
Recall $C_p^*(\mathbb{R})=\Gamma (\mathbb{S})$ where $\Gamma :\mathbb{S}\to C_p(\mathbb{R})$ is the injection mentioned in $\mathcal{M}(PI)$ (i.e. the timbre), and consider it as a metric subspace of $C(\mathbb{R})$ (the set of continuous functions on $\mathbb{R}$) endowed with the usual supremum metric. Note $C_p^*(\mathbb{R})\subseteq C_p(\mathbb{R})\subseteq C(\mathbb{R})\subseteq \mathbb{R}^{\mathbb{R}}$. There exists exactly one topology $\tau$ on $\mathbb{S}$ such that each of the following equivalent conditions holds:
\begin{enumerate}
\item $\tau =\{U\subseteq \mathbb{S}:\Gamma(U) \,\text{is open in}\, C_p^*(\mathbb{R})\}$.
\item $\tau =\{\Gamma^{-1}(U): U \,\text{is open in}\, C_p^*(\mathbb{R})\}$.
\item $\tau$ is the coarsest topology on $\mathbb{S}$ relative to which the $\Gamma$ is continuous.
\item $\tau$ is the only one topology on $\mathbb{S}$ relative to which the $\Gamma$ is a homeomorphism.
\item $\tau$ is the unique topology generated by the metric $d:\mathbb{S}\times \mathbb{S}\to \mathbb{R}^{\geq 0}$ given by
\begin{center}
$d(s,t)=\sup \{|\Gamma(s)(x)-\Gamma(t)(x)|: x\in \mathbb{R}\}$.
\end{center}
\item $\tau$ is generated by the only metric $d$ making $\Gamma$ be an isometric isomorphism (between metric spaces).
\end{enumerate}
Therefore, the topological space $(\mathbb{S},\tau)$ described above on which we work from now on is metrizable by virtue of the metric $d$ (item 5) relative to which the function $\Gamma :\mathbb{S}\to C_p(\mathbb{R})$ is a topological embedding that makes the metric space $(\mathbb{S},d)$ isometrically isomorphic to $C_p^*(\mathbb{R})$ equipped with its supremum metric. One may observe that whenever two sounds are close to each other under the topology $\tau$ (or the metric $d$), they will remain to be near in the frequency and intensity sense; equivalently, for every sound $s$
\begin{center}
$\forall \epsilon >0(\exists \delta >0(\forall x\in \mathbb{S}(d(s,x)<\delta \Rightarrow (|f(s)-f(x)|<\epsilon \wedge |\iota (s)-\iota (x)|<\epsilon))))$,
\end{center}
which means the frequency of sounds $f:\mathbb{S}\to \mathbb{R}^+$ and the intensity of sounds $\iota:\mathbb{S}\to \mathbb{R}^+$ are continuous too with respect to the topology $\tau$ on $\mathbb{S}$ --- a desired result. However, they do not act homeomorphically.
\end{remark}
\begin{remark}
Here investigating other interesting spaces existent in $\mathcal{M}(PIT)$ would be delightful. We know that the sets $\mathbb{S}/\sim$ and $\mathbb{S}/\simeq$ partition the $\mathbb{S}$. Consider the surjective maps $\mathbb{S}\to \mathbb{S}/\sim$ and $\mathbb{S}\to \mathbb{S}/\simeq$ respectively defined by $s\mapsto \tilde{s}$ and $s\mapsto \hat{s}$. It is deduced from general topology that there exist unique topologies on $\mathbb{S}/\sim$ and $\mathbb{S}/\simeq$ relative to which the aforementioned maps are quotient maps (or strongly continuous, meaning that $U\subseteq \mathbb{S}/\sim$ ($\mathbb{S}/\simeq$) is open iff its inverse image is open in $(\mathbb{S},\tau)$), called the \emph{quotient} topologies induced by those maps. Thus one may obtain the quotient spaces of $\mathbb{S}$ under the two equivalence relations $\sim$ and $\simeq$, and by using the axiom of choice prove that they are respectively homeomorphic to the real line and the circle $S^1$ (this is the natural reason why some authors display musical notes on circles) with their standard topologies (how?). Notice that the order topology on $\mathbb{S}/\sim$ derived from its total ordering $*'$ is the same as the quotient topology. Also, one can show that the Fundamental Groups of Music and Monophony endowed with the subspace topology will be topological groups, since the binary operation $\oplus$ and the unary operation $^{-1}$ (the inverse) are continuous on either space. Since the singletons are open in both topological groups, their topologies will coincide with the discrete topology and in turn these two topological groups will actually become equivalent to the famous ones; namely $(\mathbb{Z},+,\mathcal{P}(\mathbb{Z}))$ (that is infinite) and $(\mathbb{Z}_{12},\oplus,\mathcal{P}(\mathbb{Z}_{12}))$ (that is finite). Defining an intrinsic topology on the Fundamental Lattice of Polyphony to construct an appropriate topological lattice is left to the reader. It is worth noting The Cartesian interpretation of these five topological structures is more enjoyable (see Remark 4.8 and Remark 5.127).
\end{remark}
\paragraph*{}
We will need to use half-open real intervals (of the form $[a,b)$ that is left-closed and right-open or vice versa for $a,b\in \mathbb{R}$ and $a<b$) and let $l(I)$ denote the length of the interval $I$ (so $l([a,b))=b-a>0$). We also apply the subspace topology on such intervals as subsets of the real line.
\begin{definition}
By a \emph{natural sound} is meant a bounded continuous function $\gamma :[a,b)\to \mathbb{S}$, and the \emph{duration} of $\gamma$, denoted $\kappa (\gamma)$, is defined by the length of its domain, i.e. $b-a$. $\gamma$ is considered a \emph{musical sound} whenever the (continuous) map $fo\gamma$ is constant.
\end{definition}
\paragraph*{}
Something has been missed in Definition 6.3! The problem gets started where we have to get nothing by something, as once we got something from nothing; there must be silence to formulate the sense of rhythm. Of course, one could consider an external object $\xi$ which is not of the sound nature and mean any constant map from a real interval $I$ with constant value $\xi$ by a silence of duration $l(I)$; but this would be as extremely artificial as losing the underlying monolithic space of work. We incline to integrate the silence into the space of sounds for obtaining a united mathematical universe to whose whole texture the concept of continuity we have in mind motivated by the metric space $(\mathbb{S},d)$ is applicable. Thus, as a topologically experienced reader shall guess, in order to make a sufficiently nice topological space $(\mathbb{S}^*,\tau^*)$ containing sounds and silences which satisfies our physical, empirical, and natural expectations, we require the Stone-\v{C}ech compactification regarded as an extension of the topological space $(\mathbb{S},\tau)$ that desirably generalizes the notion of continuity from it; formally speaking, every bounded continuous function $\gamma: I\to \mathbb{S}$ can be uniquely extended to a bounded continuous function $\gamma: I\to \mathbb{S}^*$, and consequently the $\gamma$ still remains to be a natural sound in the fresh sense of continuity. In so doing, continuation in the compactified sense and that induced by the metric $d$ will coincide in case of working on the space of sounds. Note since $(\mathbb{S},\tau)$ is a Tychonoff space, the topological requirement is provided.
\begin{remark}
Since the topological space $(\mathbb{S},\tau)$ is locally compact and metrizable, and thus it is completely regular as well, we may consider its Stone-\v{C}ech compactifications (generating the desirable property of continuity). From topology we understand that $(\mathbb{S},\tau)$ has a unique compactification $\mathbb{S}^*$ up to homeomorphism iff it is compact or $\mathbb{S}^*-\mathbb{S}$ is a singleton. From the point that in practice we need to attach just one external object $\xi$ as the silence to $\mathbb{S}$, it follows that its minimal Stone-\v{C}ech compactification containing $\xi$ is equivalent to the one-point (or Alexandroff) compactification of $\mathbb{S}$. Hence, we work with the Alexandroff compactification $\mathbb{S}^*=\mathbb{S}\cup \{\xi\}$ that is the unique compactification of $\mathbb{S}$ relative to which every natural sound into the space of sounds is uniquely extended to some into $\mathbb{S}^*$. By definition, the topology $\tau^*$ on $\mathbb{S}^*$ is finer than $\tau$ (i.e., $\tau \subseteq \tau^*$) and contains any $\mathbb{S}^*-C$ where $C$ is a compact subset of $\mathbb{S}$. $(\mathbb{S}^*,\tau^*)$ is a compact Hausdorff topological space in which $\mathbb{S}$ is dense. Furthermore, by \cite{M1}, this space is metrizable (more strongly than being Hausdorff); one of such topologically equivalent metrics (generating the topology $\tau^*$) would be constructed by fixing an arbitrary sound $s$ as follows:
\begin{equation*}
d^*(x,y)=\left\{
{\begin{array}{*{20}{c}}
{\min \{d(x,y),h(x)+h(y)\}}&{}&{,x\neq s \wedge y\neq s}\\
{\min \{h(x),h(y)\}}&{}&{,x=s\veebar y=s}\\
{0}&{}&{,x=y=s}
\end{array}} \right.
\end{equation*}
where the bounded function $h:\mathbb{S}\to [0,1]$ is defined by $h(x)=1/(1+d(s,x))$.\\
The inclusion mapping $i:\mathbb{S}\to \mathbb{S}^*$ is a uniformly continuous topological embedding. Eventually, $(\mathbb{S}^*,\tau^*)$ is a compact metrizable space (with the metric $d^*$) which will consequently become complete and totally bounded (and so bounded). Therefore, $\mathbb{S}^*$ is the completion of $\mathbb{S}$ as well; what a nice topological space!\\
We can and do now redefine a natural sound by any path $\gamma$ in the space $\mathbb{S}^*$ on condition that $\xi \in ran(\gamma)$ implies $\gamma \equiv \xi$ (constant) in which case $\gamma$ is called a \emph{silence} (as desired) and its \emph{duration} is $\kappa(\gamma)=l(dom(\gamma))$. As a matter of convention, we presume $\Gamma(\xi)\equiv 0$ and consequently $f(\xi)=\iota(\xi)=0$, so the naive extensions of frequency, intensity, and spectrum of sounds over $\mathbb{S}^*$ is done in such a manner that they are all uniformly continuous. Any natural sound $\gamma$ with the property that $fo\gamma$ is constant would be a musical sound. We denote by $\mathbb{S}'$ the set of all natural sounds and by $\mathbb{S}'_c$ the set of all musical sounds. Note for every $\gamma \in \mathbb{S}'$ all maps $fo\gamma$, $\iota o\gamma$, and $\Gamma o \gamma$ are uniformly continuous. Topologically interested readers may define appropriate topology on $\mathbb{S}'$ for their mathematical intentions having no reason to be discussed here.
\end{remark}
\begin{corollary}
$Card(\mathbb{S}^*)=Card(\mathbb{S}')=Card(\mathbb{S}'_c)= 2^{\aleph_0}$.
\end{corollary}
\begin{corollary}
Every silence is a musical sound.
\end{corollary}
\paragraph*{}
In contrast to the above insipid consequence, by replacing abstract sounds of our music system with musical ones (the elements of $\mathbb{S}'_c$) we will still not lose the musical structure concerning the concepts of tonality and harmony (how and why?).
\begin{remark}
The one-point compactification of $\mathbb{S}$ has a nice interpretation in the Cartesian model of $\mathcal{M}(PI)$ (or $\mathcal{M}(PIT)$). Based on Poincar\'{e} conjecture proved by Perelman, stating that every three-dimensional simply connected compact manifold is diffeomorphic to the 3-sphere, it is easy to observe that the Cartesian construction of the topological space $(\mathbb{S}^*,\tau^*)$ will be the three-dimensional sphere $S^3$ whose Riemannian metric generates the same topology as $d^*$ does --- the $\tau^*$. Thus you may even take integral over this musical manifold, it's just enough to say the word!
\end{remark}
\paragraph*{}
Let $I=[a,b)$ be an interval on the real line. We know that there exists a one-to-one correspondence between the set of all finite partitions of $I$ into subintervals of the form $\{[x_{i-1},x_i)\}_{i=1}^n$ and the set of all strictly increasing finite sequences in $I$ of the form $\{x_i\}_{i=0}^n$ where $x_0=a$ and $x_n=b$. Following the section, we shall simply use such partitions of real intervals to introduce the notion of \emph{rhythm}.
\begin{definition}
Let $\mathcal{I}$ be a partition of a given interval $I$.
\begin{enumerate}
\item A function $t:I\to \mathcal{I}$ is said to be a \emph{time} if every choice function on $\mathcal{I}$ is a right inverse function of $t$; equivalently, $t$ is a left inverse function of every choice function for $\mathcal{I}$. Also, we occasionally say that the ordered pair $(t,\mathcal{I})$ is a time.
\item A function $m:\mathcal{I}\to \mathbb{S}'$ is said to be a \emph{metre} if the domain function on $\mathbb{S}'$ is a left inverse function of $m$; equivalently, $m$ is a right inverse function of the domain function. Also, we occasionally say that the ordered pair $(m,\mathcal{I})$ is a metre.
\item Any function of the form $mot:I\to \mathbb{S}'$ is called a \emph{rhythm} where $t$ (or $(t,\mathcal{I})$) is a time and $m$ (or $(m,\mathcal{I})$) is a metre. The \emph{duration} of a rhythm is defined by the length of its domain.
\end{enumerate}
\end{definition}
\begin{remark}[Notation]
If applicable, $\mathcal{R}$ denotes the set of all rhythms.
\end{remark}
\begin{lemma}
Let $\mathcal{I}$ be a partition of a given interval $I$.
\begin{enumerate}
\item A function $t:I\to \mathcal{I}$ is a time iff for every choice function $\Omega :\mathcal{I}\to I$ we have $to\Omega =Id_{\mathcal{I}}$ iff $x\in t(x)$ for all $x\in I$; in particular, $t$ is onto.
\item A function $m:\mathcal{I}\to \mathbb{S}'$ is a metre iff $(dom)om=Id_{\mathcal{I}}$ (where $dom$ is considered as a function from $\mathbb{S}'$ into $\mathcal{P}(\mathbb{R})$) iff $dom(m(\alpha))=\alpha$ for every $\alpha\in \mathcal{I}$; in particular, $\kappa om=l$ in the sense that the duration of the image of each subinterval equals the length of that subinterval.
\end{enumerate}
\end{lemma}
\paragraph*{}
The motivation of such formulations of the concepts of time and metre as displayed in Definition 6.8 was the particular conclusions of Lemma 6.10 to be immediate; as the inclusive behaviour of the bracket map on the real line correlating the integer $i$ with each interval $[i,i+1)$ ($\mathbb{Z}$ as a partition of $\mathbb{R}$) implants the meaning of time in the mind, and as the preservation of the duration of the natural sounds the subintervals are mapped to is done by metre.
\begin{proposition}
A function $\rho :I\to \mathbb{S}'$ is a rhythm iff one of the following equivalents holds:
\begin{enumerate}
\item There is only one partition $\mathcal{I}$ of $I$ for which there are a time $(t,\mathcal{I})$ and a metre $(m,\mathcal{I})$ such that the following diagram commutes:
\[\begin{tikzcd}
I\ar{r}{t}\ar{dr}{\rho} & \mathcal{I}\ar{d}{m}\\
& \mathbb{S}'
\end{tikzcd}\]
i.e., $\rho =mot$.
\item There is a piecewise-and-right continuous function $\rho ':I\to \mathbb{S}^*$ (i.e. having finitely many discontinuity points at which the function is right-continuous) whose restriction to each piece is equal to the value of $\rho$ at any points of that piece.
\end{enumerate}
In particular, $\rho$ is piecewise constant and right-continuous, and $Card(\mathcal{R})=2^{\aleph_0}$.
\end{proposition}
\begin{corollary}
For every partition $\mathcal{I}$ of $I$,
\begin{enumerate}
\item there is exactly one time $t:I\to \mathcal{I}$.
\item there are uncountably many metres $m:\mathcal{I}\to \mathbb{S}'$.
\end{enumerate}
\end{corollary}
\begin{corollary}
For every rhythm $\rho:I\to \mathbb{S}'$ the existing time $t$ and metre $m$ in item 1 and the function $\rho'$ in item 2 of Proposition 6.11 are unique, and so that for every $\alpha \in \mathcal{I}$ and every $x\in \alpha$ we have $\rho '|_{\alpha}=m(\alpha)=\rho |_{\alpha}(x)=\rho(x)$ (which is equated with a natural sound).
\end{corollary}
\paragraph*{}
So by virtue of the one-to-one correspondence between $\mathcal{R}$ and the set of all piecewise-and-right continuous functions from any interval to the set of natural sounds, which is arisen from item 2 of Proposition 6.11, we discern an alternative method to define rhythms.
\paragraph*{}
In the world of practice, every rhythm has right two basic characters one of which indicates the number of beats included in the rhythm and the other is its tempo stating how fast the rhythm is displaying. Thus, since some types of rhythms are well-proportioned and more elegant than the others, it is relevant to impute meaningful numbers for describing the features of their rhythmical treatment. In this direction, we do proceed with our own singular method to introduce the characteristics of rhythms and not apply the common trend in music theory employing fraction as a symbol for notating time signatures because of some logical disadvantage, although we essentially object to such a traditional idea. So for simplicity, we use the alternate form of partitions of $I=[a,b)$ as $\mathcal{I}=\{a=x_0< ... <x_n=b\}$, and contemporarily say that $\mathcal{I}$ is \emph{regular} whenever all elements of $\mathcal{I}$ have the same length, equivalently, $\Delta x_i=x_i-x_{i-1}$ is constant for all $i=1, ..., n$. Every partition of $I$ containing $\mathcal{I}$ (in the new form) is a \emph{refinement} of $\mathcal{I}$.
\begin{definition}
Let $\rho :I=[a,b)\to \mathbb{S}'$ be a rhythm whose partition is of the form $\mathcal{I}=\{a=x_0< ... <x_n=b\}$. $\rho$ is said to be \emph{regular} if there is a regular refinement of $\mathcal{I}\cap \rho^{-1}(\{\xi\}^c)$. Otherwise, $\rho$ is said to be \emph{irregular}.
\end{definition}
\paragraph*{}
Having focused on regular rhythms, such a definition of them presented above expresses that no matter when, i.e. at which point of the interval $I=[a,b)$, the silence appears but its duration together with that of the previous objects including natural sounds and silences, if existent, must reach a rational multiple of the duration of the rhythm by a difference of length $a$ according to the following theorem; in fact, it is necessary and sufficient that the set $\{\min \alpha :\alpha \in \mathcal{I},\, \rho |_{\alpha}\neq \xi \}$, where $\mathcal{I}$ is of its natural form, is included in a regular partition (of the new form). The idea of the proof of this theorem is based on elementary analysis which we decided not to include.
\begin{theorem}
Let $\rho :I=[a,b)\to \mathbb{S}'$ be a rhythm whose partition is of the form $\mathcal{I}=\{a=x_0< ... <x_n=b\}$ and assume
\begin{center}
$\mathcal{I} \cap \rho ^{-1}(\{\xi \}^c)\cup \{a\}=\{a=y_0<y_1< ... <y_m\leq b\}$ 
\end{center}
($m\leq n$). $\rho$ is regular iff there exists the coarsest regular partition of $I$ containing $\mathcal{I} \cap \rho ^{-1}(\{\xi \}^c)$ iff each of the following equivalents occurs:
\begin{enumerate}
\item $\{\frac{x-y}{l(I)}:x,y\in \mathcal{I}\wedge \rho (x)\neq \xi \wedge (\rho (y)\neq \xi \vee y=a)\}\subseteq \mathbb{Q}$.
\item $\{\frac{x_i-a}{b-a}: \rho (x_i)\neq \xi \}\subseteq \mathbb{Q}$.
\item $\{\frac{\Delta y_i}{b-a}: 1\leq i\leq m \}\subseteq \mathbb{Q}$.
\end{enumerate}
\end{theorem}
\begin{remark}[Notation]
Having used the notations of Theorem 6.15, we understand that once $\rho$ is regular, the coarsest regular refinement of $\mathcal{I} \cap \rho ^{-1}(\{\xi \}^c)$ is existent and unique, which we denote by $\mathcal{I}^*$.
\end{remark}
\begin{definition}
By the same notations as in Theorem 6.15 and Remark 6.16, suppose  $\rho$ is regular. Every element belonging to $\mathcal{I}^*$ is called a \emph{beat}. $\rho$ is called \emph{trivial} if $Card(\mathcal{I}^*)=1$, namely $\rho$ has only one beat, otherwise it is called \emph{non-trivial}. $\rho$ is called \emph{duple} whenever it has two beats, $\rho$ is called \emph{triple} whenever it has three beats, and so forth. The \emph{tempo} of $\rho$, denoted $T(\rho)$, is defined by the quantity $Card(\mathcal{I}^*)/l(I)$. The \emph{signature} of $\rho$ is defined by the ordered pair $(Card(\mathcal{I}^*),l(I))$.
\end{definition}
\begin{remark}
Due to the last part of Definition 6.17, having uniquely ascribed such a signature to the nature of a regular rhythm as a reasonable characteristic qualifying its intrinsic construction, one can easily earn practically enough information about the cadent structure of that rhythm including 1- the number of beats, 2- the duration of the rhythm, and 3- the tempo. As customary in music theory, a unary operation $\Im :\mathbb{W}\times \mathbb{R}^+\to \mathbb{W}\times \mathbb{R}^+$ is defined by $\Im (n,t)=(3n,t)$ together with a binary operation $\star :(\mathbb{W}\times \mathbb{R}^+)\times (\mathbb{W}\times \mathbb{R}^+)\to \mathbb{W}\times \mathbb{R}^+$ given by the formula $(m,r)\star (n,t)=(m+n,r+t)$, both acting on the set of all rhythm signatures under the first of which the duration of the the given rhythm would be invariant and the restriction of the second one to the set of signatures of those rhythms having a fixed tempo would still preserve the same tempo. Apart from the fact that the set of all rhythm signatures together with $\star$ constructs an abelian semigroup, these operations have interesting usage for classifying longer regular rhythms by virtue of a given one in practice, as follows.
\end{remark}
\begin{definition}
Every duple or triple (regular) rhythm is said to be \emph{simple}. Every (non-trivial) rhythm whose signature belongs to the range of the $\Im$ and it is neither simple nor a silence is said to be \emph{compound}. Every non-trivial rhythm whose signature belongs to the range of the $\star$ and it is neither a silence nor a simple or compound rhythm is said to be \emph{complex}.
\end{definition}
\paragraph*{}
As a straightforward result of number theory, we have
\begin{corollary}
Every regular rhythm is either silence, trivial, simple, compound, or complex.
\end{corollary}
\begin{remark}
The method we supplied to assign a signature to a given regular rhythm is at a remarkable advantage. To show this, we have to introduce the corresponding version of such a notion prevalent among musicians, under the pressure of the tradition of music theory. We first fix a positive real number $t$ as the time value of the \emph{whole note} on the base of which other species of note values are able to be defined such as the \emph{minim} as long as half a whole note ($t/2$), the \emph{crotchet} as long as half a minim ($t/4$), the \emph{quaver} as long as half a minim ($t/8$), and so forth. For every regular $n$-beat rhythm (i.e. $Card(\mathcal{I}^*)=n\in \mathbb{N}\cup \{0\}$) the duration of each of whose beats is equated with $t/2^k$ for some $k\in \mathbb{N}\cup \{0\}$, the time signature $\frac{n}{2^k}$ is considered just as a notation (not to be confused with the rational fractions). One may see that the basic quantities describing the rhythmical constitution of rhythm like duration and tempo are tacitly suggested in this signature as well. Although there is no limitation on such an assignment to rhythms and any arbitrary rhythm signature can be covered by changing $t>0$, the job is not at all done uniquely, as two musicians may consider different time signatures especially with distinct denominators for a given rhythm (similar to the scale of melody disserted in the previous section). Meanwhile, our trend in doing so specifies only one characteristic in a unique way. Of course, it is worth noting that this traditional approach would clearly build up a fluency in the international orthography of music and provide a facility for performing songs. Depending on the amount of the tempo of the rhythm the appropriate value of $t>0$ may be determined and the optimum note (getting the most iteration in appearance from the value point of view) as the unit of beat of the rhythm intended for composer may then obtained. In turn, the solmization during performance reaches to the most comfortable state of its own. Likewise, the corresponding operations $\Im$ and $\star$ in this context as the generators of compound and complex rhythms are respectively interpreted as follows:
\begin{center}
$\dfrac{3}{2}\cdot \dfrac{m}{2^k}:=\dfrac{3m}{2^{k+1}}$,\,\,\,\,\,\,\,$\dfrac{m}{2^k}+\dfrac{n}{2^k}:=\dfrac{m+n}{2^k}$.
\end{center}
One may define other technical concepts in the framework of rhythms including weak and strong beats, offbeat, syncope, and so on which we leave to the purposeful reader.
\end{remark}
\paragraph*{}
Mathematically thinking of irregular rhythms, the nature does never permit us to realize unrhythmically any sort of sonic phenomena, even the sound of rain drops; this reality arises from the analytic point that every irrational number is estimated by a rational one ($\mathbb{Q}\subseteq \mathbb{R}$ is dense), the job that the mind does unconsciously. Therefore, if a song sounds too discordant to be danced to, that is not necessarily for the sake of irregularity, but rather for that there are countably many regular times in the nature most of which are more complicated. Good old days, what a voluminous disputation it was! Once we made a private challenge in a small community of musicians and mathematicians discussing whether it is possible to distinguish the rhythm of a piece of music while somebody maybe assumes it not to be essentially rhythmic. How could it be proved that the signature of the complex rhythm of ``Nostalgia'', due to Yanni, is equivalent to $\frac{5}{8}$?
\paragraph*{}
It is necessary to mention that we did not deeply work on the subject of rhythm whose full discussion, really and truly, forms a separate independent book. Having concluded the section, we believe that there is also a rhythmic version of the Conjecture asserting the classification of pleasant rhythms, whose formulation in the constructed language makes a reconditely controversial topic to the extent of being out of the mood to be sought.

\section{Appendix}
\subsection*{{\large\textbf{\RNum{1}}}}
\paragraph*{}
Here, we wish to talk about categoricity of our axiom system provided that the theory of sets is consistent as promised in the Introduction. But, before that some significant points are recommended to be stated. With regard to Axiom 13, it is not specified whether the function $\Gamma :\mathbb{S}\to C_p(\mathbb{R})$ acts surjectively, equivalent to whether or not the range of $\Gamma$ is a proper subset of $C_p(\mathbb{R})$. What is clear in the Cartesian model supplied in Remark 4.8 is $C_p^*(\mathbb{R})\subsetneq C_p(\mathbb{R})$, because there can certainly be found many nonconstant periodic continuous (and even differentiable) functions on $\mathbb{R}$ that are not of the trigonometric form. But, this is not theoretically guaranteed in the axiom system $\mathcal{M}(PI)$. In fact, one may show this statement is independent of the $PI$-axioms (neither provable nor refutable), and completeness of the system will be lost consequently. To fill the vacuity, we make use of a delicate trick to append such a normal property to the theory $\mathcal{M}(PI)$; we replace Axiom 13 by a logically stronger one, making $\Gamma$ onto, as follows.\\

\textbf{Axiom 13'.}\, Axiom 13 + $C_p(\mathbb{R})\subseteq C_p^*(\mathbb{R})$.
\paragraph*{}
We denote the resultant axiom system by $\mathcal{M}^*(PI)$, and the union of that and the Tone Music by $\mathcal{M}^*(PIT)$. Since these new theories are actually some extension (in fact, completion) of the old ones, it follows that all theorems of $\mathcal{M}(PI)$ and $\mathcal{M}(PIT)$ are still satisfied in $\mathcal{M}^*(PI)$ and $\mathcal{M}^*(PIT)$ respectively, but the Cartesian model of $\mathcal{M}(PI)$ (and $\mathcal{M}(PIT)$) does not work for $\mathcal{M}^*(PI)$ (and a fortiori $\mathcal{M}^*(PIT)$) any more. However, by Axiom 13' we obtain $C_p^*(\mathbb{R})=C_p(\mathbb{R})$, from which by the Axiom of Extensionality $\Gamma :\mathbb{S}\to C_p(\mathbb{R})$ establishes a one-to-one correspondence and supplies a physical representation of the sounds of our music. Indeed, by virtue of such a bijection, a natural model for $\mathcal{M}^*(PI)$ (and in turn $\mathcal{M}^*(PIT)$) is automatically produced in which every sound is interpreted as a nonconstant periodic continuous function on the real line and the rest of undefined concepts in the same way as they are theoretically characterized in $\mathcal{M}(PI)$ and $\mathcal{M}(PIT)$ (and of course in $\mathcal{M}^*(PI)$ and $\mathcal{M}^*(PIT)$) will be interpretable. One can easily check that all axioms of $\mathcal{M}^*(PIT)$ are true in this model. We call the model constructed in such a naive manner the \emph{real} model of our axiomatic system.
\paragraph*{}
The following shows that the real model for $\mathcal{M}^*(PI)$ is up to isomorphism the unique one.\\

\textbf{Metatheorem 1.}\, \textit{The theory $\mathcal{M}^*(PI)$ is categorical.}
\begin{proof}[Sketch of Proof]
$\mathcal{M}^*(PI)$ has at least one model, i.e. the real model, hence it is consistent. What remains to prove is that $\mathcal{M}^*(PI)$ has at most one model, which tacitly assures its completeness as well. So having considered any two models $\mathcal{M}_1$ and $\mathcal{M}_2$ for the axiom system $\mathcal{M}^*(PI)$, we will prove $\mathcal{M}_1$ and $\mathcal{M}_2$ are isomorphic. We let $\mathbb{S}_{\mathcal{M}_i}$, $*_i$, $\cong _i$, and $\Gamma _i$ be respectively the corresponding interpretation of the technical concepts of the sounds set, the relation of being lower-pitched on sounds, congruence of intervals, and the timber in the model $\mathcal{M}_i$ ($i=1,2$). Notice that both models $\mathcal{M}_1$ and $\mathcal{M}_2$ satisfy all the theorems of $\mathcal{M}^*(PI)$, so according to Remark 6.1, $\Gamma _i$ is an isometry between $\mathbb{S}_{\mathcal{M}_i}$ and $C_p(\mathbb{R})$ ($i=1,2$). Thus having set $\Gamma =\Gamma _2^{-1}o\Gamma _1$, the function $\Gamma :\mathbb{S}_{\mathcal{M}_1}\to \mathbb{S}_{\mathcal{M}_2}$ establishes a one-to-one correspondence (actually an isometry). One may easily show the $\Gamma$ establishes also the isomorphism of models $\mathcal{M}_1$ and $\mathcal{M}_2$; because the following properties are satisfied:
\begin{enumerate}
\item $s\in \mathbb{S}_{\mathcal{M}_1} \Leftrightarrow \Gamma (s)\in \mathbb{S}_{\mathcal{M}_2}$,
\item $x*_1y\Leftrightarrow \Gamma (x)*_2\Gamma (y)$,
\item $[a,b]\cong _1 [c,d]\Leftrightarrow [\Gamma (a),\Gamma (b)]\cong _2 [\Gamma (c),\Gamma (d)]$,
\end{enumerate}
based on Axiom 13', item 3 of Theorem 3.83, and items 3 and 4 of Theorem 3.86.
\end{proof}
\paragraph*{}
In conclusion, we get\\

\textbf{Metatheorem 2.}\, \textit{The theory $\mathcal{M}^*(PIT)$ is categorical.}
\begin{proof}[Sketch of Proof]
This is immediate based on Theorem 5.27 and Theorem 5.81 stating that the primitive concepts of tonality and harmony in $\mathcal{M}^*(PIT)$ are theoretically characterized in the language of the theory $\mathcal{M}^*(PI)$.
\end{proof}

\subsection*{{\large\textbf{\RNum{2}}}}
\paragraph*{}
Here, in order to illuminate the grand scheme of the introductory structure of the oriental music, the most commonly applied eastern musical scales in the language of our music system are included (using their Persian names) where $\varepsilon =2^{1/12}$ and $s\in \mathbb{S}$. The strange point is all these scales have right seven notes just like the minor and major scales. Why are the most applicable scales in the musical world heptatonic (Axiom 20)? Nobody knows!\\ 
\begin{align*}
\text{Shur};\,\,\,\,\,\,\,\,\,\,\,\,\,\,\,\,\,\,\,\,\,\,\,\,& (\hat{s},(\varepsilon^\frac{3}{2},\varepsilon^\frac{3}{2},\varepsilon^2,\varepsilon^2,\varepsilon,\varepsilon^2,\varepsilon^2)) \\
\text{Abuata};\,\,\,\,\,\,\,\,\,\,\,\,\,\,\,\,\,\,\,\,\,\,\,\,& (\hat{s},(\varepsilon^\frac{3}{2},\varepsilon^2,\varepsilon^2,\varepsilon,\varepsilon^2,\varepsilon^2,\varepsilon^\frac{3}{2})) \\
\text{Bayate Zand};\,\,\,\,\,\,\,\,\,\,\,\,\,\,\,\,\,\,\,\,\,\,\,\,& (\hat{s},(\varepsilon^2,\varepsilon^2,\varepsilon,\varepsilon^2,\varepsilon^2,\varepsilon^\frac{3}{2},\varepsilon^\frac{3}{2})) \\
\text{Nava};\,\,\,\,\,\,\,\,\,\,\,\,\,\,\,\,\,\,\,\,\,\,\,\,& (\hat{s},(\varepsilon^2,\varepsilon,\varepsilon^2,\varepsilon^2,\varepsilon^\frac{3}{2},\varepsilon^\frac{3}{2},\varepsilon^2)) \\
\text{Dashti};\,\,\,\,\,\,\,\,\,\,\,\,\,\,\,\,\,\,\,\,\,\,\,\,& (\hat{s},(\varepsilon,\varepsilon^2,\varepsilon^2,\varepsilon^\frac{3}{2},\varepsilon^\frac{3}{2},\varepsilon^2,\varepsilon^2)) \\
\text{Afshari};\,\,\,\,\,\,\,\,\,\,\,\,\,\,\,\,\,\,\,\,\,\,\,\,& (\hat{s},(\varepsilon^2,\varepsilon^\frac{3}{2},\varepsilon^\frac{3}{2},\varepsilon^2,\varepsilon^2,\varepsilon,\varepsilon^2)) \\
\text{Saba}\, (\text{Arab Shur});\,\,\,\,\,\,\,\,\,\,\,\,\,\,\,\,\,\,\,\,\,\,\,\,& (\hat{s},(\varepsilon^\frac{3}{2},\varepsilon^\frac{3}{2},\varepsilon,\varepsilon^3,\varepsilon,\varepsilon^2,\varepsilon^2)), \,\text{or}\\
&(\hat{s},(\varepsilon^\frac{3}{2},\varepsilon^\frac{3}{2},\varepsilon^\frac{3}{2},\varepsilon^\frac{5}{2},\varepsilon,\varepsilon^2,\varepsilon^2)) \\
\text{Mahur}\, (=\text{Major}); \,\,\,\,\,\,\,\,\,\,\,\,\,\,\,\,\,\,\,\,\,\,\,\,& (\hat{s},(\varepsilon^2,\varepsilon^2,\varepsilon,\varepsilon^2,\varepsilon^2,\varepsilon^2,\varepsilon)) \\
\text{Rast}\, (\text{or Rast-Panjgah}); \,\,\,\,\,\,\,\,\,\,\,\,\,\,\,\,\,\,\,\,\,\,\,\,& (\hat{s},(\varepsilon^2,\varepsilon^2,\varepsilon,\varepsilon^2,\varepsilon^2,\varepsilon, \varepsilon^2)) \\
\text{Homayun};\,\,\,\,\,\,\,\,\,\,\,\,\,\,\,\,\,\,\,\,\,\,\,\,& (\hat{s},(\varepsilon^\frac{3}{2},\varepsilon^\frac{5}{2},\varepsilon,\varepsilon^2,\varepsilon,\varepsilon^2,\varepsilon^2)), \,\text{and}\\
&(\hat{s},(\varepsilon^2,\varepsilon^\frac{3}{2},\varepsilon^\frac{5}{2},\varepsilon,\varepsilon^2,\varepsilon,\varepsilon^2)) \\
(\text{Bayate})\,\text{Esfahan};\,\,\,\,\,\,\,\,\,\,\,\,\,\,\,\,\,\,\,\,\,\,\,\,& (\hat{s},(\varepsilon^2,\varepsilon,\varepsilon^2,\varepsilon^2,\varepsilon^\frac{3}{2},\varepsilon^\frac{5}{2},\varepsilon)) \\
\text{Shushtari};\,\,\,\,\,\,\,\,\,\,\,\,\,\,\,\,\,\,\,\,\,\,\,\,& (\hat{s},(\varepsilon,\varepsilon^3,\varepsilon,\varepsilon^2,\varepsilon,\varepsilon^2,\varepsilon^2)) \\
\text{Segah};\,\,\,\,\,\,\,\,\,\,\,\,\,\,\,\,\,\,\,\,\,\,\,\,& (\hat{s},(\varepsilon^2,\varepsilon^\frac{3}{2},\varepsilon^\frac{3}{2},\varepsilon^2,\varepsilon^\frac{3}{2},\varepsilon^\frac{3}{2},\varepsilon^2)) \\
\text{Chaargah};\,\,\,\,\,\,\,\,\,\,\,\,\,\,\,\,\,\,\,\,\,\,\,\,& (\hat{s},(\varepsilon^\frac{3}{2},\varepsilon^\frac{5}{2},\varepsilon,\varepsilon^2,\varepsilon^\frac{3}{2},\varepsilon^\frac{5}{2},\varepsilon)) \\
\end{align*}

\subsection*{{\large\textbf{\RNum{3}}}}
\paragraph*{}
Here, in the following, an attempt is made to provide a brief score of a pleasant melody based on the scale of Dashti. You may have got an instrumental song including such a melody (accompanying its MIDI file) via the following address:
\begin{center}
\url{https://www.dropbox.com/h?preview=Seyyed+Mehdi+Nemati+-+Hava.zip}
\end{center}
or to get the direct download link click 
\href{http://s4.picofile.com/file/8372294418/Seyyed_Mehdi_Nemati_Hava.zip.html}{here}.
\\
\begin{figure}[h]
\begin{center}
\includegraphics[width=12cm]{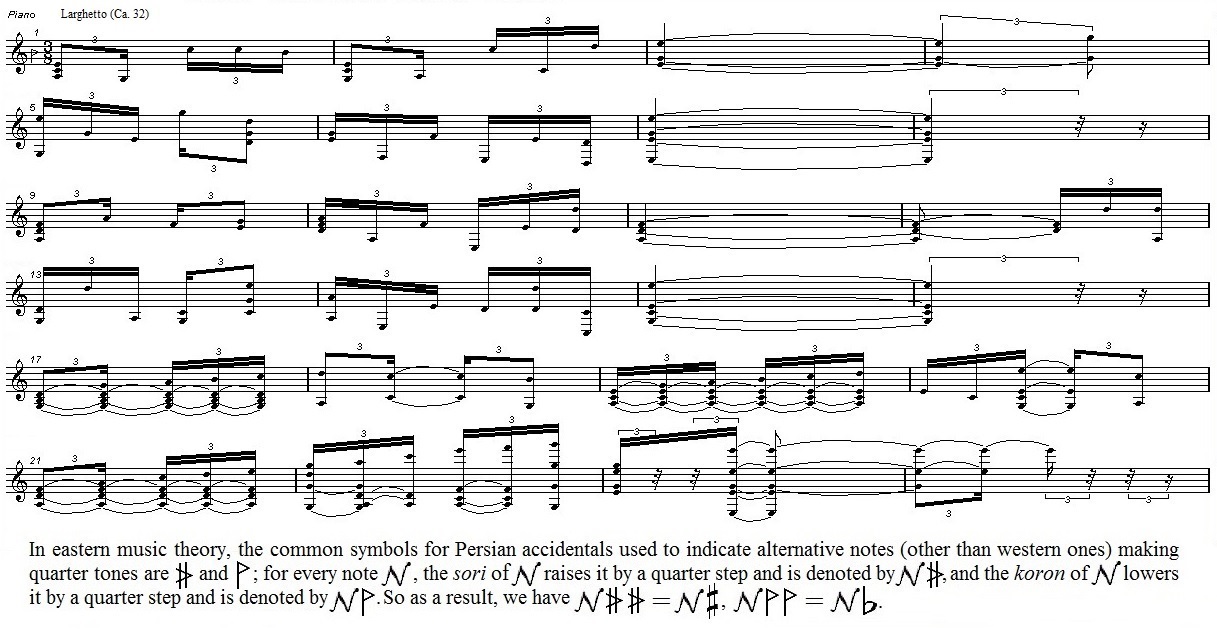}
\end{center}
\end{figure}

\subsection*{Final Discourse}
\paragraph*{}
The following would be apparent for the one who compares the structural resemblance among mathematical theories we have found axiomatized:
\paragraph*{}
`\emph{Having done not a lot, the creator of the universe has made \emph{almost} one structure on the base of which it is human to create many universes different in  structure, having done a lot. The goal may just be for joy to \underline{himself} while making sure no question ain't never ever got an absolute answer.}'

\section*{Acknowledgement}
\paragraph*{}
Mission accomplished. I would appreciate it \textbf{only} if God helped anyone (to) realize what I have meant! This is a lonely unsupported work in the privacy of my own room whose hardly-procured compilation has led to frequent mental tiredness during several years, dedicated to no one.

$\bullet$ \,\,\,\,\,\,\,\textit{E-mail}: \texttt{seyyedmehdinemati@yahoo/gmail.com}\footnote{\textsl{It was a dream inspired by Comet Hale-Bopp that was observed since childhood and needs more than two thousand years to become visible next time, in which case it is not even obvious whether there is anything remained to be done about the theory, much less anyone to respond (how awful)! Anyway, further dreams will be created and come true. Send my love, if possible.}}
\end{document}